%% file: ms.tex
\documentclass{amsart}
\usepackage{amssymb,amsfonts,amsmath,graphicx,mathtools}
\usepackage[alphabetic,y2k,lite]{amsrefs}
\usepackage{fullpage, mystyle}
\usepackage{MnSymbol}
\usepackage{tikz}
\usetikzlibrary{shapes.geometric}
\usetikzlibrary{calc}
\usetikzlibrary{scopes}
\usetikzlibrary{math}
\usetikzlibrary{decorations.markings,decorations.pathreplacing}
\usepackage{tikz-cd}

\tikzset{
every picture/.style={line width=0.8pt, >=stealth,
                       baseline=-3pt,label distance=-3pt},
dotnode/.style={fill=black,circle,minimum size=2.5pt, inner sep=1pt, outer
sep=0},
morphism/.style={fill=white,circle,draw,thin, inner sep=1pt, minimum size=15pt,
                 scale=0.8},
small_morphism/.style={fill=white,circle,draw,thin,inner sep=1pt,
                       minimum size=10pt, scale=0.8},
coupon/.style={draw,thin, inner sep=1pt, minimum size=18pt,scale=0.8},
semi_morphism/.style args={#1,#2}{
                  fill=white,semicircle,draw,thin, inner sep=1pt, scale=0.8,
                  shape border rotate=#1,
                  label={#1-90:#2}},
regular/.style={densely dashed}, 
edge/.style={very thick, draw=green, text=black},
overline/.style={preaction={draw,line width=2mm,white,-}},
thick_overline/.style={preaction={draw,line width=3mm,white,-}},
really_thick/.style={line width=3mm, gray},
boundary/.style={thick,  draw=blue, text=black},
cell/.style={fill=black!10},
subgraph/.style={fill=black!30},
midarrow/.style={postaction={decorate},
                 decoration={
                    markings,
                    mark=at position #1 with {\arrow{>}},
                 }},
midarrow/.default=0.5,
midarrow_rev/.style={postaction={decorate},
                 decoration={
                    markings,
                    mark=at position #1 with {\arrow{<}},
                 }},
midarrow_rev/.default=0.5
}

\newcommand{\al}{\alpha}
\newcommand{\albar}{{\overline{\alpha}}}
\newcommand{\be}{\beta}
\newcommand{\ga}{\gamma}
\newcommand{\Ga}{\Gamma}

\newcommand{\ph}{\varphi}

\newcommand{\Si}{\Sigma}

\newcommand{\R}{\mathbb R}

\newcommand{\lmb}{\lambda}

\newcommand{\Z}{\mathcal{Z}}      
\newcommand{\ZMu}{\mathcal{Z}^{\text{M\"u}}}
\newcommand{\VV}{\mathbf{V}}
\newcommand{\Graph}{\text{Graph}}
\newcommand{\VGraph}{\text{VGraph}}

\newcommand{\ZTV}{Z_\text{TV}}    
\newcommand{\ZRT}{Z_\text{RT}}    
\newcommand{\ZCY}{Z_\text{CY}}
\newcommand{\hZ}{{\hat{Z}}}
\newcommand{\hatZCY}{\hat{Z}_\text{CY}}
\newcommand{\hatZTV}{\hat{Z}_\text{TV}}

\newcommand{\Zel}{{\mathcal{Z}^\text{el}}}
\newcommand{\ZA}{{\cZ(\cA)}}
\newcommand{\ZZA}{{\Zel(\cA)}}
\newcommand{\Iel}{\mathcal{I}^\text{el}}
\newcommand{\torus}{\mathbf{T}^2}
\newcommand{\punctorus}{\mathbf{T}_0^2}
\newcommand{\disk}{\DD^2}
\newcommand{\Disk}{\mathcal{D}isk_n^{or}}
\newcommand{\PI}{{P\times I}}
\newcommand{\clI}{{[0,1]}}
\newcommand{\Rex}{{\mathcal{R}ex}}
\newcommand{\cB}{{\mathcal{B}}}
\newcommand{\cD}{{\mathcal{D}}}
\newcommand{\cV}{{\mathcal{V}}}
\newcommand{\cN}{{\mathcal{N}}}
\newcommand{\Ann}{\text{Ann}}
\newcommand{\hA}{{\hat{\cA}}}
\newcommand{\hB}{{\hat{\mathcal{B}}}}

\newcommand{\lact}{\vartriangleright}
\newcommand{\ract}{\vartriangleleft}
\newcommand{\ihom}[2]{\Hom_{#1}^{#2}}

\newcommand{\cM}{\mathcal{M}}      

\newcommand{\hatbox}[1]{{\hat{\boxtimes}_{#1}}}
\newcommand{\ZSig}[1]{{\hatZCY(\Sigma_{#1})}}

\newcommand{\injto}{\hookrightarrow}          

\DeclareMathOperator{\Kar}{Kar} 
\DeclareMathOperator{\htr}{hTr} 

\newcommand{\Vect}{\mathcal{V}ec}  
\newcommand{\cc}[1]{\underset{\scriptstyle #1}{\circ}}

\newcommand{\st}{\; | \;}                               

\newcommand{\ldim}[1]{d_{#1}^L}
\newcommand{\rdim}[1]{d_{#1}^R}
\newcommand{\rdual}[1]{\prescript{*}{}#1}

\begin{document}

\title{Factorization Homology and 4D TQFT}
\author{Alexander Kirillov, Jr.}
   \address{Department of Mathematics, Stony Brook University,
            Stony Brook, NY 11794, USA}
    \email{kirillov@math.stonybrook.edu}
    \urladdr{http://www.math.sunysb.edu/\textasciitilde kirillov/}
\author{Ying Hong Tham}
   \address{Department of Mathematics, Stony Brook University, 
            Stony Brook, NY 11794, USA}
    \email{yinghong.tham@stonybrook.edu}
    \urladdr{http://www.math.sunysb.edu/\textasciitilde yinghong/}

\begin{abstract}
In \ocite{balsam-kirillov}, it is shown that the Turaev-Viro invariants defined for a 
spherical fusion category $\cA$ extends to invariants of 3-manifolds with corners.
In \cite{kirillov-stringnet}, an equivalent formulation for the 2-1 part of the theory 
(2-manifolds with boundary) is described using the space of ``stringnets with boundary 
conditions" as the vector spaces associated to 2-manifolds with boundary.
Here we construct a similar theory for the 3-2 part of the 4-3-2 theory in \cite{CY}.\\
\end{abstract}
\maketitle
\input{intro}

\input{facthom_overview}

\input{balanced_tensor_product}
\input{stringnets_2D}

\input{stringnets_3D}
\input{skein}
\input{excision}
\input{computations.tex}

\input{elliptic_center}

\input{appendix}
\input{biblio}

\end{document}

%% file: intro.tex
\section{Introduction}\label{s:intro}
The notion of factorization homology for topological manifolds was introduced 
by Ayala and Francis (see \ocite{AF19, AFR}) following earlier work of 
Beilinson and Drinfeld and many others. The main idea of this construction is quite natural: 
it allows one to construct invariants of $n$-dimensional manifolds by ``gluing'' local data
associated to balls embedded in $M$. A simple example of such a construction is 
the usual homology $H_*(M,A)$, where $A$ is an abelian group. More general form of 
factorization homology uses as input the following algebraic data:

\begin{itemize}
\item An object $\cA$ in a symmetric monoidal  $\infty$-category $\mathcal{V}$ (this is the object assigned to the ball) 
\item A structure of an algebra over the operad of (framed) $n$-disks  on $\cA$; this is used to define the gluing of local data
\end{itemize}

As an output, the factorization homology of an $n$-dimensional manifold $M$ with coefficients in $\cA$ 
(denoted $\int_M\cA$) gives again an object of $\mathcal{V}$. 

Unfortunately, the precise definition of factorization homology has some drawbacks. 
First, it is given in the language of $\infty$-categories, so it is rather technical. 
More importantly, factorization homology is defined by suitable universality properties, 
so this definition is not very explicit; in fact, even existence of such an object is non-trivial. 

The main goal of our note is giving an explicit construction of factorization 
homology in two special cases:

\begin{enumerate}
\item $n=1$, $\cA$ is a spherical fusion category. 

\item $n=2$, $\cA$ is a premodular category. 
\end{enumerate}

In both cases, we take the  target category $\mathcal{V}$ to be the $(2,1)$
category $\Rex$ of essentially small finitely co-complete $\kk$-linear  
categories and right exact functors, as defined in \ocite{BBJ1}; the symmetric monoidal
structure on $\Rex$ is given by Kelly product. In particular, it contains
the $(2,1)$ category of small $\kk$-linear abelian categories as a full subcategory. We 
will discuss category $\Rex$ in more detail in \seref{s:fact_homology}.

The $n=1$ case is rather simple:  the only non-trivial 1-manifold is $S^1$,
and it is easy to show that for a spherical fusion category $\cA$,
$\int_{S^1}\cA=Z(\cA)$ is the Drinfeld center of $\cA$ (as an abelian
category; the monoidal structure on the Drinfeld center requires additional
construction). Yet we include this case  as it is necessary to understand
the  $n=2$ case.

The $n=2$ case has been studied in the papers of Ben-Zvi, Brochier, and
Jordan \ocite{BBJ1, BBJ2}.

In both cases, we show that one can give  an explicit definition of
factorization homology  $\int_M\cA$ using suitable colored graphs in
dimension $n+1$ modulo an equivalence relation  generated by local moves;
this follows the general ideas first suggested by Walker in \ocite{walker}.
For $n=1$, such local relations were first explicitly written by 
physicists Levin and Wen in \ocite{levin-wen}, who dubbed such graphs on
surfaces ``stringnets''. A rewriting of this notion in a more mathematical
language can be found in \ocite{kirillov-stringnet}, where it is shown that
the stringnets and their boundary conditions coincide with the 2-1 part of
Turaev-Viro $(2+1)$-dimensional TQFT.

For $n=2$, the corresponding colored graphs are ribbon graphs in 3
dimensions; the space of such graphs modulo local relations is commonly
called the {\em skein module},  see e.g. \ocite{freyd}. This  space is part
of a  $(3+1)$-dimensional TQFT, which is usually called Crane-Yetter TQFT,
introduced in \ocite{CY}.

The main results of this note are \thmref{t:excision-skein}, which shows that
the space of colored graphs satisfies the excision property and thus
coincides with factorization homology (\corref{c:skein_facthom},
and \thmref{thm:cy_modular}, which
shows that in the case when $n=2$ and $\cA$ is modular, the category
$\ZCY(\Sigma)$ assigned to a surface $\Sigma$  in Crane-Yetter theory based
on category $\cA$, only depends on the number of boundary components of
$\Sigma$. In particular, in the case when $\Sigma$ is closed,
$\ZCY(\Sigma)$ is trivial and doesn't depend on the genus of $\Sigma$. This
result has been widely expected before (see e.g. unpublished notes of Freed
and Teleman \ocites{freed1, freed2}), but to the best of our knowledge, no
formal proof has been published yet.

\subsection*{Acknowledgments}
The authors would like to thank Pavel Etingof, Eugene Gorsky, Jin-Cheng Guu,
Dennis Sullivan, and Pavel Safronov for helpful discussions. We woudl also like 
to thank the referees for their many suggestions. 

\subsection*{Notation and conventions}

Throughout the paper, the word ``manifold'' stands for a smooth manifold
which admits a finite open cover such that any finite intersection of the
open subsets forming the cover is diffeomorphic to affine space. It is
known that this is equivalent to requiring that $M$ is diffeomorphic to the
interior of a compact manifold with boundary.

Throughout the paper, we fix an algebraically closed field $\kk$ of
characteristic 0. All abelian categories considered in this paper will be
locally finite $\kk$-linear categories (see \ocite{EGNO}*{Section 1.8}). 
In particular, we denote by $\Vect$ the category of finite-dimensional
vector spaces over $\kk$.

We will be denoting by $\boxtimes$ the Deligne tensor product of abelian
categories, see \ocite{EGNO}*{Section 1.11}; it is well defined for locally
finite $\kk$-linear categories and coincides with the Kelly tensor product
in $Rex$ (see \ocite{BBJ1}*{Section 3.2} for references).

We will be heavily using notions of fusion category, pivotal and spherical
fusion category, and premodular category. We refer the reader to
\ocite{EGNO} for definitions and basic properties of such categories. In
most of our formulas and computations, we will be suppressing the
associativity and unit morphisms, as well as the pivotal morphism $V\simeq
V^{**}$.

We also be using graphical presentations of morphisms.
We use the convention that morphisms go from top to bottom,
and the braiding of a braided category is represented by
the ` \textbackslash ' strand going over the ` / ' strand;
see Appendix for more details.

\subsection*{Note} 
While this paper was in preparation, we have received a preprint of Juliet Cooke \ocite{cooke}, 
which contains the result very similar to ours. Yet the proof and exposition are different, 
so we hope that many readers will find our work useful.

%% file: facthom_overview.tex
\section{Factorization homology overview}\label{s:fact_homology}
In this section, we give a brief summary  of the theory of factorization homology. 
We only try to cover as much as is necessary for our purposes, referring the reader 
to review \ocite{AF19} and original papers cited there for details.

To keep things simple, we will only consider the theory for oriented manifolds, 
ignoring other possible choices of framing structures. All manifolds considered 
here will be smooth finitary, i.e. those that are interiors of compact manifolds 
with (possibly empty) boundary.

We define the symmetric monoidal category $\Disk$ as the category whose objects
are finite disjoint unions of copies of $\R^n$ (or, equivalently, open unit ball
$B^n$) and morphisms are orientation-preserving embeddings. The set of
embeddings is considered as  a topological  space, with compact-open $C^\infty$
topology, so $\Disk$ becomes a topological category, and thus, an
$\infty$--category. The monoidal structure is given by the  disjoint union.

Given a symmetric monoidal $\infty$-category $\cV$, an {\em $n$-disk algebra} in 
$\mathcal{V}$ is a symmetric monoidal  functor of $\infty$-categories 
$$
\Disk\to \cV.
$$
In particular, this defines an object $A\in \Obj(\cV)$ (namely, the image of the
standard unit $n$-disk). Abusing the language, we will also call $A$ a ``disk
algebra''. 

Given such a disk algebra $A$, one defines for any oriented $n$-manifold $M$ the
factorization homology 
$$
\int_M A\in\Obj \cV
$$
as a certain colimit, over all embeddings of collections of oriented disks into
$M$.  Note that existence of such factorization homology is not guaranteed: in
order for it to be defined, we need $\cV$ to have sufficiently many colimits.
We do not reproduce the definition here; instead, we state some of the
properties of this construction,
and give a list of properties (\thmref{t:facthom-characterize})
which characterizes factorization homology
(\ocite{AF15}, \ocite{AFT17}).
We refer the reader to the original papers for
details and proofs. 

\begin{theorem}
So defined factorization homology satisfies the following properties:
\begin{enumerate}
\item For an open $n$-ball $B^n$, we have 
$$
\int_{B^n}A = A
$$
\item Factorization homology is functorial with respect to open embeddings: for any open embedding of oriented $n$-manifolds $i\colon M\injto N$, we have a functor 
$$
i_*\colon \int_MA\to \int_N A
$$
\item It sends disjoint union to tensor product in $\mathcal{V}$: 
$$
\int_{M\sqcup N} A=\Bigl(\int_M A\Bigr) \boxtimes \Bigl(\int_N A\Bigr) 
$$
\end{enumerate}
\end{theorem}

Let us now restrict our attention to the special case when the target category
$\cV$ is the $(2,1)$ category $\Rex$ of finitely co-complete $\kk$-linear
categories with right exact functors, as defined in \ocite{BBJ1}. We won't
repeat the definition, listing instead the properties of this category; we refer
the reader to \ocite{BBJ1}*{Section 3} for proofs. 

\begin{enumerate}
	\item The category $\Rex$ is equivalent (as a $(2,1)$ category) to the
	category $\mathcal{P}r_c$ of  compactly generated presentable $\kk$-linear
	categories with compact and cocontinuous functors \ocite{BBJ1}*{Section 3.1}

	\item Category $\Rex$ is closed under small colimits
	\ocite{BBJ1}*{Proposition 3.5}

	\item Category $\mathcal{P}r_c$ (and thus $\Rex$) includes as a full
	subcategory the $2$-category of small $\kk$-linear abelian categories, with
	right exact functors

	\item Category $\Rex$ has a symmetric monoidal structure, given by so-called
	Kelly product. If $\cA, \cB$ are abelian category such that Deligne tensor
	product $\cA\boxtimes \cB$ (see \ocite{EGNO}*{Section 1.11}) is defined,
	then Kelly tensor product coincides with Deligne tensor product
	\ocite{BBJ1}*{Section 3.2}. From now on, we use notation $\boxtimes$ for the
	Kelly tensor product
   
	\item Let $\cA$ be an $E_1$ algebra in $\Rex$. Then we can define the notion
	of left (respectively, right) $\cA$-modules in $\cV$ and define relative
	tensor product $\cM\boxtimes_{\cA}\cN$ as an object of $\cV$. In the special
	case when $\cA$ is a multitensor category (i.e. a locally finite
	$\kk$-linear abelian rigid monoidal category), and $\cM$, $\cN$ are abelian
	categories, the relative Kelly product $\cM\boxtimes_{\cA}\cN$ is also an
	abelian category and coincides with the balanced tensor product (see
	\seref{s:module} below) whenever the latter is defined \ocite{BBJ1}*{Remark
	3.15, Corollary 3.18}.

	\item Balanced braided tensor categories are 2-disk algebras in $\Rex$
	(see \ocite{BBJ1}*{Section 3.3}).
\end{enumerate}

Using these properties, it was shown in \ocite{BBJ1} that factorization homology
with values in $\cV$ is well-defined (they were working in the case $n=2$;
however, it is trivial to see that the same also applies in case $n=1$)

We will now formulate the final property of the  factorization homology, called
the {\em excision property}. For simplicity, we will only do so in the case when
the target category $\cV$ is the $(2,1)$ category $\Rex$, even though it also
holds for any factorization homology.

Before formulating the excision property, we need the following lemma
about
state some simple corollaries of the properties above.  
\begin{lemma}\label{l:skein_alg}
    Let $n=1$ or $2$ and let $A$ be an $n$-disk algebra in $\cV=\Rex$. Then 
\begin{enumerate}
\item For any $n-1$-dimensional oriented manifold $N$, let $A(N)=\int_{N\times I} A$, 
where $I=(0,1)$ is the open interval. Then $A(N)$ has a canonical structure of an 
$E_1$   algebra in $\Rex$, with the multiplication coming from embedding 
$(N\times I)\sqcup(N\times I)\to N\times I$ (``stacking'').

\item Let $M$ be an $n$-dimensional manifold with boundary;
we denote by $M^{o}$ the interior of $M$.
Let $N$ be a boundary component of $M$.
Assume that we are given a diffeomorphism of a neighborhood of
$N$ in $M^o$ with $N\times (0,1)$; 
we will call such an isomorphism a {\em collaring}
or a {\em collared structure} at $N$.

Then this gives on  $\int_{M^o} A$ a natural structure of a module  over $A(N)$. 
\end{enumerate}
\end{lemma}

One of the goals of this paper is to show that in some special cases, this
algebra and module structure coincide with so-called {\em skein algebra}
(respectively, {\em skein module}), see \ocite{walker}, \ocite{freyd}.

We can now state the final property of factorization homology. 
\begin{theorem}\label{t:excision}
    In the assumptions of \lemref{l:skein_alg}, 
    factorization homology satisfies the following {\em excision property}: 

Let $M_1$, $M_2$ be  $n$-manifolds with boundary, and $M_i^o$ the interior of 
$M_i$. Let $N_1$, $N_2$ be  connected components of the boundary of $M_1, M_2$ 
respectively,  together with a diffeomorphism $N_1\simeq \ov{N_2}$ (where bar 
stands for opposite orientation). Moreover, assume we are given collared 
structure at $N_1$, $N_2$ as in \lemref{l:skein_alg}.

Let $M$ be the manifold obtained by gluing together $M_1$ with $M_2$ using $\ph$; 
choice of collared structures gives a smooth structure on $M$. 

Then one has an equivalence of categories
$$
\int_{M} A=\Bigl(\int_{M^o_1} A\Bigr)\boxtimes_{A(N)} \Bigl(\int_{M^o_2} A\Bigr)
$$
where $\boxtimes_A$ is the relative tensor product in $\Rex$. 
\end{theorem}

Finally, the following properties uniquely characterize factorization homology
(this is a recasting of \ocite{BBJ1}*{Theorem 2.5},
in turn based on \ocite{AF15}, \ocite{AFT17}):

\begin{theorem}
Let $A$ be a 2-disk algebra in $\Rex$.
Then the functor $\int_- A$ satisfies, and is characterized by,
the following properties:
\begin{enumerate}
\item If $U$ is contractible, then there is an equivalence in $\Rex$,
	\[ \int_U A \simeq A \]

\item The $E_1$ ``stacking'' structure on $A(N)$ in \lemref{l:skein_alg}
	is unique (any two diffeomorphisms $Y \simeq N\times I$
	respecting the fibre structure $N \times I \xrightarrow{proj} I$
	induces equivalent $E_1$-algebra structures on $\int_Y A$).

\item The excision property of \thmref{t:excision} holds.
\end{enumerate}
\label{t:facthom-characterize}
\end{theorem}

\comment{
Moreover, we claim that this (together with $\int_{B^n} A = A$) uniquely defines
factorization homology. Let $Man^n$ be the category of oriented $n$-manifolds
with morphisms being orientation-preserving embeddings; it has a natural
structure of a symmetric monoidal $\infty$ category which contains the category
$\Disk$ as a subcategory 

\begin{theorem}
\label{t:uniqueness}
    Let $n=1$ or $2$ and let $A$ be an $n$-disk algebra in $\cV=\Rex$. 
    Let $F\colon Man^n\to \Rex$ be a symmetric  monoidal functor which has the following properties 
    \begin{enumerate}
        \item Restriction of $F$ to $\Disk$ coincides with the disk algebra in $\Rex$ given by $A$
        \item $F$ satisfies the excision property: in the notation of \thmref{t:excision}, we have 
        $$
        F(M)=F(M_1^o)\boxtimes_{F(N\times I)} F(M_2^o)
        $$
   \end{enumerate}
Then for any oriented $n$-manifold $M$,  $F(M)$ is isomorphic in $\Rex$ to the factorization homology $\int_M A$. 
\end{theorem}

Note that we claim a relatively weak statement of isomorphism, without discussing whether this isomorphism is canonical. 
\begin{proof}
 We give a proof for $n=2$; for $n=1$, it is trivial.    
    
 First, notice that $S^2_n$, the sphere with $n$ punctures, can be
 obtained by gluing two copies $X,X'$  of the $n$-gon with vertices removed.
 Since the interior of each of $X, X'$ is a disk,
 we see that $F(X^o)=\int_{X^o}A\simeq A$;
 similarly, since boundary $N=\del X$ is a disjoint union of intervals,
 $N\times I$ is a disjoint union of disks and thus $F(N\times I)=A(N)$.
 Thus, excision property for $F$ and for factorization homology
 implies that $F(S^2_n)\simeq \int_{S_2^n}A$.
 In particular, it implies that for an annulus
 $\Ann=S^1\times I\simeq S^2_2$, we have $F(\Ann)\simeq \int_{\Ann} A$.

Since any oriented surface $\Sigma$ (possibly with punctures) can be
obtained by gluing pieces isomorphic to $S^2_{n,k}$,
the sphere with $k$ boundary components and $n$ punctures,
along union of circles, and for each such piece,
its interior is the sphere with $n+k$ punctures,
it again follows from excision for $F$
and for factorization homology that $F(\Sigma)\simeq \int_\Sigma A$.

\end{proof}

}

%% file: balanced_tensor_product.tex
\section{Module categories, balanced tensor product, and center}\label{s:module}
In this section, we review the results about balanced tensor product of module 
categories.
Our main goal is to give two constructions of the center
of an $\cC$-bimodule category $\cM$ -
$\cZ_\cC(\cM)$ (\defref{d:center}) and $\htr_\cC(\cM)$ (\defref{d:htr}),
and show that when $\cC$ is pivotal multifusion,
they are equivalent (\thmref{t:htr}).

Recall our convention that all categories considered in this paper
are locally finite  $\kk$-linear. Most of the time, they will be abelian; however, 
in some cases we will need to use $\kk$-linear additive (but not necessarily abelian) 
categories. For such a  category $\cA$, we will denote by $\Kar(\cA)$ the Karoubi envelope (also known as idempotent completion) of $\cA$. By definition, an object of $\Kar(\cA)$ is a pair $(A, p)$, where $A$ is an object of $\cA$ and $p\in \Hom_\cA(A,A)$ is an 
idempotent: $p^2=p$. Morphisms in $\Kar(\cA)$ are defined by 
$$
\Hom_{\Kar(\cA)}((A_1, p_1), (A_2, p_2))=\{f\in \Hom_\cA(A_1, A_2)\st p_2fp_1=f\}
$$

Throughout this section $\cC$ is a pivotal category,
though in the definitions $\cC$ is only required to be monoidal.
When $\cC$ is multifusion,
we use the conventions and notation laid out in the Appendix.
In particular, $\Irr(\cC)$ is the set of isomorphism classes,
$\Irr_0(\cC)$ are those simples appearing as direct summands
of the unit $\one$,
$\{X_i\}$ will be a fixed set of representatives of $\Irr(\cC)$,
$d_i^R$ is the (right) dimension of $X_i$,
and we will be using graphical presentation of morphisms.

We assume that the reader is familiar with the notions of module categories;
for a left module category $\cM$ over $\cC$, we will denote the action of 
$A\in\cC$ on $M\in\cM$ by $A\lact M$. Similarly, we use $M\ract A$ for right action. 
In this paper, all module categories are assumed to be semisimple (as abelian categories).


This section is organized as follows:
Subsection~\ref{s:zM} provides the definition
and some properties of $\cZ_\cC(\cM)$,
Subsection~\ref{s:htrM} does so for $\htr_\cC(\cM)$,
and Subsection~\ref{s:equiv} shows that when
$\cC$ is pivotal multifusion,
these definitions are essentially the same.

\subsection{$\cZ_\cC(\cM)$} \label{s:zM}\par \noindent

The following definition is essentially given in  \ocite{GNN}*{Definition 2.1}
(there $\cC$ is assumed to be fusion).

\begin{definition}\label{d:center}
Let $\cC$ be a finite multitensor  category,
and let $\cM$ be a $\cC$-bimodule category.
The center of $\cM$, denoted $\cZ_\cC(\cM)$, is the category with the following objects and morphisms:

Objects: pairs $(M,\ga)$, where $M\in \cM$ and $\ga$ is an isomorphism of functors  
$\ga_A\colon A \lact M\to M \ract A$, $A\in \cC$ (half-braiding) satisfying natural compatibility conditions.

Morphisms: $\Hom ((M,\ga), (M',\ga'))=\{f\in \Hom_\cM(M,M')\st f\ga=\ga' f\}$. 
\end{definition}

In particular, in the special case $\cM=\cC$, this construction gives the Drinfeld center $\cZ(\cC)$. 

\begin{remark}
    Equivalently, the center $\cZ_\cC(\cM)$ can be described as the category of $\cC$-bimodule functors $\cC\to \cM$; see \ocite{GNN} for details. 
\end{remark}    


\begin{theorem}\label{t:center}\par\noindent
Let $\cC$ be pivotal multifusion,
and $\cM$ a $\cC$-bimodule category.
\item Let $F\colon  \cZ_\cC(\cM)\to \cM$ be the natural forgetful functor $F\colon  (M,\ga) \mapsto M$. 
Then it has a two-sided adjoint functor $I\colon \cM\to \cZ_\cC(\cM)$, given by 
\begin{equation}
\label{e:induction}
I(M)=\bigoplus_{i\in \Irr(\cC)} X_i \lact M \ract X_i^*
\end{equation}
with the half-braiding shown in \firef{f:I(M)}.
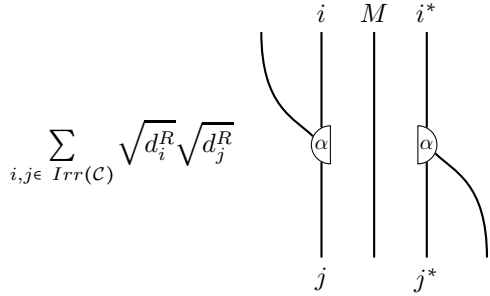
\begin{figure}[ht]
$\displaystyle{\sum_{i,j \in\ Irr(\cC)}\sqrt{d_i^R}\sqrt{d_j^R}}$\quad 
\begin{tikzpicture}
\draw (0,1.5)--(0,-1.5);
\node[above] at (0,1.5) {$M$};
\node[semi_morphism={90,}] (L) at (-0.7,0) {$\al$};
\node[semi_morphism={270,}] (R) at (0.7,0) {$\al$};
\draw (L)-- +(0,1.5) node[above] {$i$}; \draw (L)-- +(0,-1.5) node[below] {$j$}; 
\draw (R)-- +(0,1.5)node[above] {$i^*$}; \draw (R)-- +(0,-1.5)node[below] {$j^*$}; 
\draw (-1.5, 1.5) .. controls +(down:1cm) and +(135:0.5cm) .. (L);
\draw (1.5, -1.5) .. controls +(up:1cm) and +(-45:0.5cm) .. (R);
\end{tikzpicture}
\caption{Half-braiding on $I(M)$. See Notation~\ref{n:summation} in Appendix for definition of $\al$.}
\label{f:I(M)} 
\end{figure}

The adjunction isomorphism for $F\colon \cZ(\cM) \rightleftharpoons \cM \colon  I$,
\[
\Hom_{\cZ(\cM)}((M_1,\ga), I(M_2)) \simeq \Hom_\cM(M_1,M_2)
\]
is given by:
\begin{align}
\label{e:adj_isom}
\sum_{i\in \Irr(\cC)} 
\begin{tikzpicture}
\node[morphism] (ph) at (0,0) {$\ph_i$};
\draw (ph) -- +(90:1cm) node[above] {$M_1$};
\draw (ph) -- +(270:1cm) node[below] {$M_2$};
\draw[midarrow] (ph) to[out=180,in=90] (-1,-1);
\node at (-0.9,-0.1) {$i$};
\draw[midarrow_rev] (ph) to[out=0,in=90] (1,-1);
\node at (0.9,-0.1) {$i$};
\end{tikzpicture}
&\mapsto
\sum_{l \in \Irr_0(\cC)}
\begin{tikzpicture}
\draw[midarrow_rev={0.5}] (0,0) circle(0.5cm);
\node at (-0.7,0.1) {$l$};
\draw (0,1) -- (0,-1) node[pos=0, above] {$M_1$} node[below] {$M_2$};
\node[small_morphism] at (0,0.5) {$\ga$};
\node[small_morphism] at (0,-0.5) {\small $\ph_l$};
\end{tikzpicture}
\\
\label{e:adj_isom_2}
\sum_{j\in \Irr(\cC)} \sqrt{d_j^R}
\begin{tikzpicture}
\node[small_morphism] (ga) at (0,0.3) {$\ga$};
\draw (ga) -- (0,1) node[above] {$M_1$};
\draw (ga) -- (0,-1) node[below] {$M_2$};
\draw[midarrow] (ga) to[out=180,in=90] (-1,-1);
\node at (-0.9,0.1) {$j$};
\draw[midarrow_rev] (ga) to[out=0,in=90] (1,-1);
\node at (0.9,0.1) {$j$};
\node[small_morphism] at (0,-0.3) {\small $f$};
\end{tikzpicture}
&\leftmapsto
\hspace{10pt}
\begin{tikzpicture}
\node[morphism] (f) at (0,0) {$f$};
\draw (f) -- +(90:1cm) node[above] {$M_1$};
\draw (f) -- +(270:1cm) node[below] {$M_2$};
\end{tikzpicture}
\end{align}
(Note the sum on the right in \eqnref{e:adj_isom}
is over $\Irr_0(\cC)$ and not $\Irr(\cC)$.)
The other adjunction isomorphism
for $I\colon \cM \rightleftharpoons \cZ(\cM) :F$,
\[
\Hom_\cM(M_1,M_2) \simeq \Hom_{\cZ(\cM)}(I(M_1), (M_2,\ga)) 
\]
is given by a similar formula,
essentially obtained by rotating all the diagrams above.
\end{theorem}

Note that the isomorphisms here differ slightly from that
of \ocite{kirillov-stringnet}.
We punt the proof to the Appendix (on page \pageref{pf:t:center}).

An important special case is when $\cM=\cM_1\boxtimes\cM_2$, where $\cM_1$ is a right module  category over a pivotal multifusion category  $\cC$, and $\cM_2$ is a left module category over $\cC$. In this case, by \ocite{ENO10}*{Proposition~3.8},
one has that $\cZ_\cC(\cM_1\boxtimes \cM_2)$ is naturally equivalent to 
the balanced tensor product of categories:
\begin{equation}\label{e:center-product}
\cZ_\cC(\cM_1\boxtimes \cM_2)\simeq \cM_1\boxtimes_\cC \cM_2
\end{equation}
where the balanced tensor product is defined by the universal property: for any abelian  category $\cA$, we have a natural equivalence
$$
Fun_{bal}(\cM_1\times \cM_2, \cA)=Fun(\cM_1\boxtimes_\cC \cM_2, \cA)
$$
where $Fun$, $Fun_{bal}$ stand for category of  $\kk$-linear additive functors (respectively, category of $\kk$-linear additive  $\cC$- balanced functors); see details in \ocite{ENO10}*{Definition 3.3}.

Under the equivalence \eqref{e:center-product}, the natural functor
$\cM_1\boxtimes \cM_2\to \cM_1\boxtimes_\cC \cM_2$ 
is identified with the functor
$I\colon  \cM_1\boxtimes \cM_2\to \cZ_\cC(\cM_1\boxtimes \cM_2)$ 
constructed in \thmref{t:center}.

Recall that a functor $F\colon \cA\to \cB$, where $\cB$ is abelian and $\cA$ additive (not necessarily abelian) is called \emph{dominant} if any object of $\cB$ appears as a subquotient of $F(X)$ for some $X \in \Obj\cA$. Similarly, we say that a full  subcategory $\cA\subset \cB$ is dominant if any object of $\cB$ appears as a subquotient of some $X\in \Obj \cA$. In the case when $\cA$ is a full additive subcategory in a semisimple abelian category $\cB$, this immediately implies that the  Karoubi envelope of $\cA$ is equivalent to $\cB$ (in particular, this implies that $\Kar(\cA)$ is abelian).

\begin{proposition} \label{p:I_dominant}
Under the hypotheses of \thmref{t:center},
the functor $I\colon \cM \to \cZ(\cM)$ is dominant. Moreover, 
 any object $(M,\ga)$ is a direct summand of $I(M)$.
\end{proposition}
\begin{proof}
The adjunction isomorphism applied to $\id_M\in \Hom_\cM(M,M)$
provides the inclusion $(M,\ga) \subseteq I(M)$,
and the other adjunction isomorphism gives the projection
$I(M) \to (M,\ga)$;
see \lemref{l:Mga_proj} in Appendix for details.
\end{proof}

\begin{proposition} \label{p:ZM_ss}
Under the hypotheses of \thmref{t:center},
if $\cM$ is finite semisimple, then so is $\cZ_\cC(\cM)$.
\end{proposition}
\begin{proof}
Using exactness in $\cM$ of the left and right actions,
abelianness of $\cM$ transfers to $\cZ(\cM)$.
For example, the kernel $K$ of a morphism
$f: M_1 \to M_2$ such that $f\in \Hom_{\cZ(\cM)}((M_1,\ga^1),(M_2,\ga^2))$
would inherit a half-braiding $\ga^1|_K$.
Semisimplicity follows easily.
Finiteness follows from \prpref{p:I_dominant};
$I$ ensures there can't be too many simples in $\cZ(\cM)$.
See \ocite{tham_elliptic} for a similar proof for $\cM = \cC$.
\end{proof}

%
%
%

For applications, we will need to consider centers over
a full, dominant, monoidal subcategory $\cC' \subseteq \cC$.
Equivalently,
$\cC'$ is a pivotal category
whose Karoubi envelope is multifusion.

\begin{lemma} \label{l:ZM_sub}
Let $\cC'$ be a pivotal locally finite $\kk$-linear additive category
whose Karoubi envelope $\cC = \Kar(\cC')$ is multifusion.
Let $\cM$ be a $\cC$-bimodule category,
and hence naturally a $\cC'$-bimodule category (as before, we assume that
 $\cM$ is a semisimple abelian category).
Then there is a natural equivalence
\[
  \cZ_\cC(\cM) \simeq \cZ_{\cC'}(\cM)
\]
In particular, for right, left $\cC$-module categories
$\cM_1,\cM_2$, there is a natural equivalence
\[
  \cM_1 \boxtimes_\cC \cM_2 \simeq \cM_1 \boxtimes_{\cC'} \cM_2.
\]
\end{lemma}

\begin{proof}
The equivalence is given as follows:
objects $(M,\gamma)$ in $\cZ_\cC(\cM)$ are naturally objects in $\cZ_{\cC'}(\cM)$
by forgetting some of the half-braiding, i.e. $(M, \gamma|_{\cC'})$;
morphisms $f: (M,\gamma) \to (M',\gamma')$
are naturally morphisms $f: (M,\gamma|_{\cC'}) \to (M',\gamma|_{\cC'})$.
We need to check that this is an equivalence.

The functor is essentially surjective: any half-braiding over $\cC'$
can be completed to a half-braiding over $\cC$.
To see this, let $\gamma$ be a half-braiding over $\cC'$.
Let $X \in \Obj \cC \backslash \Obj \cC'$,
and let it be a direct summand of some $Y \in \Obj \cC'$,
$X \overset{\iota}{\underset{p}{\rightleftharpoons}} Y$.
Then we define the extension of $\gamma$ to $X$ by
$\gamma_X = (\id_{M_2} \ract p) \circ \gamma_Y \circ (\iota \lact \id_{M_1})$.
It is easy to check, using the semisimplicity of $\cC$,
that $\gamma_X$ is independent on the choice
of $Y$ and $p,\iota$.
It is also easy to check that the resulting extension
is indeed natural in $X$.

For morphisms, it is clear that this functor is faithful.
To show fullness, consider
$f \in \Hom_{\cZ_{\cC'}(\cM)}((M_1,\gamma^1),(M_2,\gamma^2))$.
We need to check that it also intertwines half-braiding
with $X \in \cC$, but this follows easily from the
definition of the extension of half-braiding given above.

Note since $\gamma$ has a unique extension to all of $\cC$,
this proof actually shows that the equivalence is an isomorphism.\\
\end{proof}

Note that in the proof above, we do not use the rigidity of $\cC$,
but we need it to conclude the second statement concerning
balanced tensor products.

\subsection{$\htr_\cC(\cM)$} \label{s:htrM}\par \noindent

Next we define the other notion of center.

\begin{definition}\label{d:htr}
Let $\cC$ be monoidal, and
$\cM$ a $\cC$-bimodule category.
Define the {\em horizontal trace} 
$\htr_\cC(\cM)$ as the  category with the following objects and morphisms: 

Objects: same as in $\cM$

Morphisms: $\Hom_{\htr_\cC(\cM)}(M_1,M_2)=
\bigoplus_X \ihom{\cM}{X}(M_1,M_2)/\sim$,
where $\ihom{\cM}{X}(M_1,M_2) := \Hom_\cM(X \lact M_1, M_2 \ract X)$,
the sum is over all (not necessarily simple) objects $X\in \cC$,
and $\sim$ is the equivalence relation generated by
the following:

For any $\psi\in \ihom{\cM}{Y,X}(M_1,M_2) := \Hom_\cM(Y \lact M_1, M_2 \ract X)$
  and $f\in \Hom_\cC(X,Y)$, we have 

\[
\begin{tikzpicture}
\node[morphism] (psi) at (0,0) {$\psi$};
\draw (psi) -- +(0,1) node[above] {$M_1$};
\draw (psi) -- +(0,-1) node[below] {$M_2$};
\coordinate (L) at (-1,1);
\coordinate (R) at (1,-1);
\node[above] at (L) {$X$};
\node[below] at (R) {$X$};
\draw (L) to[out=-90,in=170] (psi);
\draw (R) to[out=90,in=-10] (psi);
\node[morphism] at (-0.8,0.4) {$f$};
\node at (-0.5,-0.15) {$Y$};
\end{tikzpicture}
\quad \sim \quad
\begin{tikzpicture}
\node[morphism] (psi) at (0,0) {$\psi$};
\draw (psi) -- +(0,1) node[above] {$M_1$};
\draw (psi) -- +(0,-1) node[below] {$M_2$};
\coordinate (L) at (-1,1);
\coordinate (R) at (1,-1);
\node[above] at (L) {$Y$};
\node[below] at (R) {$Y$};
\draw (L) to[out=-90,in=170] (psi);
\draw (R) to[out=90,in=-10] (psi);
\node[morphism] at (0.8,-0.4) {$f$};
\node at (0.5,0.15) {$X$};
\end{tikzpicture}
\]
In other words,
$\Hom_{\htr_\cC(\cM)}(M_1,M_2) = \int^X \ihom{\cM}{X,X}(M_1, M_2)$
is the coend of the functor
$\ihom{\cM}{-,-}(M_1, M_2) : \cC^{\text{op}} \times \cC \to \Vect$
(see e.g. \ocite{maclane}).

Composition is given by
\[
\ihom{\cM}{Y}(M_2,M_3)) \tnsr \ihom{\cM}{X}(M_1,M_2)
\to \ihom{\cM}{Y\tnsr X}(M_1,M_3)
\]
which sends $\psi \tnsr \vphi$ to
\[
  Y \lact (X \lact M_1)
  \xxto{\id_Y \lact \psi} Y \lact (M_2 \ract X)
  \simeq (Y \lact M_2) \ract X
  \xxto{\vphi \ract \id_X} (M_3 \ract Y) \ract X
\]

For right, left $\cC$-module categories $\cM_1,\cM_2$,
we denote $\cM_1 \hatbox{\cC} \cM_2 = \htr_\cC(\cM_1 \boxtimes \cM_2)$.
\end{definition}

When the context is clear, we will drop the subscript $\htr = \htr_\cC$.
We will write $[\vphi] \in \Hom_{\htr(\cM)}(M_1,M_2)$ for the morphism
represented by $\vphi \in \ihom{\cM}{X}(M_1,M_2)$ for some $X$.\\

%

It can be shown that in a certain sense this definition is dual to the definition of center 
given above and is closely related to the notion of co-center as described in \ocite{DSSP}*{Section 3.2.2}. However, we will not be discussing the exact relation here. 

It is easy to see that the category $\htr(\cM)$ is 
additive but not necessarily abelian.

There is a natural inclusion functor $\htr \colon  \cM \to \htr(\cM)$
which is identity on objects,
and on morphisms it is the natural map
$\Hom_\cM(M_1,M_2) = \ihom{\cM}{\one}(M_1,M_2)
\to \Hom_{\htr(\cM)}(M_1,M_2)$.

The horizontal trace construction is functorial,
and in particular, we have
\begin{lemma}
Given a functor of $\cC$-bimodule categories
$F: \cM \to \cM'$,
there is a natural functor $\htr(F): \htr(\cM) \to \htr(\cM')$
that is the same as $F$ on objects.
\end{lemma}
\begin{proof}
Straightforward exercise left to the reader.
\end{proof}

We also consider $\cC' \subseteq \cC$
as in \lemref{l:ZM_sub},
but here we do not need rigidity nor semisimplicity on $\cC$:

\begin{lemma} \label{l:htr_sub}
Let $\cC'$ be monoidal,
and let $\cC = \Kar(\cC')$ be its Karoubi envelope.
Let $\cM$ be a $\cC$-bimodule category,
and hence naturally a $\cC'$-bimodule category.
Then there is a natural equivalence
\[
  \htr_{\cC'}(\cM) \simeq \htr_\cC(\cM)
\]
In particular, for right, left $\cC$-module categories
$\cM_1,\cM_2$, there is a natural equivalence
\[
  \cM_1 \hatbox{\cC} \cM_2 \simeq \cM_1 \hatbox{\cC'} \cM_2
\]

\end{lemma}

\begin{proof}
The equivalence is given by the identity map on objects,
and for two objects $M_1,M_2 \in \Obj \cM$,
the map on morphisms is given by completing the bottom arrow:
\[
\begin{tikzcd}
  \bigoplus_{X \in \cC'} \ihom{\cM}{X}(M_1,M_2)
    \ar[r] \ar[d]
  & \bigoplus_{X \in \cC} \ihom{\cM}{X}(M_1,M_2)
    \ar[d]
  \\
  \Hom_{\htr_{\cC'}(\cM)}(M_1, M_2)
    \ar[r]
  & \Hom_{\htr_\cC(\cM)}(M_1, M_2)
\end{tikzcd}
\]
It remains to prove that the bottom arrow is an isomorphism.

Let us first observe the following.
Let $X, Y\in \Obj \cC'$, and suppose $X$ is a
direct summand of $Y$, with
$X \overset{\iota}{\underset{p}{\rightleftharpoons}} Y$.
Let $\vphi \in \ihom{\cM}{X}(M_1,M_2)$.
Then $\vphi = \vphi \circ p \circ \iota
\sim \iota \circ \vphi \circ p \in \ihom{\cM}{Y}(M_1,M_2)$,
where we write $p,\iota$ instead of
$p \lact \id_{M_1}, \id_{M_2} \ract \iota$
for simplicity. This works for $\cC$ too.
Thus one can identify $\ihom{\cM}{X}(M_1,M_2)$
with a subspace of $\ihom{\cM}{Y}(M_1,M_2)$.

Surjectivity: Essentially, we need to show that any
morphism in $\htr_\cC(\cM)$ can be ``absorbed'' into $\htr_{\cC'}(\cM)$.
Let $[\vphi] \in \Hom_{\htr_\cC(\cM)}(M_1,M_2)$
be represented by some $\vphi \in \ihom{\cM}{X}(M_1,M_2)$.
By the above observation,
we can choose $Y \in \Obj \cC'$
with $X$ a direct summand of $Y$,
then $\vphi \in \ihom{\cM}{X}(M_1,M_2)$
is identified with some morphism in
$\ihom{\cM}{Y}(M_1,M_2)$,
so $[\vphi]$ is in the image.

Injectivity: Essentially, we need to show that
relations can also be ``absorbed'' into $\htr_{\cC'}(\cM)$.
Let $[\vphi] \in \Hom_{\htr_{\cC'}(\cM)}(M_1,M_2)$
that is sent to 0.
By the observation above, we may represent it by some
$\vphi \in \ihom{\cM}{Y}(M_1,M_2)$ for some $Y\in \Obj \cC'$.
Since it is 0 in $\Hom_{\htr_\cC(\cM)}(M_1,M_2)$,
there exists
\begin{itemize}
\item a finite collection of objects $J = \{A_j\} \subset \Obj \cC$
  so that $A_0 = Y$.
\item $\Phi_i \in \ihom{\cM}{A_{m_i},A_{n_i}}(M_1,M_2)$,
\item $f_i : A_{n_i} \to A_{m_i}$,
\end{itemize}
such that $\vphi = \sum_i f_i \circ \Phi_i - \Phi_i \circ f_i
\in \bigoplus_{A_j \in J} \ihom{\cM}{A_j}(M_1,M_2)$.

We want to be able to replace the $A_j$'s with objects in $\cC'$.
For each $j \neq 0$, choose some $B_j \in \Obj \cC'$
such that $A_j$ is a direct summand of $B_j$:
$A_j \overset{\iota_j}{\underset{p_j}{\rightleftharpoons}} B_j$.
For $j=0$, we take $B_0 = A_0 = Y$ and $\iota_0 = p_0 = \id_Y$.
This gives us maps
$\Theta_j: \psi \mapsto \iota_j \circ \psi \circ p_j:
  \ihom{\cM}{A_j}(M_1,M_2) \to \ihom{\cM}{B_j}(M_1,M_2)$.
Denote $\Theta = \sum \Theta_j$.

Now consider
\begin{itemize}
\item $L = \{B_j\} \subset \Obj \cC'$,
\item $\Psi_i = \iota_{n_i} \circ \Phi_i \circ p_{m_i}
  \in \ihom{\cM}{B_{m_i},B_{n_i}}(M_1,M_2)$,
\item $g_i = \iota_{m_i} \circ f_i \circ p_{n_i} :
  B_{n_i} \to B_{m_i}$.
\end{itemize}
It is a simple matter to verify that
$g_i \circ \Psi_i - \Psi \circ g_i =
  \Theta(f_i \circ \Phi_i - \Phi_i \circ f_i)$.
Hence $\vphi = \Theta(\vphi)
  = \Theta(\sum f_i \circ \Phi_i - \Phi_i \circ f_i)
  = \sum g_i \circ \Psi_i - \Psi \circ g_i$
is 0 in $\Hom_{\htr_{\cC'}}(M_1,M_2)$.

\end{proof}

\subsection{Equivalence} \label{s:equiv}\par \noindent

\begin{theorem}
  \label{t:htr}
Let $\cC$ be pivotal multifusion,
and $\cM$ a $\cC$-bimodule category.
One has a natural equivalence 
\[
  \Kar(\htr(\cM)) \simeq \cZ_\cC(\cM)
\]
Under this equivalence, the inclusion functor $\htr: \cM\to \htr(\cM)$ is 
identified with the functor $I : \cM\to \Z_\cC(\cM)$.

In particular, for right, left $\cC$-modules $\cM_1,\cM_2$,
we have
\[
  \cM_1 \boxtimes_\cC \cM_2 \simeq \cZ_\cC(\cM_1 \boxtimes \cM_2)
    \simeq \Kar(\cM_1 \hatbox{\cC} \cM_2)
\]

\end{theorem}
Before proving the theorem, we will need the following lemma.
\begin{lemma}\label{l:isom1}
The natural linear map 
\begin{equation}\label{e:isom1}
  \bigoplus_{i \in \Irr(\cC)} \ihom{\cM}{X_i}(M_1,M_2) \to \Hom_{\htr(\cM)}(M_1, M_2)
\end{equation}
is an isomorphism.
\end{lemma}
\begin{proof}
To prove the statement, we define a linear map  
$$
\Hom_{\htr(\cM)}(M_1, M_2) \to 
\bigoplus_{i \in \Irr(\cC)} \ihom{\cM}{X_i}(M_1,M_2)
$$
by 

\begin{equation}\label{e:isom2}
\psi\mapsto 
\sum_{i\in \Irr(\cC)} d_i^R \quad
\begin{tikzpicture}
\node[morphism] (psi) at (0,0) {$\psi$};
\draw (psi) -- +(0,1.5) node[above] {$M_1$};
\draw (psi) -- +(0,-1.5) node[below] {$M_2$};
\node[semi_morphism={90,}] (al1) at (-1,0.6) {$\al$};
\node[semi_morphism={90,}] (al2) at (1,-0.6) {$\al$};
\draw[midarrow={0.6}] (al1) to[out=-90,in=180] (psi);
\draw[midarrow={0.6}] (psi) to[out=0,in=90] (al2);
\node at (-0.7,-0.3) {$X$};
\node at (0.7,0.3) {$X$};
\draw[midarrow_rev] (al1) -- (-1,1.5) node[above] {$X_i$};
\draw[midarrow] (al2) -- (1,-1.5) node[below] {$X_i$};
\end{tikzpicture}
\end{equation}
for $\psi \in \ihom{\cM}{X}(M_1,M_2)$;
$\al$ is a sum over dual bases - see Notation~\ref{n:summation} in the Appendix.
\eqnref{e:isom2} is well-defined by \lemref{l:pairing_property}.
Using \lemref{l:summation}, it is easy to see that
\eqref{e:isom1} and \eqref{e:isom2} are mutually inverse. 
\end{proof}

\begin{proof}[Proof of \thmref{t:htr}]
Define the functor $G : \htr(\cM)\to \cZ_\cC(\cM)$ on objects by $G(M)=I(M)$, and on morphisms by 
\begin{equation}
G(\psi) = 
\sum_{i,j \Irr(\cC)} \sqrt{d_i^R}\sqrt{d_j^R}
\begin{tikzpicture}
\node[morphism] (psi) at (0,0) {$\psi$}; 
\node[semi_morphism={90,}] (L) at (-1,0) {$\al$};
\node[semi_morphism={270,}] (R) at (1,0) {$\al$};
\draw (psi)-- +(0,1.5) node[above] {$M_1$}; \draw (psi)-- +(0,-1.5) node[below] {$M_2$}; 
\draw[midarrow_rev] (L)-- +(0,1.5) node[pos=0.5,left] {$i$};
\draw[midarrow] (L) -- +(0,-1.5) node[pos=0.5,left] {$j$}; 
\draw[midarrow] (R) -- +(0,1.5) node[pos=0.5,right] {$i$};
\draw[midarrow_rev] (R)-- +(0,-1.5) node[pos=0.5,right] {$j$};
\draw[midarrow={0.6}] (L) to[out=-80,in=180] (psi);
\node at (-0.5, -0.5) {$X$};
\draw[midarrow_rev={0.6}] (R) to[out=-100,in=0] (psi);
\node at (0.5, -0.5) {$X$};
\end{tikzpicture}
\end{equation}
for $\psi \in \ihom{\cM}{X}(M_1,M_2)$;
once again see Notation~\ref{n:summation} in the Appendix
for definition of $\al$.

It is easy to check the following properties: 
\begin{enumerate}
\item $G$ is well-defined on morphisms (i.e. it preserves the equivalence relation):
this follows from \lemref{l:halfbrd}.

\item $G$ is dominant: any $Y\in \cZ_\cC(\cM)$ appears as a direct summand 
of  $G(M)$ for some $M\in \cM$. Namely, if $Y=(M,\ga)$,
then it appears as a direct summand of $G(M)$;
the projection to $Y$ is, up to a factor, $G(\sum d_i^R \ga_{X_i})$
(see \lemref{l:Mga_proj} in Appendix for proof;
compare \prpref{p:I_dominant}).

\item $G$ is bijective on morphisms:
by adjointness property (\thmref{t:center}), we have 
$$
\Hom_{\cZ_\cC(\cM)}(I(M), I(M')) \cong \Hom_{\cM}(I(M), M')
  =\bigoplus_{i}\Hom_\cM(X_i \lact M \ract X_i^*,  M')
$$
and by \lemref{l:isom1}, the right hand side coincides with $\Hom_{\htr(\cM)}(M,M')$.

\end{enumerate}

This immediately implies the statement of the theorem
by the universal properties of Karoubi envelopes.

\end{proof}

By \lemref{l:ZM_sub} and \lemref{l:htr_sub},
we extend the above theorem to $\cC' \subseteq \cC$:
\begin{corollary} \label{cor:center}
Let $\cC'$ be a pivotal category
whose Karoubi envelope $\cC = \Kar(\cC')$ is multifusion.
Let $\cM$ be a $\cC$-bimodule category,
and hence naturally a $\cC'$-bimodule category.
Then we have
\[
  \Kar(\htr_{\cC'}(\cM)) \simeq \Kar(\htr_\cC(\cM)) \simeq \cZ_\cC(\cM) \simeq \cZ_{\cC'}(\cM).
\]
\end{corollary}

Note $\Kar(\cM)$ inherits a $\cC'$-bimodule structure from $\cM$.
For example, $A \lact (M,p) = (A \lact M, \id_A \lact p)$.
We compare these constructions for $\cM$ and its Karoubi envelope:
\begin{lemma}
\label{lem:M_Karubi}
Under the same hypotheses as \corref{cor:center},
\[
  \Kar(\htr_{\cC'}(\cM)) \simeq \Kar(\htr_{\cC'}(\Kar(\cM)))
\]
In particular, if $\cM'$ is a dominant submodule category of $\cM$,
then
\[
  \Kar(\htr_{\cC'}(\cM')) \simeq \Kar(\htr_{\cC'}(\cM))
\]
\end{lemma}
\begin{proof}
The natural inclusion $\cM \to \Kar(\cM)$ 
is a full, dominant functor of $\cC'$-bimodules,
and it is easy to see that the corresponding functor
$\htr(\cM) \to \htr(\Kar(\cM))$
is also full and dominant.
It follows that the induced functor
on their Karoubi envelopes is an equivalence.

The second statement follows because $\Kar(\cM') \simeq \Kar(\cM)$.
\end{proof}

%% file: stringnets_2D.tex
\section{Colored Graphs in Turaev-Viro theory}\label{s:TV}

In this section, we recall the definition of colored graphs
(called stringnets in \ocite{kirillov-stringnet})
in Turaev--Viro theory. This is intended to serve as a reminder only; proofs 
are omitted. Details and proofs can be found in \ocite{kirillov-stringnet}.

Throughout this section, all surfaces are assumed to be oriented. We denote by 
$\cA$ a spherical fusion category. We will be heavily using graphical presentation of morphisms 
in $\cA$; we give a summary of our notation and conventions in the Appendix.

For a finite graph $\Ga$ embedded in surface $\Si$, we denote 
by  $E(\Ga)$ the set of edges. Note that edges are not
oriented. Let $E^{or}$ be the set of oriented edges, i.e. pairs $\ee=(e,
\text{orientation of } e)$; for such an oriented edge $\ee$, we denote by
$\bar{\ee}$ the edge with opposite orientation.

If $\Si$ has a boundary, the graph is allowed to have uncolored one-valent
vertices on $\del \Si$ but no other common points with $\del \Si$; all
other  vertices will  be called interior.  We will  call the edges of $\Ga$
terminating at these  one-valent vertices {\em legs}.   
\begin{definition}\label{d:coloring} Let $\Si$ an oriented surface
(possibly with boundary) and $\Ga\subset \Si$ --- an embedded graph as
defined above.  A {\em coloring} of $\Ga$ is the
following data:

  \begin{itemize}
    \item Choice of an object $V(\ee)\in \Obj \cA$ for every oriented edge
        $\ee\in E^{or}(\Ga)$ so that $V(\ov{\ee})=V(\ee)^*$.
    \item Choice of a vector $\ph(v)\in \<V(\ee_1),\dots,V(\ee_n)\>$ 
      (see  Appendix \eqref{e:vev})  for  every interior vertex $v$, where 
      $\ee_1, \dots, \ee_n$ are edges incident to $v$, taken in counterclockwise 
      order and with outward orientation (see Appendix \firef{f:coloring}). 
\end{itemize}


We will denote the set of all colored graphs on a surface $\Si$ by
$\Graph(\Si)$.
\end{definition}

\comment{
\begin{figure}[ht]
\begin{tikzpicture}
\node[morphism] (ph) at (0,0) {$\ph$};
\draw[->] (ph)-- +(240:1cm) node[pos=0.7, left] {$V_n$} ;
\draw[->] (ph)-- +(180:1cm);
\draw[->] (ph)-- +(120:1cm);
\draw[->] (ph)-- +(60:1cm);
\draw[->] (ph)-- +(0:1cm);
\draw[->] (ph)-- +(-60:1cm) node[pos=0.7, right] {$V_1$};
\end{tikzpicture}
\caption{Labeling of colored graphs}\label{f:coloring}
\end{figure}
}

Note that if $\Si$ has a boundary, then every colored graph $\Ga$ defines
a collection of points $B=\{b_1,\dots, b_n\}\subset \del \Si$ (the
endpoints of the legs of $\Ga$) and a collection of objects $V_b\in \Obj\
\cA$ for every $b \in B$: the colors of the legs of $\Ga$ taken with
outgoing orientation. We will denote the pair $(B, \{V_b\})$ by $\VV=\Ga\cap \del\Si$ 
and call it {\em boundary value}. We will denote 
$$
\Graph(\Si, \VV)=
    \text{set of all colored graphs in $\Si$ with boundary value } \VV.
$$ 

We can also consider formal linear combinations of colored graphs. Namely,
for fixed boundary value $\VV$ as above, we will denote 
\begin{equation}\label{e:vgr}
\VGraph(\Si,\VV)=\{\text{formal linear combinations of graphs }\Ga\in
\Graph(\Si,\VV)\}
\end{equation}
In particular, if $\del \Si=\varnothing$, then the only possible boundary
condition is trivial ($B=\varnothing$); in this case, we will just write
$\VGraph(\Si)$. 

It follows from result of Reshetikhin and Turaev that for every colored graph 
$\Ga$ in a disk $D\subset \R^2$, one can define its ``evaluation'' 
\begin{equation}\label{e:TV-evaluation}
    \<\Ga\>_D\in\<V(\ee_1),\dots, V(\ee_n)\>
  \end{equation}
  where $\ee_1,\dots, \ee_n$ are the edges of $\Ga$ meeting the boundary
  of $D$ (legs), taken in counterclockwise order and with outgoing orientation; 
  in particular, in the case when $\Ga$ is a star graph, with one vertex colored 
  by $\ph\in \<V(\ee_1),\dots, V(\ee_n)\>$, then $\<\Ga\>=\ph$.

We call a formal linear combination of colored graphs $\Ga=\sum c_i \Ga_i\in \VGraph(\Si, \VV)$ 
a {\em null graph} if there exists an embedded disk $D\injto \Si$ such that all graphs $\Ga_i$ 
meet boundary of $D$ transversally, all $\Ga_i$ coincide outside of $D$ (as colored graphs) and 
$$
\<\Ga\>_D=\sum c_i\<\Ga_i\cap D\>_D=0.
$$
We will say $\Gamma$ is null with respect to $D$.
We can now give the main definition of this section. 

\begin{definition}\label{d:skein_2d}
For an oriented surface $\Si$ and boundary condition $\VV=(B, \{V_b\})$ on $\del\Si$, 
we define the stringnet space by 
\begin{equation}
\ZTV(\Si, \VV)=\VGraph(\Si,\VV)/N 
\end{equation}
where $N$ is the subspace spanned by all null graphs (for all possible embedded disks).  
\end{definition}

As an example, it was shown in \ocite{kirillov-stringnet} that 
$$
\ZTV(S^2)=\ZTV(\R^2)=\kk
$$

We can now define the category of boundary conditions. 

\begin{definition}\label{d:Chat}
Let $N$ be an oriented 1-dimensional manifold,
possibly non-compact.
Suppose first $N$ has no boundary.
Define  $\hatZTV(N)$ as the category whose objects are finite 
subsets  $B\subset N$ together with a choice of object $V_b\in \Obj \cA$ for 
every point $b\in B$; we will use the notation $\VV=(B, \{V_b\})$ for 
such an object, and $B$ is called the set of marked points of $\VV$.
Define the  morphisms in $\hatZTV(N)$ by  
$$
\Hom_{\hatZTV(N)}(\VV, \VV')=\ZTV(N\times [0,1]; \VV^*,\VV'),\qquad
\VV=(B,\{V_b\}), \quad 
\VV'=(B',\{V_{b'}\})
$$ 
where $\VV^*, \VV'$ means the boundary condition obtained by putting
points $b\in B$ on the ``top'' $N\times\{1\}$, colored by objects
$V_b^*$ for outgoing legs (and thus colored by $V_b$ for incoming legs),
and putting points $b'\in B'$ on the ``bottom'' $N\times\{0\}$, colored by
objects $V_{b'}$ for outgoing legs.

\begin{figure}[ht]
\begin{tikzpicture}
\node[morphism] (ph) at (0,0) {$\ph$};
\coordinate (top) at (0,1);
\coordinate (bot) at (0,-1);
\coordinate (left) at (-1.2,0);
\coordinate (right) at (1.2,0);
\draw[<-] (ph)-- (-1,0|-top) node[pos=0.7, left] {$V_1$} ;
\draw[<-] (ph)-- (1,0|-top) node[pos=0.7, right] {$V_n$};
\draw[->] (ph)-- (-1,0|-bot) node[pos=0.7, left=2pt] {$V'_1$} ;
\draw[->] (ph)-- (1,0|-bot) node[pos=0.7, right] {$V'_m$};
\draw (left|-top)--(right|-top); 
\draw (left|-bot)--(right|-bot); 
\end{tikzpicture}
  \caption{Morphisms in $\hatZTV(N)$}
  \label{f:morphisms_ZTV}
\end{figure}

This category is additive and $\kk$-linear. We denote by 
\begin{equation}\label{e:ZTV-cat}
\ZTV(N)=\Kar(\hatZTV(N))
\end{equation}
its Karoubi envelope.

For $N$ with boundary, we define $\hatZTV(N) = \hatZTV(N \backslash \del N)$,
$\ZTV(N) = \ZTV(N \backslash \del N)$.
\end{definition}

It is immediate from the definition that 
$$
\ZTV(I)\simeq \cA.
$$
where $I$ is an open/closed interval.

It has been shown in \ocite{kirillov-stringnet} that $\ZTV(S^1)=\Z(\cA)$ is the 
Drinfeld center of $\cA$. We will reprove it (in a slightly different way) as a 
special case of a more general result later.

%% file: stringnets_3D.tex
\section{Skeins in Crane-Yetter Theory}
\label{sec:CY}

In this section, we give a definition of colored graphs/skeins
in Crane-Yetter theory, mirroring closely the previous section,
and we will reuse many definitions.
This definition essentially coincides with those given in
\ocite{freyd}, \ocite{cooke};
we use framed graphs instead of ribbons and coupons.

Throughout this section, all 3-manifolds are assumed to be oriented,
and may be non-compact and/or with boundary.
$\cA$ will be a skeletal premodular category;
see appendix for a summary of notation and conventions.

We will consider finite framed graphs $\Gamma$ in a 3-manifold $M$,
that is, $\Gamma$ is a smoothly embedded graph in $M$ with finitely
many edges, and each edge comes with a transversal
ray field along it
(that is, each point $p$ on an edge
is assigned a ray $\rho_p$ in $T_p M$
emanating from the origin, varying smoothly with $p$);
the transversal ray field $\rho_p$ is the \emph{framing} of the edge.
We also impose the condition that edges are not tangent to each other
at a vertex (this is necessary for the ``infinitesimal spheres''
discussion below).
From here on, we will simply refer to finite framed graphs as
graphs.

Graphs are allowed to intersect the boundary $\del M$
transversally; each point of intersection of $\Gamma$
with $\del M$ should be a vertex of $\Gamma$,
and they are the \emph{boundary vertices} of $\Gamma$.
Other vertices of $\Gamma$ are the \emph{interior vertices}.
Furthermore, the framing on $\Gamma$ induces at each
boundary vertex $b$ a ray in $T_b (\del M)$,
a \emph{framing} on $b$.
This makes the boundary $\del M$
an \emph{extended surface}, a surface
together with a configuration of finitely many
framed points.

For each interior vertex $v$,
the ``infinitesimal sphere'' at $v$ also acquires
an extended surface structure as follows.
The space of rays emanating from the origin in $T_v M$
is a sphere $S_v^2$, which we call the infinitesimal sphere.
An edge $e$ leaving $v$ has a tangent vector $\nu_e$ at $v$,
which gives us a point $\overline{\nu_e} \in S_v^2$.
The framing on $e$ at $v$ is a ray $\rho_v$ in $T_v M$;
the quarter plane spanned by $\nu_e$ and $\rho_v$ in $T_v M$
defines a ray in $T_{\overline{\nu_e}} S_v^2$,
i.e. a framing of $\overline{\nu_e}$.
The collection of such framed points $\overline{\nu_e}$ is the
extended surface structure that $S_v^2$ inherits from the graph
($\overline{\nu_e}$ are distinct by the extra condition
of non-tangency of edges at vertices).

Given an input premodular category $\cA$,
and given an extended sphere $S$ where
each marked point $p_i$ is colored with an object
$V_i \in \cA$,
the Reshetikhin-Turaev construction functorially
yields a vector space $\ZRT(S; V_1,\ldots, V_k)$
\ocite{rt}.
In particular, this vector space is (non-canonically)
isomorphic to $\eval{V_1,\ldots,V_k}$.

\begin{definition}
A coloring of a graph $\Gamma \subset M$
is the following data:
\begin{itemize}
\item Choice of an object $V(\ee) \in \Obj \cA$
  for each oriented edge $\ee \in E^{or}(\Gamma)$,
    so that $V(\ov{\ee}) = V(\ee)^*$.
\item Choice of a vector
  $\vphi(v) \in \ZRT(S_v^2; V(\ee_1),\ldots,V(\ee_n))$,
  for each interior vertex $v$,
  where $\ee_i$ are the edges incident to $v$,
  taken with outward orientation (pointing away from $v$).
\end{itemize}
\end{definition}

If $M$ has boundary, then we can color each boundary
vertex of $\Gamma$ with the color of the incident
edge (taken with outgoing orientation).
The pair $(B,\{V_b\})$ of the set of boundary vertices
with a coloring is the \emph{boundary value} of $\Gamma$.
We will denote
\[
  \Graph(M,\VV) =
    \text{set of all colored graphs
      in $M$ with boundary value } \VV
\]
and similarly consider formal linear combinations:
\[
\VGraph(M,\VV)=\{\text{formal linear combinations of graphs }
  \Gamma \in \Graph(M,\VV)\}
\]

It follows from result of Reshetikhin and Turaev that for every colored graph 
$\Ga$ in a ball $D\subset \R^3$, one can define its ``evaluation'' 
\[
  \eval{\Ga}_D\in \ZRT(\del D; V(\ee_1),\ldots,V(\ee_n))
    \cong \eval{V(\ee_1),\ldots,V(\ee_n)}
\]
where $\ee_1,\dots, \ee_n$ are the edges of $\Ga$ meeting the boundary
of $D$ (legs), taken with outgoing orientation; 
in particular, in the case when $\Ga$ is a star graph
in the unit ball in $\R^3$,
with one vertex at the center colored by
$\vphi \in \ZRT(S_v^2; V(\ee_1),\ldots,V(\ee_n))$,
then $\eval{\Ga}=\vphi$.
\footnote{
Here the identification
$\ZRT(S_v^2; V(\ee_i)) \cong \ZRT(\del D; V(\ee_i))$
is made using the natural maps 
$\del D \hookrightarrow \R^3 \backslash 0
\simeq T_0 \R^3 \backslash 0 \to S_v^2$.
}

We call a formal linear combination of colored graphs
$\Gamma = \sum c_i \Gamma_i \in \VGraph(M,\VV)$
a \emph{null graph} if there exists en embedded closed ball
$D \hookrightarrow M$ such that all $\Gamma_i$
meet $\del D$ transversally,
all $\Gamma_i$ coincide outside of $D$
as colored graphs,
and
\[
  \eval{\Gamma}_D = \sum c_i \eval{\Gamma_i}_D = 0
\]
(Note $D$ is allowed to touch the boundary $\del M$.)
We will say $\Gamma$ is null with respect to $D$.\\

We can now give the main definition of this section:

\begin{definition}\label{d:skein_3d}
For an oriented 3-manifold $M$ and boundary condition
$\VV=(B, \{V_b\})$ on $\del M$, 
we define the space of skeins by
\[
  \ZCY(M, \VV)=\VGraph(M,\VV)/N 
\]
where $N$ is the subspace spanned by all null graphs
(for all possible embedded disks).  
\end{definition}

We can now define the category of boundary conditions. 

\begin{definition}
Let $\Sigma$ be an oriented surface, possibly non-compact.
Suppose first $\Sigma$ has no boundary. Define $\hatZCY(\Sigma)$
as the category whose objects are finite subsets $B \subset \Sigma$, together
with a framing and coloring $V_b\in \Obj \cA$ for each point $b\in B$;
we will use the notation $\VV = (B,\{V_b\})$
for such an object (suppressing the framing),
and we call $B$ the set of marked points of $\VV$.
Define the morphisms in $\hatZCY(\Sigma)$ by
\[
  \Hom_{\ZSig{}}(\VV,\VV')
  = \ZCY(\Sigma \times [0,1]; \VV^*,\VV'),
  \qquad \VV=(B,\{V_b\}),
  \quad \VV'=(B',\{V_{b'}\})
\]
where $\VV^*, \VV'$ means the boundary condition obtained by putting
points $b\in B$ on the ``top'' $\Sigma \times\{1\}$, colored by objects
$V_b^*$ for outgoing legs (and thus colored by $V_b$ for incoming legs),
and putting points $b'\in B'$ on the ``bottom'' $N\times\{0\}$, colored by
objects $V_{b'}$ for outgoing legs.\\

$\ZSig{}$ is additive and $\kk$-linear.
We denote by
\begin{equation}
  \label{e:zcy}
\ZCY(\Sigma) = \Kar(\ZSig{})
\end{equation}
its Karoubi envelope.

For $N$ with boundary, we define $\hatZCY(N) = \hatZCY(N \backslash \del N)$,
$\ZCY(N) = \ZCY(N \backslash \del N)$.

\end{definition}

It is immediate from the definition that for a 2-disk $D^2$, $\ZCY(D^2) \simeq \cA$.

%% file: skein.tex
\section{Generalities of Skein Modules and Categories of Boundary Values}

In this section, we consider properties
of skein modules and categories of boundary values
that are common for both the Turaev-Viro theory
and Crane-Yetter theory.
Subsection~\ref{s:skein_modules}
is focused on the space
of relations (i.e. the null graphs $N \subset \VGraph(Y,\VV)$),
in particular how they are generated.
In subsection~\ref{s:cat_bval},
we exhibit a ``stacking'' monoidal structure
on the category of boundary values
of manifolds of the form $P\times (0,1)$,
and show it to be pivotal.

Throughout this section, $n = 1$ or 2.
We will use $Z,\hat{Z}$ to denote either
$\ZTV, \hatZTV$ (when $n=1$)
or $\ZCY, \hatZCY$ (when $n=2$),
so that $Z(n\text{-manifold})$ is a category,
and $Z((n+1)\text{-manifold}; \VV)$ is a vector space.
$\cA$ is spherical fusion for $n=1$,
and is premodular for $n=2$.
Denote by $I = (0,1)$, the \emph{open} interval.

\subsection{Skein Modules}
\label{s:skein_modules}

Recall that a null graph in $Y$ is null
with respect to some $(n+1)$ ball $D$,
and $D$ is allowed to touch the boundary $\del Y$.
In future applications,
it will be convenient to only consider balls
$D$ that do not meet $\del Y$,
such balls can be displaced by
ambient isotopy but balls meeting $\del Y$ may not.
Boundary vertices are univalent,
so graphs have simple behavior near the boundary.
If we exclude balls $D$ that meet $\del Y$,
the resulting space of null graphs $N'$
will be strictly smaller than $N$,
but not by much; the following lemma says
we just need to include
equivalence of graphs under ambient isotopy rel boundary:

\begin{lemma}
\label{lem:null_isotopy}
Let $Y$ be an $(n+1)$-manifold,
possibly with boundary or non-compact,
and let $\VV \in \Obj \hat{Z}(\del Y)$
be a fixed boundary value.
Define $N' \subset N \subset \VGraph(Y,\VV)$
to be the subspace generated by
graphs that are null with respect to a ball
that does not meet the boundary $\del Y$.
Define $N'' \subset \VGraph(Y,\VV)$
to be relations obtained by ambient isotopy,
i.e. generated by graphs $\Gamma^1 - \Gamma^0$,
where $\Gamma^t = \vphi^t(\Gamma)$,
$\vphi^t$ is a compactly-supported ambient isotopy fixing $\del Y$.
Then $N = N' + N''$.
\end{lemma}

\begin{proof}
It suffices to show that $N \subset N' + N''$.
Let $\Gamma = \sum c_i\Gamma_i$ be a null graph
with some boundary value $\VV$,
null with respect to a ball $D \subset Y$,
and suppose $D$ meets the boundary $\del Y$.
We would like to shrink $D$ to not meet $\del Y$
while maintaining that $\Gamma$ be null with respect to it.
Clearly if $D$ does not meet any point in $\VV$
then we can do this, and then $\Gamma \in N'$.

Suppose $D$ does contain some boundary vertex $b \in \VV$.
For each $i$, apply a small ambient isotopy $\vphi_i^t$ supported
in a small neighborhood of $b$ so that
the resulting graphs $\vphi_i^1(\Gamma_i)$
agree in a (possibly smaller) neighborhood of $b$.
\[
\begin{tikzpicture}
\draw (0,1) -- (2,1)
  node[pos=0.9,above] {$\del Y$};
\node[label={$b$}] (b) at (1,1) {};
\draw (0.6,1) to[out=180,in=60] (0,0.5);
\draw (1.4,1) to[out=0,in=120] (2,0.5);
\node at (0.2,0.5) {$D$};
\node[label={[shift={(0.4,-0.3)}] $\Gamma_1$}] (ga1) at (0.5,0.1) {};
\draw (1,1) to[out=-130,in=50] (ga1);
\node[label={[shift={(0.1,0)}] $\Gamma_2$}] (ga2) at (1.5,0) {};
\draw (1,1) to[out=-50,in=130] (ga2);
\end{tikzpicture}
\tikz \node at (0,0.5) {$\longrightarrow$};
\begin{tikzpicture}
\draw (0,1) -- (2,1)
  node[pos=0.9,above] {$\del Y$};
\node[label={$b$}] (b) at (1,1) {};
\draw (0.6,1) to[out=180,in=60] (0,0.5);
\draw (1.4,1) to[out=0,in=120] (2,0.5);
\draw (0.7,1) to[out=0,in=180] (1.05,0.9);
\draw (1.05,0.9) to[out=0,in=180] (1.3,1);
\node at (0.2,0.5) {$D$};
\node[label={[shift={(0.2,-0.3)}] \tiny $\vphi_1^1(\Gamma_1)$}] (ga1) at (0.5,0.1) {};
\draw (1,1) to[out=-50,in=50] (ga1);
\node[label={[shift={(0.35,-0.1)}] \tiny $\vphi_2^1(\Gamma_2)$}] (ga2) at (1.5,0) {};
\draw (1,1) to[out=-50,in=130] (ga2);
\end{tikzpicture}
\]
Then we can push $D$ slightly inwards away from the boundary at $b$,
and note that this new graph
$\Gamma' = \sum c_i \vphi_i^1(\Gamma_i)$
will be null with respect to the deformed $D$.
This reduces the number of points in $\VV$ that $D$ contains,
so after performing this finitely many times,
we are back to the case considered above where $D$
does not contain any boundary vertices.
Thus we see that repeated applications of isotopies
(i.e. relations in $N''$) takes $\Gamma$ to another graph
$\Gamma' \in N'$; in other words,
$\Gamma \in N'' + N'$.
\end{proof}

The following lemma says that isotopies can be broken into
a sequence of ``smaller" ones:
\begin{lemma}
\label{lem:isotopy_cover}
Let $\vphi^t$ be an isotopy of diffeomorphisms $\vphi^t\colon Y \to Y$
that is supported on some compact set $K$.
Let $\{U_i\}$ be a finite open cover.
Then there exists a sequence of isotopies $\vphi_j^t$
such that each $\vphi_j^t$ is supported on some
$U_{a_j} \cap K$,
and the isotopies concatenate to give a piecewise-smooth
isotopy from $\vphi^0$ to $\vphi^1$.
\end{lemma}
Proof can be found in \ocite{edwards-kirby}*{Corollary 1.3}.

In other words, given two diffeomorphisms $\vphi^0,\vphi^1$ that are isotopic,
there is another sequence of isotopies that takes $\vphi^0$ to $\vphi^1$
such that each is supported on a subset of $Y$.
One can make the new isotopies as close to the original isotopy as needed.

Finally, we show that the subspace of null graphs
are spanned by those that are null with respect to ``small" balls.
More precisely,

\begin{proposition}
\label{prp:null_cover}
Let $Y$ be an $(n+1)$-manifold,
possibly with boundary or non-compact.
Let $\{U_i\}$ be a finite open cover of $Y$.
Let $\VV \in \Obj \hat{Z}(\del Y)$
be a fixed boundary value.
Define $N_i \subset N \subset \VGraph(Y,\VV)$
to be the subspace of null graphs
in $Y$ with boundary value $\VV$
that are null with respect to some closed ball $D$
contained in $U_i$.
Then the space of null graphs is generated by $N_i$'s,
i.e.
\[
  N = \sum N_i
\]
\end{proposition}

\begin{proof}
Let $\Gamma = \sum c_j \Gamma_j \in N$
be a null graph.
By \lemref{lem:null_isotopy},
$\Gamma$ can be written as a sum of null graphs
$\Gamma' + \Gamma''$,
where $\Gamma' = \sum c_j' \Gamma_j'$
is a sum of graphs, each $\Gamma_j'$
is null with respect to some ball
not meeting $\del Y$,
and $\Gamma'' = \sum c_j'' \Gamma_j''$
is a sum of graphs,
each $\Gamma_j''$ is of the form
$(\Gamma_j'')^1 - (\Gamma_j'')^0$
for some smooth isotopy $(\Gamma_j'')^t$.\\

Consider one such $\Gamma_j''$,
and suppose that $\vphi^t:Y\to Y$
is an ambient isotopy
supported on a compact subset $K\subset Y$,
such that $(\Gamma_j'')^t  = \vphi^t(\Theta)$
for some graph $\Theta$.
By \lemref{lem:isotopy_cover},
there is a sequence of isotopies $\vphi_k^t$,
such that each $\vphi_k^t$ is supported on some
$U_{a_k} \cap K$, and the isotopies
concatenate to give a piecewise-smooth
isotopy from $\vphi^0$ to $\vphi^1$.
Then
$\Gamma_j'' = \vphi^1(\Theta) - \vphi^0(\Theta)
  = \sum_k \vphi_k^1(\Theta) - \vphi_k^0(\Theta)
  \in \sum N_i$.
Thus, in the sum $\Gamma = \Gamma' + \Gamma''$,
we have $\Gamma'' \in \sum N_i$.\\

Now consider a term $\Gamma_j'$ in $\Gamma'$,
and suppose it is null with respect to some ball
$D$ not meeting $\del Y$.
There exists an ambient isotopy of identity
$\vphi^t: Y \to Y$ that moves $D$
into some open set $U_a$.
Then $\vphi^1(\Gamma_j') \in N_a$.
But by the same argument as above,
$\vphi^1(\Gamma_j') - \Gamma_j' \in \sum N_i$.
Hence, we conclude that
$\Gamma' = \sum c_j\Gamma_j' \in \sum N_i$,
and we are done.
\end{proof}

\subsection{Categories of Boundary Values}
\label{s:cat_bval}

\begin{lemma}
Let $X_1,X_2$ be $n$-manifolds without boundary,
possibly non-compact.
Let $\vphi: X_1 \to X_2$
be an orientation-preserving embedding.
Then $\vphi$ induces an obvious inclusion functor
\[
  \vphi_*: \hZ(X_1) \to \hZ(X_2)
\]
that sends objects to their image under $\vphi$,
and sends morphisms to their image under $\vphi \times \id_I$.
This descends to the Karoubi envelopes
\[
  \vphi_*: Z(X_1) \to Z(X_2)
\]
Furthermore, an isotopy $\vphi^t: X_1 \to X_2$
induces a natural isomorphism from $\vphi_*^0$ to $\vphi_*^1$,
and isotopic isotopies induce the same natural isomorphisms.
\end{lemma}

\begin{proof}
Clear.
\end{proof}

\begin{lemma}
\label{lem:disjoint_union}
Under the same hypothesis above,
\begin{align*}
  \hZ(X_1 \sqcup X_2) \simeq \hZ(X_1) \boxtimes \hZ(X_2) \\
  Z(X_1 \sqcup X_2) \simeq Z(X_1) \boxtimes Z(X_2)
\end{align*}
\end{lemma}

\begin{proof}
The proof for $\hZ$ is clear:
the inclusions of $X_1$ and $X_2$ into $X_1\sqcup X_2$
together induce
$\hZ(X_1) \boxtimes \hZ(X_2) \to \hZ(X_1 \sqcup X_2)$,
and this is easily seen to be an isomorphism of categories.
The equivalence for $Z$ then follows by universal property,
and the fact that the Deligne-Kelly tensor product
of two finite semisimple abelian categories
is also a finite semisimple  abelian category.
\end{proof}

Finally we discuss the ``stacking" monoidal structure
of some special $n$-manifolds.
Let $P$ be a $(n-1)$-manifold without boundary,
possibly disconnected (with finitely many components) or non-compact.
For $n=1$, $P$ is just a collection of points.
For $n=2$, $P$ is a collection of open intervals and circles.\\

Let $I = (0,1)$,
and let $m : I \sqcup I \to I$
be $x/2$ on the first $I$ and $(x+1)/2$
on the second $I$.
This is part of an $A_\infty$-space structure,
as defined in \ocite{stasheff}:
$m$ is not associative,
but there is a ``straight line" isotopy
$m_3^t: I \sqcup I \sqcup I \to I$
from $m_3^0 = m \circ (m \sqcup \id_I)$
to $m_3^1 = m \circ (\id_I \sqcup m)$,
relating two extreme ways of including three intervals
into one.\\

Let
\begin{align*}
  \tilde{m} &: P \times I \sqcup P \times I
    = P \times (I \sqcup I)
    \to P \times I \\
  \tilde{m}_3^t &: P \times I \sqcup P \times I \sqcup P \times I
    = P \times (I \sqcup I \sqcup I) \to P \times I
\end{align*}

\begin{proposition}
\label{prp:stacking}
There is a monoidal structure on $\hZ(P\times I)$
given as follows:
\begin{itemize}
  \item The tensor product is
  \[
    \tnsr := \tilde{m}_* : \hZ(P\times I) \boxtimes \hZ(P\times I)
      \to \hZ(P\times I)
  \]
  \item The unit $\one$ is the empty configuration.
    (Left, right unit constraints are given in proof.)
  \item The associativity constraint $\alpha$ is the natural isomorphism
    that is induced by $\tilde{m}_3^t$.
\end{itemize}
Similarly, there is a monoidal structure on $Z(P\times I)$.
\end{proposition}

\begin{proof}
Left unit constraint $l_A : A \tnsr \one \to A$
is given by a ``straight line" graph,
likewise for right unit constraint.
That $\alpha$ satisfies the pentagon relations
follows from the fact that any two inclusions
$I^{\sqcup 4} \hookrightarrow I$
are isotopic, and any two isotopies are themselves isotopic.
The result for $Z(P\times I)$ follows from universal property.
\end{proof}

\begin{proposition}
\label{p:stack_pivotal}
The monoidal structure on $\hZ(P\times I)$ and $Z(P\times I)$
given in \prpref{prp:stacking}
is pivotal.
\end{proposition}

\begin{remark}
The input category $\cA$ has to be spherical,
but the resulting categories $Z(P\times I)$ may not be;
in future work, we will show that $Z(S^1 \times I)$ is
pivotal but not spherical.
\end{remark}

\begin{proof}
It suffices to prove this for $\hZ(P\times I)$,
since its Karoubi envelope will inherit the pivotal structure.

The rigid and pivotal structures come from topological constructions.
Denote by $\theta: P \times I \to P \times I$
be the orientation-reversing diffeomorphism
which flips $I$,
i.e. $(p,x) \mapsto (p, 1-x)$.
Denote by $\Theta: P \times I \times [0,1] \to P \times I \times [0,1]$
the orientation-preserving diffeomorphism
that rotates the $I \times [0,1]$ rectangle by $180^\circ$,
i.e. $(p,x,t) \mapsto (p,1-x,1-t)$.

Denote by $\upsilon$ the map that takes $P \times I \times [0,1]$,
squeezes it in half along the $I$ direction,
bends it like an accordion so that the left side collapses,
and puts it back in $P \times I \times [0,1]$
so that the top and bottom are now attached to the top
(see \firef{f:ups}).

\begin{figure}[ht]
$\upsilon,\upsilon',\eta, \eta':$
\begin{tikzpicture}
\begin{scope}[shift={(0,-0.5)}]
\draw (0,0) -- (1,0) node[pos=0, below] {\tiny $A$} node[below] {\tiny $B$};
\draw (0,1) -- (1,1) node[pos=0, above] {\tiny $C$} node[above] {\tiny $D$};
\draw[dotted] (0,0) -- (0,1);
\draw[dotted] (1,0) -- (1,1);
\end{scope}
\end{tikzpicture}
$\to$
\begin{tikzpicture}
\begin{scope}[shift={(0,-0.5)}]
\draw[dotted] (0,1) arc (180:360:0.5cm);
\draw (0,1) -- (1,1)
  node[pos=0, above] {\tiny $B$}
  node[pos=0.5,above] {\tiny $A/C$}
  node[above] {\tiny $D$};
\draw[white] (0.5,0.95) -- (0.5,1.05);
\end{scope}
\end{tikzpicture}
,
\begin{tikzpicture}
\begin{scope}[shift={(0,-0.5)}]
\draw[dotted] (0,1) arc (180:360:0.5cm);
\draw (0,1) -- (1,1)
  node[pos=0, above] {\tiny $C$}
  node[pos=0.5,above] {\tiny $D/B$}
  node[above] {\tiny $A$};
\draw[white] (0.5,0.95) -- (0.5,1.05);
\end{scope}
\end{tikzpicture}
,
\begin{tikzpicture}
\begin{scope}[shift={(0,-0.5)}]
\draw[dotted] (1,0) arc (0:180:0.5cm);
\draw (0,0) -- (1,0)
  node[pos=0,below] {\tiny $A$}
  node[pos=0.5,below] {\tiny $B/D$}
  node[below] {\tiny $C$};
  \draw[white] (0.5,-0.05) -- (0.5,0.05);
\end{scope}
\end{tikzpicture}
,
\begin{tikzpicture}
\begin{scope}[shift={(0,-0.5)}]
\draw[dotted] (1,0) arc (0:180:0.5cm);
\draw (0,0) -- (1,0)
  node[pos=0,below] {\tiny $D$}
  node[pos=0.5,below] {\tiny $C/A$}
  node[below] {\tiny $B$};
\draw[white] (0.5,-0.05) -- (0.5,0.05);
\end{scope}
\end{tikzpicture}
\caption{The maps $\upsilon,\upsilon',\eta,\eta'$ for $P = \{*\}$}
\label{f:ups}
\end{figure}
$\upsilon', \eta,\eta'$ are defined similarly.

Let $\VV = (B, \{V_b\}) \in \Obj \hZ(P \times I)$.
Its left dual $\VV^*$ is given by
$(\theta(B), \{V_b^*\})$,
that is, apply the flipping diffeomorphism $\theta$
defined above to the marked points,
and label them by the left duals of the original labeling.
Similarly, the right dual is
$\rdual \VV = (\theta(B), \{\rdual V_b\})$.
(It is not too important to distinguish $V_b^*$ from $\rdual V_b$
since $\cA$ itself is pivotal.)

The left evaluation and coevaluation morphisms for $\VV$ are obtained by
applying $\upsilon$ and $\eta$ to $\id_\VV$, respectively.
Similarly, the right evaluation and coevaluation morphisms for $\VV$
are obtained by applying $\upsilon'$ and $\eta'$ to $\id_\VV$, respectively.
It is easy to see that these morphisms have the required properties.

Given a morphism
$f \in \Hom_{\hZ(P\times I)}(\VV, \VV')$
represented by a graph $\Gamma$,
it is easy to check that
its left and right duals are given by applying
the rotation $\Theta$ to $\Gamma$,
and keeping all orientations and labels of
the edges of $\Gamma$.

The pivotal structure is essentially the identity morphism,
but with one vertex on each vertical line labeled by $\delta$,
the pivotal structure of $\cA$.
\end{proof}

\begin{example}
  \label{xmp:tv_circle}
We pointed out at the end of \secref{s:TV} that
$\ZTV(I) \simeq \cA$. Giving $\ZTV(I)$ the stacking monoidal structure
above, we see that this equivalence is a tensor equivalence
respecting the pivotal structure.
\end{example}

\begin{example}
  \label{xmp:cy_annulus}
Similarly, we had $\ZCY(I \times I) \simeq \cA$.
$I \times I$ can stack in two ways, along the first copy of $I$
(horizontal stacking) or the second (vertical stacking).
They both give monoidal structures equivalent to $\cA$'s.
\end{example}

\begin{proposition}
The $E_1$-algebra structure of $Z(P\times I)$ is unique
in the sense of \thmref{t:facthom-characterize};
that is, any automorphism of $P \times I$
induces an $E_1$-algebra self-equivalence on $Z(P\times I)$.
\label{p:E1-uniqueness}
\end{proposition}

\begin{proof}
It is not hard to see that it suffices to consider $P$ connected.
For $P= *, I, S^1$, the space of self-diffeomorphisms of $P$ is connected,
so any fiber-preserving automorphism of $P\times I$
is isotopic to the identity,
so the induced functor of the automorphism
is an equivalence of $E_1$-algebras.
\end{proof}

Next we consider (left) module categories over $Z(P\times I)$.
First, $I$ is a left module over the $A_\infty$ space $I$,
as follows. Let $f : I \to (1/2, 1) \subset I$ be some inclusion
that is identity near 1.
The embedding $n = (\cdot/2) \sqcup f : I \sqcup I \to I$
gives left multiplication,
and it is associative up to some isotopy,
that is, the two inclusions
$n_3^0 := n \circ (\id_I \sqcup n)$ and $n_3^1 := n \circ (m \sqcup \id_I)$
are isotopic via some isotopy $n_3^t$.
It is not hard to see that any two such left module structures are
equivalent.\\

Now let $X$ be a collared $n$-manifold,
i.e. we have an embedding $P \times I \hookrightarrow X$,
where the 0 end in $I$ escapes to infinity in $X$.
Crossing with $P$, we can upgrade the above left module structure
on $I$ to $X$,
obtaining a left multiplication
$\tilde{n} : P \times I \sqcup X \to X$
and an isotopy $\tilde{n}_3^t$ from
$\tilde{n} \circ (\id_{P\times I} \sqcup \tilde{n})$
to $\tilde{n} \circ (\tilde{m} \sqcup \id_X)$.
(See \lemref{l:skein_alg}.)

\begin{proposition}
\label{p:collared_module}
Given a collared $n$-manifold $X$,
there is a left $\hZ(P \times I)$-module category structure on $\hZ(X)$
given by
\[
  \lact \, := \tilde{n}_* : \hZ(P\times I) \boxtimes \hZ(X)
    \to \hZ(X)
\]
and the associativity constraint is given by the
natural isomorphism induced by the isotopy $\tilde{n}_3^t$.
Such a structure is unique up to equivalence.\\

Similarly there is a left $Z(P\times I)$-module category structure on $Z(X)$.
\end{proposition}

\begin{proof}
Similar to \prpref{prp:stacking}.
\end{proof}

There is a similar story for right module structure,
where $X$ is a collared $n$-manifold so that 1 escapes to infinity.

%% file: excision.tex
\section{Excision for $\ZTV, \ZCY$}

In this section, we prove the main result of the paper,
that $\ZTV$ and $\ZCY$ satisfy excision.
As in the previous section,
essentially the same proof works for
both the Turaev-Viro and Crane-Yetter theory,
so we adopt the same notation as before,
namely $Z, \hZ$ stands for either of the theories.\\

Let $X$ be an $n$-manifold without boundary,
with finitely many components,
possibly non-compact.
To present $X$ as the quotient of some $n$-manifold $X'$
by some gluing,
consider a smooth function
$f: X \to S^1 = \RR / 2\ZZ$,
together with a trivialization of $P$-bundles
$P \times I \simeq f^\inv(I)$
for some $(n-1)$-manifold $P$.
Take $X'$ to be the ``preimage of $(0,3)$ under $f$";
more precisely, pullback $f$ along the universal
covering map $\RR \to \RR/2\ZZ$ to get
$\tilde{f}: \tilde{X} \to \RR$,
and take $X' = \tilde{f}^\inv((0,3))$
(see figure below).
So $X$ is obtained from $X'$ by gluing the parts over
$(0,1)$ and $(2,3)$.

\[
\input{fig_Xprime.tikz}
\]

\begin{remark}
\label{rmk:gluing}
Excision is usually phrased in terms of gluing
two collared manifolds.
In the above language,
that will correspond to the case when
$X' = X_1 \sqcup X_2$,
where $X_1 = \tilde{f}^\inv((0,1.5))$,
$X_2 = \tilde{f}^\inv((1.5,3))$,
so that the pullback map $X' \to X$
is the gluing/overlapping of $X_1$ and $X_2$
over $(0,1)$, the collared neighborhoods.
\end{remark}

Since $\tilde{f}^\inv((0,1)) \simeq \tilde{f}^\inv((2,3))
\simeq f^\inv(I)$ naturally,
the trivialization $\PI \simeq f^\inv(I)$
gives a left and right $\PI$-module structure on $X'$,
and makes $\hZ(X')$ a $\hZ(\PI)$-bimodule category
(likewise for $Z$).\\

The natural gluing map $X' \to X$
is the composition $X' \subset \tilde{X} \to X$.
We can also embed $X'$ in $X$ as follows:
consider a following sequence of maps
$X' \to \PI \sqcup X' \sqcup \PI \to X' \to X$;
the first map is just the obvious inclusion,
the second is the left and right module maps
``squeezing" $X'$ into itself,
and the third map is the natural quotient map.
It is easy to see that the composition
is an embedding, in fact a diffeomorphism
onto $X \backslash f^\inv(0.5)$.
We denote this composition by $i$.\\

Since $i:X' \to X$ is an embedding,
it induces a functor
$i_* : \hZ(X') \to \hZ(X)$.
Recall that there is a natural functor
$\htr: \hZ(X') \to \htr(\hZ(X'))$
that is identity on objects.

\begin{lemma}
The inclusion functor $i_*: \hZ(X') \to \hZ(X)$
extends along $\htr$ to a functor
$i_*: \htr(\hZ(X')) \to \hZ(X)$.
\end{lemma}
\begin{proof}
Consider a map $\Psi: X' \times \clI \to X \times \clI$
described as a composition of operations
given by the figures below (with further explanations later):
\\
\input{fig_gluing_2.tikz}
\\
The first figure depicts $X' \times \clI$,
with the $\clI$ factor going in the vertical direction.
The foliation depicted consists of the obvious horizontal leaves
$X' \times \{r\}$;
we depict the foliation only to better explain the operations
we perform below.
The left and right parts are the $P\times I$ portions
that would glue to give $X$.
Operation (1) ``pinches'' the left vertical side down
and the right vertical side up.
Operation (2) ``squeezes'' the bottom left and top right portions.
Operation (3) glues the two vertical sides.

Below we depict a graph $\Gamma$ in $X' \times \clI$
with incoming boundary value $A \lact M$ and
outgoing boundary value $N \ract A$,
representing an element of $\Hom_{\htr(\hZ(X'))}(M,N)$.
It is sent to a graph $\Psi(\Gamma)$ in $X \times \clI$
with incoming boundary $M$ and outgoing boundary $N$,
representing an element of $\Hom_{\hZ(X)}(M,N)$.

\input{fig_gluing_1.tikz}

Note that operation (1) creates corners in the top left and bottom right,
so $\Psi$ is not exactly a smooth map;
however, it is an embedding when restricted to $X' \times (0,1)$,
and can easily be slightly perturbed to be a smooth embedding.
As we ultimately care about the images of graphs $\Psi(\Gamma)$
up to isotopy, we will not bother with the details of this perturbation
nor the non-smoothness of $\Psi$ at the corner.

The only points in $X \times \clI$ that are hit more than once
are in $f^\inv(0.5)$;
we call this the \emph{seam}.
In the figure above, the seam is depicted as the vertical dotted line
in the right most figure.
The seams is also the image of the top left and bottom right boundary pieces
(the parts labeled $A$).
The image of $\Psi|_{X' \times (0,1)}$ is
exactly $X \times (0,1) \backslash $seam.

We claim that the following map is well-defined:
\begin{align*}
\Hom_{\htr(\hZ(X'))}(M,N) &\to \Hom_{\hZ(X)}(i_*(M),i_*(N))
\\
\Gamma &\mapsto \Psi(\Gamma)
\end{align*}
It is not hard to see that the assignment
$\Gamma \mapsto \Psi(\Gamma)$
yields a well-defined map
$\ihom{\hZ(X')}{A}(M,N) \to \Hom_{\hZ(X)}(i_*(M),i_*(N))$;
a graph $\Gamma=\sum c_i\Gamma_i$
that is null with respect to some ball $D$
would have image $\Psi(\Gamma)$ null with respect to $\Psi(D)$.
We need to check that the relations $\sim$ in
$\Hom_{\htr(\hZ(X'))}(M,N) = \bigoplus \ihom{\hZ(X')}{A}(M,N)/\sim$
are satisfied.
Recall that relations are generated by
$\Theta \circ (\psi \lact \id_M) - (\id_N \ract \psi) \circ \Theta$,
where $\Theta \in \ihom{\hZ(X')}{A,B}(M,N)$
and $\psi \in \Hom_{\hZ(P\times I)}(B,A)$.
We see that
\[
\begin{tikzpicture}
\begin{scope}[shift={(0,-0.5)}]
\node at (1.5,0.5) {$\Theta$};
\draw (0,1) -- (1,1)
  node[pos=0.35,above] {$B$};
\node[dotnode] at (0.7,1) {};
\draw (0,1) to[out=-80,in=-100] (0.7,1);
\node at (0.35,0.9) {\tiny $\psi$};
\draw[regular] (1,1) -- (2,1)
  node[pos=0.5,above] {$M$};
\draw (2,1) -- (3,1);
\draw (0,0) -- (1,0);
\draw[regular] (1,0) -- (2,0)
  node[pos=0.5,below] {$N$};
\draw (2,0) -- (3,0)
  node[pos=0.65,below] {$B$};
\node[dotnode] at (2.3,0) {};
\end{scope}
\end{tikzpicture}
\; - \;
\begin{tikzpicture}
\begin{scope}[shift={(0,-0.5)}]
\node at (1.5,0.5) {$\Theta$};
\draw (0,1) -- (1,1)
  node[pos=0.35,above] {$B$};
\node[dotnode] at (0.7,1) {};
\draw[regular] (1,1) -- (2,1)
  node[pos=0.5,above] {$M$};
\draw (2,1) -- (3,1);
\draw (0,0) -- (1,0);
\draw[regular] (1,0) -- (2,0)
  node[pos=0.5,below] {$N$};
\draw (2,0) -- (3,0)
  node[pos=0.65,below] {$B$};
\node[dotnode] at (2.3,0) {};
\draw (2.3,0) to[out=80,in=100] (3,0);
\node at (2.65,0.1) {\tiny $\psi$};
\end{scope}
\end{tikzpicture}
\; \mapsto \;
\begin{tikzpicture}
\begin{scope} [shift={(0,-0.5)}]
\draw[regular] (0,1) -- (2,1);
\draw (0.5,1) -- (1.5,1);
\draw[regular] (0,0) -- (2,0);
\draw (0.5,0) -- (1.5,0);
\draw[dotted] (1,0) -- (1,1);
\draw[dotted] (1,0) to[out=135,in=-135] (1,1);
\node at (0.9,0.5) {\tiny $\psi$};
\end{scope}
\end{tikzpicture}
\; - \;
\begin{tikzpicture}
\begin{scope} [shift={(0,-0.5)}]
\draw[regular] (0,1) -- (2,1);
\draw (0.5,1) -- (1.5,1);
\draw[regular] (0,0) -- (2,0);
\draw (0.5,0) -- (1.5,0);
\draw[dotted] (1,0) -- (1,1);
\draw[dotted] (1,0) to[out=45,in=-45] (1,1);
\node at (1.1,0.5) {\tiny $\psi$};
\end{scope}
\end{tikzpicture}
= 0
\]
We leave checking that composition is respected as
a simple exercise.
\end{proof}

We want to show that $i_*$ is an equivalence,
and will be considering $\Psi^\inv$ applied to graphs.
It is not clear that this is well-defined,
e.g. moving parts of a graph in $X\times \clI$ across the seam
could result in different graphs with
different boundary conditions in $X' \times \clI$.
However, the relation
$\Theta \circ (\psi \lact \id_M) - (\id_N \ract \psi) \circ \Theta$
essentially takes care of this ambiguity.

Let us make this precise.
Consider a small neighborhood
$P \times (0.5-\veps,0.5+\veps) \times \clI$
of the seam in $X \times \clI$.
Consider the following vector field $\nu$:
at $(p,x,t) \in P \times (0.5-\veps,0.5+\veps) \times \clI$,
the vector field has value $\sigma(x)\sin(\pi t) \frac{\del}{\del x}$,
where $\sigma(x)$ is a smooth non-negative cut-off function on $(0,1)$
that has support exactly $(0.5-\veps,0.5+\veps)$.
This vector field $\nu$ has the following displacing property:
for any compact subset $K$ in $P \times (0,5-\veps,0.5+\veps) \times (0,1)$
(i.e. near the seam and not touching the boundary),
the flow eventually pushes $K$ off of the seam,
i.e. there is some $\alpha$ such that
the flow under $\nu$ after time $\alpha$
does not intersect the seam.

Let $\zeta^\alpha$ be the isotopy generated by $\nu$.
Denoting by $L_0$ the seam,
we define $L_\alpha = \zeta^\alpha(L_0)$.
Let $\Psi_\alpha$ be the composition
$\zeta^\alpha \circ \Psi$.
Then $L_\alpha$ is the ``seam" for $\Psi_\alpha$.

Suppose a graph $\Gamma$ in $X\times \clI$
intersects the seam $L_0$ transversally,
in that the edges meet $L_0$ transversally
and no vertices are on $L_0$.
Then $\Gamma$ defines a boundary value at the seam:
the marked points are the points of intersection,
and coloring is the color associated to the edge
taken with right-ward orientation
(that is, in direction of $\nu$).
In particular, the boundary value of
$i_*(\Gamma)$ in the figure above is $A$.
If $\Gamma$ intersects $L_\alpha$ transversally,
then we can also define its boundary value
at $L_\alpha$ similarly; to be precise,
it is the boundary value of $\zeta^{-\alpha}(\Gamma)$
at the seam.

\begin{lemma}
\label{lem:alpha_independence}
Let $\Gamma$ be a graph in $X \times \clI$
that represents a morphism in
$\Hom_{\hZ(X)}(i_*(M),i_*(N))$
for some $M,N \in \Obj \htr(\hZ(X'))$.
Choose some $\alpha$ such that
$\Gamma$ is transverse to $L_\alpha$,
and suppose it defines the boundary value $A_\alpha$.
We see that $\Psi_\alpha^\inv(\Gamma)$
is a graph in $X' \times \clI$ representing a morphism in
$\Hom_{\hZ(X')}(A_\alpha \lact M, N \ract A_\alpha)$.
Then as a morphism in $\Hom_{\htr(\hZ(X'))}(M,N)$,
$\Psi_\alpha^\inv(\Gamma)$ is independent of
such a choice of $\alpha$.
\end{lemma}

\begin{proof}
Clear from the picture.
\end{proof}

We come to the main ``topological" result of the paper:
\begin{theorem}
  \label{t:hat_excision}
The extension $i_*: \htr(\hZ(X')) \to \hZ(X)$
is an equivalence.
\end{theorem}

\begin{proof}
It was already evident from the object map that
$\htr(\hZ(X')) \to \hZ(X)$ is essentially surjective
- it only misses objects that have marked points on $f^\inv(0.5)$,
but such an object is isomorphic to an object
with points moved slightly off of $f^\inv(0.5)$.\\

To show that $i_*$ is fully faithful,
fix objects $M,N \in \htr(\hZ(X'))$.
By \lemref{lem:alpha_independence},
the family $\Psi_\alpha^\inv$ of maps
defines a map
$\Phi : \VGraph(X\times \clI; i_*(M)^*,i_*(N))
\to \Hom_{\htr(\hZ(X'))}(M,N)$.\\

Let us show that $\Phi$ factors through the projection
$\VGraph(X\times \clI; i_*(M)^*,i_*(N)) \to \Hom_{\hZ(X)}(i_*(M),i_*(N))$.
We make the following observation:
If $\Gamma = \sum c_i\Gamma_i$ is null with respect to some
closed ball $D \subset X \times \clI$,
and there is some $L_\alpha$ that does not meet $D$
and is transversal to $\Gamma$,
then $\Phi(\Gamma) = \Psi_\alpha^\inv(\Gamma)$
is null with respect to $\Psi_\alpha^\inv(D)$.\\

Let $0 < \beta < 0.5$ be such that
$i_*(M)$ and $i_*(N)$ do not have any marked points
in $f^\inv((0.5-\beta,0.5+\beta)) \subset X$;
denote $J = (0.5-\beta, 0.5+\beta)$.
Consider the open cover $\{U_1,U_2\}$ of $X \times \clI$,
where $U_1 = f^\inv(J)$
and $U_2 = X\times \clI \backslash L_0$.
By \prpref{prp:null_cover},
the space of null graphs is generated by graphs that
are null with respect to balls $D$ contained
in either $U_1$ or $U_2$, thus
it suffices to show that $\Phi$ sends such graphs to 0.
By the previous observation, it suffices to check that
there exists an $\alpha$ that doesn't intersect such $D$.

For $D \subset U_2$,
such $L_\alpha$ exists by Sard's theorem
- for small enough $\alpha$, $L_\alpha$ does not
intersect $D$, so it suffices to consider
transversality with $\Gamma$, which is a generic condition.\\

Now suppose $D \subset U_1$.
Since there are no marked points on the boundary in $U_1$,
by \lemref{lem:null_isotopy},
we may assume that $D$ does not meet the boundary.
As we pointed out, the vector field $\nu$ defining the isotopy
$\zeta^\alpha$ has the property that it will displace
$D$ off of $L_0$.
So if $\zeta^\alpha(D)$ does not intersect $L_0$,
we can take $L_{-\alpha + \veps}$,
where small $\veps$ is chosen to get transversality with $\Gamma$,
and we are done.
\end{proof}

Combining the topological result above with
the algebraic results of \secref{s:module},
we have the main result of the paper:
\begin{theorem}
\label{t:excision-skein}
There is an equivalence
\[
  \cZ_{Z(P\times I)}(Z(X')) \simeq Z(X)
\]
In particular, when $X = X_1 \cup X_2$
as in \rmkref{rmk:gluing},
\[
  Z(X_1) \boxtimes_{Z(P\times I)} Z(X_2) 
  \simeq \cZ_{Z(P\times I)}(Z(X_1 \sqcup X_2))\simeq Z(X)
\]
\end{theorem}

\begin{proof}
We claim that $Z(P\times I)$ is multifusion; we justify this claim later.
By \prpref{p:stack_pivotal}, $\hZ(P\times I)$ is pivotal.
In reference to the notation in \secref{s:module},
take $\cC' = \hZ(P\times I), \cC = Z(P\times I),
\cM' = \hZ(X'), \cM = Z(X')$.
Then we have
\begin{align*}
\cZ_{Z(P\times I)}(Z(X'))
&\simeq \Kar(\htr_{\hZ(P\times I)}(Z(X')))
  & \text{ by \corref{cor:center},}\\
&\simeq \Kar(\htr_{\hZ(P\times I)}(\hZ(X')))
  & \text{ by \lemref{lem:M_Karubi},}\\
&\simeq \Kar(\hZ(X))
  & \text{ by extending $i_*$ from \thmref{t:hat_excision} to Kar} \\
&= Z(X)
\end{align*}

The second statement follows from the first
by applying \lemref{lem:disjoint_union} and \eqnref{e:center-product}.

Now we need to justify $Z(P\times I)$ being multifusion.
This is true for $P = \{*\}$ and for $P = I$.
By \xmpref{xmp:annulus}, which uses the argument above for $P = I$,
we have $Z(S^1 \times I) \simeq \cZ(\cA)$ as a $\kk$-linear abelian category
(but not as a monoidal category, see \rmkref{rmk:ZA_stacking}); in particular,
this implies that it is semisimple with finitely many simple objects.  Since by
\prpref{p:stack_pivotal}, the stacking monoidal structure on $Z(P\times I)$ is
rigid and pivotal, this shows that $Z(P\times I)$ is a pivotal  multifusion
category.

So $Z(P\times I)$ is pivotal multifusion for any connected $P$;
the claim follows for a disjoint union of finitely many such $P$'s.
\end{proof}

\begin{corollary}
\label{c:skein_facthom}
In each of the two cases below 
\begin{itemize} 
\item  $n=1$, $\cA$ a spherical fusion category 
\item $n=2$, $\cA$ a premodular category
\end{itemize}
for an $n$-manifold $X$ the category $Z(X)$ of boundary 
values for colored graphs constructed above coincides with the factorization homology $\int_X \cA$. 
\end{corollary}
\begin{proof}
We verify that $Z(-)$ satisfies the three characterizing properties
laid out in \thmref{t:facthom-characterize}.

\begin{enumerate}
\item For both cases of $n$, $\cA$ defines an $n$-disk algebra in $\cV=\Rex$,
	thus defining factorization homologies $\int_- \cA$
	which coincide with $Z(-)$ on the $n$-disk.

\item \prpref{p:E1-uniqueness} proves the uniqueness of the $E_1$-algebra
	structure	on $Z(P\times I)$.

\item \thmref{t:excision-skein} proves the excision property (as in
	\thmref{t:excision}.

\end{enumerate}
Thus, $Z(X) \simeq \int_X \cA$.
\end{proof}

\begin{corollary} \label{c:skein_ss}In the assumptions of \corref{c:skein_facthom}, 
$Z(X)$ is a finite semisimple category. 
\end{corollary}
\begin{proof}
Any connected $X$ can be built from $I^n$
by a sequence of gluings of collared manifolds.
For example, for $n=2$,
gluing opposite edges of a square gives an annulus,
and gluing boundaries of the annulus together
gives the torus. 

By \thmref{t:excision-skein}, the corresponding category $Z(X)$ thus can be
constructed from the  Deligne product of several copies of  $Z(I^n) \simeq \cA$
by repeatedly applying the center construction, replacing a category $\cM$ by
$\cZ_{Z(P\times I)}(\cM)$. Since it was shown in the proof of
\thmref{t:excision-skein} that $Z(P\times I)$ is pivotal multifusion, it now follows
from \prpref{p:ZM_ss} that applying the center construction always gives a
finite semisimple category. Thus, $Z(X)$ is a finite semisimple category.

\end{proof}

%% file: fig_Xprime.tikz
\begin{tikzpicture}
\draw (1.5,1) circle (0.4cm);
\draw (3.5,1) circle (0.4cm);
\draw[line width=1mm, white] (-0.4,1) -- (4.4,1);
\draw (-0.4,1.05) -- (4.4,1.05);
\draw (-0.4,0.95) -- (4.4,0.95);
\draw[fill, white] (1.5,1) circle (0.37cm);
\draw[fill, white] (3.5,1) circle (0.37cm);
\draw[decorate,decoration={brace,amplitude=10pt,raise=4pt}]
 (0,1.5) -- (3,1.5)
 node[pos=0.5, yshift=0.8cm] {$X'$};
\draw[dotted] (0,1.5) -- (0,-1);
\draw[dotted] (3,1.5) -- (3,-1);
\draw[decorate,decoration={brace,amplitude=3pt,raise=2pt}]
 (0,1.1) -- (1,1.1)
 node[pos=0.5, yshift=0.4cm] {\small $P \times I$};
\draw[decorate,decoration={brace,amplitude=3pt,raise=2pt}]
 (2,1.1) -- (3,1.1)
 node[pos=0.5, yshift=0.4cm] {\small $P \times I$};

\draw[->] (1.5,0.4) -- (1.5,-0.6) node[pos=0.5,right] {$\tilde{f}$};
\node at (-1.5,1) {$\tilde{X}$};
\node at (-1.5,-1) {$\RR$};

\draw (-0.4,-1) -- (4.4,-1);
\foreach \x in {0,...,4} {
 \node at (\x,-1.3) {\x};
 \draw (\x,-0.9) -- (\x,-1.1);
}
\node at (0.5,-0.7) {$I$};

\end{tikzpicture}
\hspace{5pt}
\begin{tikzpicture}
\node at (0,1) {$\to$};
\node at (0,-1) {$\to$};
\end{tikzpicture}
\hspace{5pt}
\begin{tikzpicture}

\draw (0.5,1) circle (0.4cm);
\draw[line width=1mm, white] (-0.5,1) -- (1.5,1);
\draw (-0.5,1.05) -- (1.5,1.05);
\draw[dotted] (-0.7,1.05) -- (1.7,1.05);
\draw (-0.5,0.95) -- (1.5,0.95);
\draw[dotted] (-0.7,0.95) -- (1.7,0.95);
\draw[fill, white] (0.5,1) circle (0.37cm);

\draw[->] (0.5,0.4) -- (0.5,-0.6) node[pos=0.5,right] {$f$};
\node at (2.5,1) {$X$};
\node at (2.5,-1) {$\RR/2\ZZ$};

\draw (-0.5,-1) -- (1.5,-1);
\draw[dotted] (-0.7,-1) -- (1.7,-1);
\node at (0,-1.3) {1};
\node at (1,-1.3) {2/0};
\foreach \x in {0,...,1} {
 \draw (\x,-0.9) -- (\x,-1.1);
}
\end{tikzpicture}

%% file: fig_gluing_2.tikz
\begin{tikzpicture}

\draw (0,1) -- (1,1) node[pos=0.25,above] {};
\draw[dashed] (1,1) -- (2,1) node[pos=0.5,above] {};
\draw (2,1) -- (3,1);
\draw (0,0) -- (1,0);
\draw[dashed] (1,0) -- (2,0) node[pos=0.5,below] {};
\draw (2,0) -- (3,0) node[pos=0.75,below] {};
\foreach \x in {1,...,9} {
	\draw[thin] (0,0.1*\x) -- (1,0.1*\x);
	\draw[thin] (2,0.1*\x) -- (3,0.1*\x);
}
\draw[->] (-0.2,0.9) -- (-0.2,0.1);
\draw[->] (3.2,0.1) -- (3.2,0.9);

\end{tikzpicture}
\hspace{3pt}
\begin{tikzpicture}
\node at (0,0.5) {$\to$};
\node at (0,0.7) {\tiny (1)};
\end{tikzpicture}
\hspace{3pt}
\begin{tikzpicture}

\draw (0.5,1) -- (1,1);
\draw[dashed] (1,1) -- (2,1) node[pos=0.5,above] {};
\draw (2,1) -- (3,1);
\draw (0,0) -- (1,0);
\draw[dashed] (1,0) -- (2,0) node[pos=0.5,below] {};
\draw (2,0) -- (2.5,0);
\draw (0,0) -- (0.5,1) node[pos=0.6,left] {};
\draw (2.5,0) -- (3,1) node[pos=0.4,right] {};
\foreach \x in {1,...,9} {
	\draw[thin] (0,0) -- (0.5,0.1*\x);
	\draw[thin] (0.5,0.1*\x) -- (1,0.1*\x);
	\draw[thin] (2,0.1*\x) -- (2.5,0.1*\x);
	\draw[thin] (2.5,0.1*\x) -- (3,1);
}
\draw[->] (0.1,-0.2) -- (0.9,-0.2);
\draw[->] (2.9,1.2) -- (2.1,1.2);
\end{tikzpicture}
\hspace{3pt}
\begin{tikzpicture}
\node at (0,0.5) {$\to$};
\node at (0,0.7) {\tiny (2)};
\end{tikzpicture}
\hspace{3pt}
\begin{tikzpicture}

\draw (0.5,1) -- (1,1);
\draw[dashed] (1,1) -- (2,1) node[pos=0.5,above] {};
\draw (2,1) -- (2.5,1);
\draw (0.5,0) -- (1,0);
\draw[dashed] (1,0) -- (2,0) node[pos=0.5,below] {};
\draw (2,0) -- (2.5,0);
\draw (0.5,0) -- (0.5,1) node[pos=0.6,left] {};
\draw (2.5,0) -- (2.5,1) node[pos=0.4,right] {};
\foreach \x in {1,...,9} {
	\draw[thin] (0.5,0) -- (0.5 + 0.05 * \x, 1 - 0.1 * \x);
	\draw[thin] (0.5 + 0.05 * \x, 1 - 0.1 * \x) -- (1, 1 - 0.1 * \x);
	\draw[thin] (2, 0.1 * \x) -- (2.5 - 0.05 * \x, 0.1 * \x);
	\draw[thin] (2.5 - 0.05 * \x, 0.1*\x) -- (2.5,1);
}
\end{tikzpicture}
\hspace{3pt}
\begin{tikzpicture}
\node at (0,0.5) {$\to$};
\node at (0,0.7) {\tiny (3)};
\end{tikzpicture}
\hspace{3pt}
\begin{tikzpicture}

\draw[dashed] (-0.5,1) -- (0,1);
\draw (0,1) -- (1,1) node[pos=0.5,above] {};
\draw[dashed] (1,1) -- (1.5,1);
\filldraw (0.5,1) circle (0.5pt); 
\draw[dashed] (-0.5,0) -- (0,0);
\draw (0,0) -- (1,0) node[pos=0.5,below] {};
\draw[dashed] (1,0) -- (1.5,0);
\filldraw (0.5,0) circle (0.5pt); 
\draw[densely dotted] (0.5,0) -- (0.5,1);
\foreach \x in {1,...,9} {
	\begin{scope}[shift={(0,0)}]
		\draw[thin] (0.5,0) -- (0.5 + 0.05 * \x, 1 - 0.1 * \x);
		\draw[thin] (0.5 + 0.05 * \x, 1 - 0.1 * \x) -- (1, 1 - 0.1 * \x);
	\end{scope}
	\begin{scope}[shift={(-2,0)}]
		\draw[thin] (2, 0.1 * \x) -- (2.5 - 0.05 * \x, 0.1 * \x);
		\draw[thin] (2.5 - 0.05 * \x, 0.1*\x) -- (2.5,1);
	\end{scope}
}

\end{tikzpicture}

%% file: fig_gluing_1.tikz
\begin{tikzpicture}

\node at (-0.5,1) {1};
\draw (0,1) -- (1,1) node[pos=0.25,above] {$A$};
\filldraw (0.5,1) circle (0.5pt);
\draw[dashed] (1,1) -- (2,1) node[pos=0.5,above] {$M$};
\draw (2,1) -- (3,1);
\node at (-0.5,0) {0};
\draw (0,0) -- (1,0);
\draw[dashed] (1,0) -- (2,0) node[pos=0.5,below] {$N$};
\draw (2,0) -- (3,0) node[pos=0.75,below] {$A$};
\filldraw (2.5,0) circle (0.5pt);
\draw[dotted] (0,1) -- (0,0);
\draw[dotted] (3,0) -- (3,1);

\node at (1.5,0.5) {\small $\Gamma$};
\node[dotnode] (b) at (0.75,0.5) {}; 
\draw (b) to[bend right] (0.5,0); 
\draw (b) to[out=180,in=-90] (0.25,1); 
\draw (b) -- (1,0.6); 
\draw[dashed] (1,0.6) -- (1.25,0.7); 
\draw (b) -- (1,0.4); 
\draw[dashed] (1,0.4) -- (1.25,0.3); 
\node[dotnode] (a) at (2.25,0.5) {}; 
\draw (a) to[out=0, in=90] (2.75,0); 
\draw (a) to[bend right] (2.5,1); 
\draw (a) to[out=-100, in=90] (2.2,0); 
\draw (a) -- (2,0.5); 
\draw[dashed] (2,0.5) -- (1.75,0.5); 

\node at (-2,0.5) {$X' \times \clI$};
\draw[->] (-2,0.2) -- (-2,-0.7) node[pos=0.5,left] {$\tilde{f}$};
\node at (-2,-1) {$\RR$};

\draw (0,-1) -- (1,-1) node[pos=0,below] {0} node[pos=1,below] {1};
\draw[dashed] (1,-1) -- (2,-1);
\draw (2,-1) -- (3,-1) node[pos=0,below] {2} node[pos=1,below] {3};

\end{tikzpicture}
\hspace{3pt}
\begin{tikzpicture}
\node at (0,0.5) {$\mapsto$};
\node at (0,0.7) {\tiny (1)};
\end{tikzpicture}
\hspace{3pt}
\begin{tikzpicture}

\draw (0.5,1) -- (1,1);
\draw[dashed] (1,1) -- (2,1) node[pos=0.5,above] {$M$};
\draw (2,1) -- (3,1);
\draw (0,0) -- (1,0);
\draw[dashed] (1,0) -- (2,0) node[pos=0.5,below] {$N$};
\draw (2,0) -- (2.5,0);
\filldraw (0.5,1) circle (0.5pt);
\draw (0,0) -- (0.5,1) node[pos=0.6,left] {$A^*$};
\filldraw (2.5,0) circle (0.5pt);
\draw (2.5,0) -- (3,1) node[pos=0.4,right] {$A$};

\node at (1.5,0.5) {\small $\Gamma$};
\node[dotnode] (b) at (0.75,0.5) {}; 
\draw (b) to[bend right] (0.5,0); 
\draw (b) to[out=180,in=-30] (0.25,0.5); 
\draw (b) -- (1,0.6); 
\draw[dashed] (1,0.6) -- (1.25,0.7); 
\draw (b) -- (1,0.4); 
\draw[dashed] (1,0.4) -- (1.25,0.3); 
\node[dotnode] (a) at (2.25,0.5) {}; 
\draw (a) to[out=0, in=120] (2.75,0.5); 
\draw (a) to[bend right] (2.5,1); 
\draw (a) to[out=-100, in=90] (2.2,0); 
\draw (a) -- (2,0.5); 
\draw[dashed] (2,0.5) -- (1.75,0.5); 

\draw (0,-1) -- (1,-1) node[pos=0,below] {0} node[pos=1,below] {1};
\draw[dashed] (1,-1) -- (2,-1);
\draw (2,-1) -- (3,-1) node[pos=0,below] {2} node[pos=1,below] {3};

\end{tikzpicture}
\hspace{3pt}
\begin{tikzpicture}
\node at (0,0.5) {$\mapsto$};
\node at (0,0.7) {\tiny (2),(3)};
\end{tikzpicture}
\hspace{3pt}
\begin{tikzpicture}

\draw[dashed] (-0.5,1) -- (0,1);
\draw (0,1) -- (1,1) node[pos=0.5,above] {$i_*(M)$};
\draw[dashed] (1,1) -- (1.5,1);
\filldraw (0.5,1) circle (0.5pt); 
\draw[dashed] (-0.5,0) -- (0,0);
\draw (0,0) -- (1,0) node[pos=0.5,below] {$i_*(N)$};
\draw[dashed] (1,0) -- (1.5,0);
\filldraw (0.5,0) circle (0.5pt); 
\draw[densely dotted] (0.5,0) -- (0.5,1);

\node[dotnode] (a) at (0.2,0.5) {}; 
\node at (-0.7,0.5) {\small $\Psi(\Gamma)$};
\draw (a) -- (0.5,0.5); 
\draw (a) to[bend right=10] (0.25,1); 
\draw (a) to[bend right=10] (0.2,0); 
\draw (a) -- (0,0.5); 
\draw[dashed] (0,0.5) -- (-0.25,0.5); 
\node[dotnode] (b) at (0.75,0.5) {}; 
\draw (b) to[bend right=10] (0.75,0); 
\draw (b) -- (0.5,0.5); 
\draw (b) -- (1,0.6); 
\draw[dashed] (1,0.6) -- (1.25,0.7); 
\draw (b) -- (1,0.4); 
\draw[dashed] (1,0.4) -- (1.25,0.3); 

\node at (2.5,0.5) {$X \times \clI$};
\draw[->] (2.5,0.2) -- (2.5,-0.7) node[pos=0.5,right] {$f$};
\node at (2.5,-1) {$\RR/2\ZZ$};

\draw[dashed] (-0.5,-1) -- (0,-1);
\draw (0,-1) -- (1,-1) node[pos=0,below] {0} node[pos=1,below] {1};
\draw[dashed] (1,-1) -- (1.5,-1);

\end{tikzpicture}

\comment{
\begin{tikzpicture}

\foreach \x in {0,...,10}
  \tikzmath{\y = sin(deg(pi*\x/10))/5;}
  \node[dotnode] at (\x/10, \y) {};

\end{tikzpicture}
}

%% file: computations.tex
\section{Examples and Computations}
\label{s:computations}

In this section, we present some examples and computations
using the results obtained so far.

\begin{example}
  $\ZTV(S^1) \simeq \cZ(\cA)$.
This follows from applying \thmref{t:excision-skein} to
$X' = (0,3)$, $X = S^1 = \RR/2\ZZ$
(see \xmpref{xmp:tv_circle}).
\hfill $\qed$

This example, is, of course, well known: see, e.g., \ocite{kirillov-stringnet}, \ocite{DSSP}.
\end{example}

\begin{example}
\label{xmp:annulus}
$\ZCY(\Ann) \simeq \cZ(\cA)$ as abelian categories,
where $\Ann = I \times S^1$ is the annulus.
Here we get $\Ann$ by gluing $I \times I$ to itself
in the vertical direction (see \xmpref{xmp:cy_annulus}).
The result follows from applying \thmref{t:excision-skein} to
$X' = I \times (0,3)$, $X = \Ann = I \times \RR/2\ZZ$,
with $P=I$.

Again, this result is not new: see, e.g., \ocite{BBJ1}.

Let us flesh out some details.
Define $\hA = \htr(\cA)$, where $\cA$ is an $\cA$-bimodule
by left, right multiplication.
\thmref{t:hat_excision} gives an equivalence
$\hA \simeq \hatZCY(\Ann)$, pictorially given by
the following figure on the left:
\begin{equation}
\label{e:hA_Ann}
\input{fig_hA_Ann.tikz}
\end{equation}
Here the $A$ loop is given a trivial (e.g. always horizontal) framing.
It is clear from this picture that $\End_\hA(\one)$
is commutative.\\

By \prpref{prp:stacking},
$\Ann = S^1 \times I$ has a horizontal stacking operation
that, under the equivalence $\hA \simeq \hatZCY(\Ann)$ above,
is given by a map
$\ihom{\cA}{A_1}(Y_1,Y_1') \tnsr \ihom{\cA}{A_2}(Y_2,Y_2')
\to \ihom{\cA}{A_1 \tnsr A_2}(Y_1 \tnsr Y_2, Y_1' \tnsr Y_2)$
described as follows:
\begin{equation}
\input{fig_hA_stack.tikz}
\end{equation}
This stacking operation gives rise to the monoidal structure
that is defined in \prpref{prp:stacking},
where we take $P=S^1$.
\hfill $\qed$
\end{example}

\begin{remark}
Note that the stacking operation in \xmpref{xmp:annulus}
does \emph{not} result in the usual tensor product on the Drinfeld center
$\cZ(\cA)$ (the latter can be defined as the functor assigned to pair of pants
in Turaev-Viro theory). It is explored in more detail in \ocite{tham_reduced},
where the tensor product is defined
purely in terms of $\cA$ (i.e. without recourse to topology);
it is shown that the stacking tensor product is typically not spherical 
(but pivotal) and not fusion (but multifusion).
\label{rmk:ZA_stacking}
\end{remark}

Next we will be concerned with relating $\ZCY$ of a surface $\Sigma$
with that of a punctured one $\Sigma_0$, that is,
$\Sigma_0 = \Sigma \backslash \{p\}$.
We will think of $\Sigma$ as obtained from $\Sigma_0$ by
gluing with an open disk, ``sealing'' the puncture:
$\Sigma = \Sigma_0 \cup \DD^2$,
implicitly choosing some collared structure on $\Sigma_0$ and $\DD^2$.\\

Recall $\hA := \htr(\cA)$ from \xmpref{xmp:annulus}.
There is a right action of $\Hom_\hA(\one,\one)$ on the
morphisms of $\hatZCY(\Sigma_0)$, by
``pushing in'' from the puncture, i.e.
$\Hom_{\hatZCY(\Sigma_0)}(Y,Y') \tnsr \Hom_\hA(\one,\one)
\to \Hom_{\hatZCY(\Sigma_0)}(Y \ract \one, Y' \ract \one)
\cong \Hom_{\hatZCY(\Sigma_0)}(Y, Y')$.
It is easy to see that for $\Gamma \in \Hom_{\hatZCY(\Sigma_0)}(Y,Y')$
and $f,g \in \Hom_\hA(\one,\one)$,
$\Gamma \ract (f \circ g) = (\Gamma \ract f) \ract g$.
Moreover, for $\Gamma' \in \Hom_{\hatZCY(\Sigma_0)}(Y',Y'')$,
$(\Gamma' \circ \Gamma) \ract (f \circ g) =
(\Gamma' \ract f) \circ (\Gamma \ract g)$.\\

Let $\pi = \sum d_i/\cD \cdot \id_{X_i}
\in \bigoplus \ihom{\cA}{X_i}(\one,\one) = \Hom_\hA(\one,\one)$.
(Note: $\cD$ and simples $X_i$ are of $\cA$, and not of $\cZ(\cA)$.)
$\pi$ is an idempotent in $\Hom_\hA(\one,\one)$,
and hence also acts as an idempotent on $\Hom_{\hatZCY(\Sigma_0)}(Y,Y')$.

\begin{proposition}
\label{prp:punctured_glue}
Let $\Sigma_0 = \Sigma \backslash \{p\}$ as above.
Consider the category $\hB$ 
consisting of the same objects as $\hatZCY(\Sigma_0)$,
but morphisms given by
\[
  \Hom_\hB(Y,Y') = \im(\Hom_{\hatZCY(\Sigma_0)}(Y,Y') \rcirclearrowleft \pi)
\]
Then the restriction to $\hB$ of the inclusion functor
corresponding to $i: \Sigma_0 \hookrightarrow \Sigma$
is an equivalence:
\[
  i_*|_\hB : \hB \simeq \hatZCY(\Sigma)
\]
\end{proposition}

\begin{proof}
First note that $\hB$ is indeed closed under composition
of morphisms because $\pi$ is idempotent.
It is clear that $i_*|_\hB$ is essentially surjective.
To prove fully faithfulness, consider two objects
$Y,Y' \in \hatZCY(\Sigma_0)$.
Abusing notation, we also denote $i_*(Y), i_*(Y')
\in \Obj \hatZCY(\Sigma_0)$ by $Y,Y'$.
We call the vertical segment
$p\times \clI \subset \Sigma \times \clI$
the \emph{pole},
so that $\Sigma_0 \times \clI = \Sigma \times \clI 
\backslash $pole.\\

We construct an inverse map to $i_*$.
Let $U$ be a small open neighborhood of $p$ in $\Sigma$,
and let $\cN = U \times \clI \subset \Sigma \times \clI$
be a small open neighborhood of the pole.
Choose $U$ small enough so that it does not contain any
marked points of $Y, Y'$.
Consider a graph $\Gamma \in \Graph(\Sigma \times \clI; Y^*,Y')$.
Define $j(\Gamma)$ as follows: if $\Gamma$ intersects the pole,
then use an isotopy supported in $\cN$
to push $\Gamma$ off of it, resulting in a new graph $\Gamma'$.
Now $\Gamma'$ can be considered a graph in
$\Graph(\Sigma_0 \times \clI; Y^*,Y')$.
Then we define $j(\Gamma) = \Gamma' \ract \pi$.\\

We need to check that $j$ is well-defined.
Firstly, the (linear combination of) graphs
$\Gamma' \ract \pi$ is independent of the choice of isotopy
- this follows from the sliding lemma (\lemref{lem:sliding}).
More generally, it means that for any isotopy $\vphi$ of
$\Sigma\times \clI$ supported on $\cN$,
$j(\Gamma) = j(\vphi(\Gamma))$.\\

Now we check that $j$ sends null graphs to 0.
Take the two set open cover $\{\cN, \Sigma_0\times \clI\}$
of $\Sigma \times \clI$, and apply \prpref{prp:null_cover}.
Let $\Gamma = \sum c_i\Gamma_i$ be null with respect to
some ball $D$.
If $D \subset \Sigma_0\times \clI$, clearly
$j(\Gamma)$ is null with respect to $D$.
If $D \subset \cN$, we may assume $D$ doesn't touch the
boundary (by choice of $U$),
so we can isotope it with some isotopy $\vphi$
supported on $\cN$ so that $\vphi(D)$ doesn't meet the pole.
Then clearly $j(\vphi(\Gamma))$ is null with respect to $\vphi(D)$.

Finally, it is easy to see that $j$ is inverse to $i_*$.
For example, $i_* \circ j$ amounts
to adding a trivial dashed circle, which is equivalent to 1
by \lemref{l:dashed_circle}.
\end{proof}

\begin{corollary}
  \label{cor:muger}
  $\ZCY(S^2) \simeq \ZMu(\cA)$,
  the M\"uger center of $\cA$,
  and in particular, when $\cA$ is modular,
  $\ZCY(S^2) \simeq \Vect$.
\end{corollary}
\begin{proof}
Think of the disk $\disk$ as a punctured sphere,
so by \prpref{prp:punctured_glue}, we have that
$\hatZCY(S^2) \simeq \hB$,
where $\hB$ is the category
with the same objects as $\hatZCY(\disk) \simeq \cA$,
but morphisms are, for $A,A'\in \cA$,
\[
\Hom_\hB(A,A') =
\Bigg\{
\frac{1}{\cD}
\begin{tikzpicture}
\draw[regular] (-0.3,-0.1) to[out=90,in=90] (0.3,-0.1);
\draw[overline] (0,0.6) -- (0,-0.6);
\node[small_morphism] at (0,0.3) {\tiny $f$};
\draw[overline,regular] (-0.3,-0.1) to[out=-90,in=-90] (0.3,-0.1);
\end{tikzpicture}
\st f \in \Hom_\cA(A,A')
\Bigg\}
\]
In particular, when $A = A' = X_i$ a simple object,
it follows from \ocite{muger}*{Corollary 2.14} that
simple objects that are not transparent,
i.e. not in the M\"uger center,
are killed:
\[
\End_\hB(X_i) =
\begin{cases}
  \End_\cA(X_i) & \text{ if } X_i \in \ZMu(\cA)\\
  0 & \text{ otherwise}
\end{cases}
\]
It follows that $\hB$ coincides with the M\"uger center,
which is already abelian,
and so $\ZCY(S^2) \simeq \Kar(\hB) = \ZMu(\cA)$.
\end{proof}

%% file: fig_hA_Ann.tikz
\begin{tikzpicture}
\begin{scope}[shift={(0,-1.5)}]
\node at (0,3) {$\hA$};
\node at (0,2) {$Y$};
\node at (0,0) {$Y'$};
\node[small_morphism] (ph) at (0,1) {$\vphi$};
\draw (ph) -- (0,0.2);
\draw (ph) -- (0,1.8);
\draw (ph) -- (-0.5,1.5) node[pos=1,above left] {$A$};
\draw (ph) -- (0.5,0.5) node[pos=1,below right] {$A$};
\end{scope}
\end{tikzpicture}
\begin{tikzpicture}
\begin{scope}[shift={(0,-1.5)}]
\node at (0,0) {$\mapsto$};
\node at (0,1) {$\mapsto$};
\node at (0,2) {$\mapsto$};
\node at (0,3) {$\longrightarrow$};
\end{scope}
\end{tikzpicture}
\begin{tikzpicture}
\begin{scope}[shift={(0,-1.5)}]
\node at (0,3) {$\hatZCY(\Ann)$};
\draw (0,0) ellipse (1cm and 0.5cm);
\draw (0.75,1) arc(0:180:0.75cm and 0.375cm);
\draw[overline] (-0.5,0) -- (-0.5,2);
\draw[overline] (0.5,0) -- (0.5,2);
\draw (0,0) ellipse (0.5cm and 0.25cm); 
\draw[overline] (0,2) ellipse (1cm and 0.5cm);
\draw (0,2) ellipse (0.5cm and 0.25cm);
\draw (-1,0) -- (-1,2);
\draw (1,0) -- (1,2);
\node[dotnode, label={[shift={(-0.1,-0.04)}] $Y$}] (top) at (0.75,2) {};
\node[dotnode, label={[shift={(-0.08,-0.43)}] $Y'$}] (bottom) at (0.75,0) {};
\draw[overline, midarrow] (0.75,1) arc(0:-180:0.75cm and 0.375cm)
  node[pos=0.5, above] {\small $A$};
\draw (top) -- (bottom);
\node[dotnode, label={[shift={(0.15,-0.1)}] \tiny $\vphi$}] at (0.75,1) {};
\end{scope}
\end{tikzpicture}
\begin{tikzpicture}
\begin{scope}[shift={(0,-1.5)}]
\node at (0,0) {$\mapsto$};
\node at (0,1) {$\mapsto$};
\node at (0,2) {$\mapsto$};
\node at (0,3) {$\simeq$};
\end{scope}
\end{tikzpicture}
\begin{tikzpicture}
\begin{scope}[shift={(0,-1.5)}]
\node at (0,3) {$\hatZCY(\Ann)$};
\draw (0,0) ellipse (1cm and 0.5cm);
\draw (0.75,1) arc(0:180:0.75cm and 0.375cm);
\draw[overline] (-0.5,0) -- (-0.5,2);
\draw[overline] (0.5,0) -- (0.5,2);
\draw (0,0) ellipse (0.5cm and 0.25cm); 
\draw[overline] (0,2) ellipse (1cm and 0.5cm);
\draw (0,2) ellipse (0.5cm and 0.25cm);
\draw (-1,0) -- (-1,2);
\draw (1,0) -- (1,2);
\node[dotnode, label={[shift={(0.2,-0.04)}] $Y$}] (top) at (-0.75,2) {};
\node[dotnode, label={[shift={(0.08,-0.43)}] $Y'$}] (bottom) at (-0.75,0) {};
\draw[overline, midarrow_rev] (0.75,1) arc(0:-180:0.75cm and 0.375cm)
  node[pos=0.5, above] {\small $A$};
\draw (top) -- (bottom);
\node[dotnode, label={[shift={(-0.15,-0.1)}] \tiny $\vphi$}] at (-0.75,1) {};
\end{scope}
\end{tikzpicture}

%% file: fig_hA_stack.tikz
\begin{tikzpicture}
\begin{scope}[shift={(0,-1.5)}]
\node at (0,3) {$\hA$};
\node at (0,2) {$Y_1$};
\node at (0,0) {$Y_1'$};
\node[small_morphism] (ph) at (0,1) {$\vphi_1$};
\draw (ph) -- (0,0.2);
\draw (ph) -- (0,1.8);
\draw (ph) -- (-0.5,1.5) node[pos=1,above left] {$A_1$};
\draw (ph) -- (0.5,0.5) node[pos=1,below right] {$A_1$};
\end{scope}
\end{tikzpicture}
\begin{tikzpicture}
\begin{scope}[shift={(0,-1.5)}]
\node at (0,1) {$\boxtimes$};
\node at (0,3) {$\boxtimes$};
\end{scope}
\end{tikzpicture}
\begin{tikzpicture}
\begin{scope}[shift={(0,-1.5)}]
\node at (0,3) {$\hA$};
\node at (0,2) {$Y_2$};
\node at (0,0) {$Y_2'$};
\node[small_morphism] (ph) at (0,1) {$\vphi_2$};
\draw (ph) -- (0,0.2);
\draw (ph) -- (0,1.8);
\draw (ph) -- (-0.5,1.5) node[pos=1,above left] {$A_2$};
\draw (ph) -- (0.5,0.5) node[pos=1,below right] {$A_2$};
\end{scope}
\end{tikzpicture}
\begin{tikzpicture}
\begin{scope}[shift={(0,-1.5)}]
\node at (0,1) {$\mapsto$};
\node at (0,3) {$\longrightarrow$};
\end{scope}
\end{tikzpicture}
\begin{tikzpicture}
\begin{scope}[shift={(0,-1.5)}]
\node at (0,3) {$\hA$};
\node at (0,2) {$Y_1 \tnsr Y_2$};
\node at (0,0) {$Y_1' \tnsr Y_2'$};
\node[small_morphism] (ph1) at (-0.2,0.9) {\tiny $\vphi_1$};
\node[small_morphism] (ph2) at (0.2,1.1) {\tiny $\vphi_2$};
\draw (ph2) -- (0.2,1.8);
\draw (ph2) -- (0.2,0.2);
\draw (ph2) -- +(150:1cm);
\draw (ph2) -- +(-30:0.8cm);
\draw[overline] (ph1) -- (-0.2,1.8);
\draw (ph1) -- (-0.2,0.2);
\draw (ph1) -- +(150:0.8cm);
\draw[overline] (ph1) -- +(-30:1cm);
\end{scope}
\end{tikzpicture}

%% file: elliptic_center.tex
\section{Crane-Yetter and the Elliptic Drinfeld Center}
\label{sec:elliptic}

In \ocite{tham_elliptic}, the second author constructed a category
similar to the Drinfeld center, but instead the objects
have two half-braidings that satisfy some compatibility.
In this section, we show that this category is
the category of boundary values on the once-punctured torus.\\

We note that all morphisms depicted using graphical calculus
are over $\cA$, but they may represent morphisms
in a different category. In particular,
dashed lines do not need an orientation
and in makes sense to use the circular $\al$
instead of the semicircular one.

For the reader's convenience, we recall the definition
and some properties of the elliptic Drinfeld center:

\begin{definition}
Let $\cA$ be a premodular category.
The category $\ZZA$ consists of objects of the form
$(A,\lambda^1,\lambda^2)$, where $\lambda^1,\lambda^2$
are half-braidings on $A$ that satisfy
\begin{align} \label{e:COMM}
\begin{tikzpicture}
\begin{scope} [shift={(0,-0.7)}]
\node[small_morphism] (lmb1) at (0,0.4) {\tiny $\lmb^1$};
\node[small_morphism] (lmb2) at (0,1) {\tiny $\lmb^2$};
\draw (0,0) -- (lmb1)
  node[pos=0,below] {$A$};
\draw (lmb1) -- (lmb2);
\draw (lmb2) -- (0,1.4)
  node[pos=1,above] {$A$};
\draw (lmb1) to[out=180,in=-90] (-1,1.4);
\draw (lmb1) to[out=0,in=90] (1,0);
\draw (lmb2) to[out=180,in=-90] (-0.5,1.4);
\draw[overline] (lmb2) to[out=0,in=90] (0.5,0);
\end{scope}
\end{tikzpicture}
=
\begin{tikzpicture}
\begin{scope} [shift={(0,-0.7)}]
\node[small_morphism] (lmb1) at (0,1) {\tiny $\lmb^1$};
\node[small_morphism] (lmb2) at (0,0.4) {\tiny $\lmb^2$};
\draw (0,0) -- (lmb2)
  node[pos=0,below] {$A$};
\draw (lmb2) -- (lmb1);
\draw (lmb1) -- (0,1.4)
  node[pos=1,above] {$A$};
\draw (lmb2) to[out=180,in=-90] (-0.5,1.4);
\draw (lmb2) to[out=0,in=90] (0.5,0);
\draw[overline] (lmb1) to[out=180,in=-90] (-1,1.4);
\draw (lmb1) to[out=0,in=90] (1,0);
\end{scope}
\end{tikzpicture}
\end{align}
We call the relation \eqnref{e:COMM} ``COMM''.
The morphisms $\Hom_\ZZA((A,\lmb^1,\lmb^2),(A',\mu^1,\mu^2)$
are morphisms of $\cA$ that intertwine both half-braidings,
i.e.
\[
  \Hom_\ZZA((A,\lmb^1,\lmb^2),(A',\mu^1,\mu^2))
    :=
  \Hom_{\ZA}((A,\lmb^1),(A',\mu^1))
  \cap \Hom_{\ZA}((A,\lmb^2),(A',\mu^2))
\]
\comment{
\[
  \Hom_\ZZA((A,\lmb^1,\lmb^2),(A',\mu^1,\mu^2)
    := \{\vphi \in \Hom_\cA(A,A') \st
        (\id_{A'}\tnsr \vphi) \circ \lmb^i = \mu^i \circ (\vphi \tnsr \id_A)\}
\]
}

\end{definition}

\begin{proposition}[\ocite{tham_elliptic}, Prop 3.4]
$\ZZA$ is a finite semisimple category.
\end{proposition}

\begin{proposition}[\ocite{tham_elliptic}, Prop 3.5, Prop 3.8]
\label{prp:elliptic_adjoint}
The forgetful functor $\mathcal{F}^\text{el} : \ZZA \to \cA$ has a two-sided
adjoint $\Iel: \cA \to \ZZA$,
where on objects, $\Iel$ sends
\[
  A \mapsto (\dirsum_{i,j} X_iX_jAX_j^*X_i^*, \Gamma^1, \Gamma^2)
\]
where
\[
\Gamma^1 =
\begin{tikzpicture}
\draw (0,-1) -- (0,1);
\draw (-0.3,-1) -- (-0.3,1);
\draw (0.3,-1) -- (0.3,1);
\node[small_morphism] (al1) at (-0.6,0) {\tiny $\albar$};
\draw (al1) -- (-0.6,-1);
\draw (al1) -- (-0.6,1);
\draw (al1) to[out=180,in=-90] (-1.2,1);
\node[small_morphism] (al2) at (0.6,0) {\tiny $\albar$};
\draw (al2) -- (0.6,-1);
\draw (al2) -- (0.6,1);
\draw (al2) to[out=0,in=90] (1.2,-1);
\end{tikzpicture}
,
\Gamma^2 =
\begin{tikzpicture}
\draw (0,-1) -- (0,1);
\draw (0.6,-1) -- (0.6,1);
\node[small_morphism] (al1) at (-0.3,0) {\tiny $\albar$};
\draw (al1) -- (-0.3,-1);
\draw (al1) -- (-0.3,1);
\draw (al1) to[out=180,in=-90] (-1.2,1);
\draw[overline] (-0.6,-1) -- (-0.6,1);
\node[small_morphism] (al2) at (0.3,0) {\tiny $\albar$};
\draw (al2) -- (0.3,-1);
\draw (al2) -- (0.3,1);
\draw[overline] (al2) to[out=0,in=90] (1.2,-1);
\end{tikzpicture}
\]
where $\albar$ is defined in \lemref{l:halfbrd}.

On morphisms, $f\in \Hom_\cA(A,A')$,
\[
  \Iel(f) = \dirsum_{i,j} \id_{X_i X_j} \tnsr f \tnsr \id_{X_j^* X_i^*}
\]

We refer to \ocite{tham_elliptic} for the functorial isomorphisms
giving the adjunction.

Furthermore, $\Iel$ is dominant.
\end{proposition}

\begin{theorem}[\ocite{tham_elliptic}, Theorem 4.3]
\label{thm:modular}
When $\cA$ is modular,
there is an equivalence
\begin{align*}
\cA &\simeq \ZZA \\
A &\mapsto (\bigoplus_i X_i A X_i^*, \Gamma, \Omega)
\end{align*}
where $\Gamma$ is the half-braiding on $I(X)$ in \thmref{t:center},
and $\Omega = c_{X_i^*,-}^\inv \circ c_{-,A} \circ c_{-,X_i}$,
where $c_{-,-}$ is the braiding on $\cA$.
\end{theorem}

\begin{proposition}
\label{prp:zza_punctorus}
Let $\punctorus$ be the once-punctured torus.
There is an equivalence
\[
  \ZCY(\punctorus) \cong \ZZA
\]
Under this equivalence, the inclusion functor
$\cA \simeq \ZCY(\DD^2) \to \ZCY(\punctorus)$
is identified with $\Iel: \cA \to \ZZA$.
\end{proposition}

\begin{proof}
Think of the once-punctured torus as an open disk,
drawn like a `+' sign, with opposite sides identified
($\Ann = S^1 \times I$):
\[
\begin{tikzpicture}
\node at (1,1) {$\DD^2$};
\draw[fill=gray!30] (0.5,2) rectangle (1.5,2.5);
\draw[fill=gray!30,draw=none] (0.51,1.9) rectangle (1.49,2.1);
\node at (1,2.25) {2};
\draw (1.5,2) to[out=-90,in=180] (2,1.5);
\draw[fill=gray!30] (2,1.5) rectangle (2.5,0.5);
\draw[fill=gray!30,draw=none] (1.9,1.49) rectangle (2.1,0.51);
\node at (2.25,1) {3};
\draw (2,0.5) to[out=180,in=90] (1.5,0);
\draw[fill=gray!30] (1.5,0) rectangle (0.5,-0.5);
\draw[fill=gray!30,draw=none] (1.49,0.1) rectangle (0.51,-0.1);
\node at (1,-0.25) {4};
\draw (0.5,0) to[out=90,in=0] (0,0.5);
\draw[fill=gray!30] (0,0.5) rectangle (-0.5,1.5);
\draw[fill=gray!30,draw=none] (0.1,0.51) rectangle (-0.1,1.49);
\node at (-0.25,1) {1};
\draw (0,1.5) to[out=0,in=-90] (0.5,2);
\end{tikzpicture}
\begin{tikzpicture}
  \draw[->] (0,1) -- (1,1) node[pos=0.5,above] {glue 1,3};
\end{tikzpicture}
\begin{tikzpicture}
\node at (1,1) {$\Ann$};
\draw[fill=gray!30] (0.5,2) rectangle (1.5,2.5);
\draw[fill=gray!30,draw=none] (0.51,1.9) rectangle (1.49,2.1);
\node at (1,2.25) {2};
\draw (1.5,2) to[out=-90,in=180] (2,1.5);
\draw[regular] (2,1.5) -- (2.5,1.5);
\draw[regular] (2,0.5) -- (2.5,0.5);
\draw (2,0.5) to[out=180,in=90] (1.5,0);
\draw[fill=gray!30] (1.5,0) rectangle (0.5,-0.5);
\draw[fill=gray!30,draw=none] (1.49,0.1) rectangle (0.51,-0.1);
\node at (1,-0.25) {4};
\draw (0.5,0) to[out=90,in=0] (0,0.5);
\draw[regular] (0,0.5) -- (-0.5,0.5);
\draw[regular] (0,1.5) -- (-0.5,1.5);
\draw (0,1.5) to[out=0,in=-90] (0.5,2);
\end{tikzpicture}
\begin{tikzpicture}
  \draw[->] (0,1) -- (1,1) node[pos=0.5,above] {glue 2,4};
\end{tikzpicture}
\begin{tikzpicture}
\node at (1,1) {$A$};
\draw[regular] (1.5,2) -- (1.5,2.5);
\draw (1.5,2) to[out=-90,in=180] (2,1.5);
\draw[regular] (2,1.5) -- (2.5,1.5);
\draw[regular] (2,0.5) -- (2.5,0.5);
\draw (2,0.5) to[out=180,in=90] (1.5,0);
\draw[regular] (1.5,0) -- (1.5,-0.5);
\draw[regular] (0.5,0) -- (0.5,-0.5);
\draw (0.5,0) to[out=90,in=0] (0,0.5);
\draw[regular] (0,0.5) -- (-0.5,0.5);
\draw[regular] (0,1.5) -- (-0.5,1.5);
\draw (0,1.5) to[out=0,in=-90] (0.5,2);
\draw[regular] (0.5,2) -- (0.5,2.5);
\end{tikzpicture}
\]

The left most figure shows how $\ZCY(\DD^2) \simeq \cA$
is a module category over $\ZCY(I\times I) \simeq \cA$
in four ways;
we think of the 1,2 edges as acting on the left,
3,4 edges as acting on the right.
The actions are just usual left and right multiplication.

By \thmref{t:hat_excision},
the first ``glue 1,3" arrow induces an equivalence
$\hatZCY(\Ann) \simeq \htr_{\hatZCY(I\times I)}(\hatZCY(\DD^2))
\simeq \htr_{\cA}(\cA)$
(see also \xmpref{xmp:annulus}).
Again by \thmref{t:hat_excision},
the second ``glue 2,4" arrow induces an equivalence
$\hatZCY(\punctorus) \simeq
\htr_{\hatZCY(I\times I)}(\hatZCY(\Ann))
\simeq \htr_{\cA}(\htr_{\cA}(\cA))$.

Let us give a more explicit description of the last equivalence.
For $A,A' \in \Obj \cA$,
\[
  \Hom_{\htr(\htr(\cA))}(A,A')
  \cong \int^{B_2} \Hom_{\htr(\cA)}(B_2 \tnsr A, A' \tnsr B_2)
  \cong \int^{B_2} \int^{B_1}
      \Hom_\cA(B_1 \tnsr B_2 \tnsr A, A' \tnsr B_2 \tnsr B_1)
\]
Under the equivalence, a morphism
$\vphi \in \Hom_\cA(B_1 \tnsr B_2 \tnsr A, A' \tnsr B_2 \tnsr B_1)$,
shown on the left in the figure below,
is sent to the graph
in $\punctorus \times [0,1]$ shown on the right:
\begin{equation}
\label{eqn:morphism_punctorus}
\begin{tikzpicture}
\begin{scope}[shift={(0,-1.5)}]
\node[morphism] (ph) at (0,1.5) {$\vphi$};
\draw (ph) -- +(90:1cm) node[pos=1,above] {$A$};
\draw (ph) -- +(-90:1cm) node[pos=1,below] {$A'$};
\draw[midarrow_rev] (ph) -- +(-150:1cm)
  node[pos=1,below left] {1}
  node[pos=0.5,below] {\tiny $B_1$};
\draw[midarrow_rev] (ph) -- +(150:1cm)
  node[pos=1,above left] {2}
  node[pos=0.5,above] {\tiny $B_2$};
\draw[midarrow] (ph) -- +(30:1cm)
  node[pos=1,above right] {3}
  node[pos=0.5,above] {\tiny $B_1$};
\draw[midarrow] (ph) -- +(-30:1cm)
  node[pos=1,below right] {4}
  node[pos=0.5,below] {\tiny $B_2$};
\end{scope}
\end{tikzpicture}
\mapsto
\begin{tikzpicture}
\begin{scope}[shift={(0,-1.5)}]
\draw (-0.2,2.75) to[out=0,in=-60] (0.3,3);
\draw (1.3,3) to[out=-60,in=180] (2,2.75);
\draw (2.2,2.25) to[out=180,in=120] (1.7,2);
\draw (0.7,2) to[out=120,in=0] (0,2.25);
\draw (-0.2,0.75) to[out=0,in=-60] (0.3,1);
\draw (1.3,1) to[out=-60,in=180] (2,0.75);
\draw (2.2,0.25) to[out=180,in=120] (1.7,0);
\draw (0.7,0) to[out=120,in=0] (0,0.25);
\draw[dotted] (-0.2,0.75) -- (-0.2,2.75);
\draw[dotted] (0.3,1) -- (0.3,3);
\draw[dotted] (1.3,1) -- (1.3,3);
\draw[dotted] (2,0.75) -- (2,2.75);
\draw[dotted] (2.2,0.25) -- (2.2,2.25);
\draw[dotted] (1.7,0) -- (1.7,2);
\draw[dotted] (0.7,0) -- (0.7,2);
\draw[dotted] (0,0.25) -- (0,2.25);
\node[small_morphism] (ph) at (1,1.5) {\tiny $\vphi$};
\node[dotnode] (top) at (1,2.5) {};
\node[dotnode] (bottom) at (1,0.5) {};
\draw (ph) -- (top) node[pos=1,above] {$A^*$};
\draw (ph) -- (bottom) node[pos=1,below] {$A'$};
\draw[midarrow_rev] (ph) -- (-0.1,1.5)
  node[pos=1,left] {\small 1}
  node[pos=0.4,above] {\tiny $B_1$};
\draw (ph) -- (0.8,2)
  node[pos=1,above] {\small 2};
\draw (ph) -- (2.1,1.5)
  node[pos=1,right] {\small 3};
\draw[midarrow={0.6}] (ph) -- (1.2,1)
  node[pos=1,below] {\small 4};
\node at (1.3, 1.3) {\tiny $B_2$};
\end{scope}
\end{tikzpicture}
\end{equation}

Now we define a functor $\htr_\cA(\htr_\cA(\cA)) \to \ZZA$.
On objects, it sends $A \mapsto \Iel(A)$.
On morphisms, the morphism in \eqnref{eqn:morphism_punctorus}
is sent to

\begin{equation}
\label{eqn:punctorus_elliptic}
\begin{tikzpicture}
\node[small_morphism] (ph) at (0,0) {$\vphi$};
\draw (ph) -- +(90:1.2cm) node[pos=1,above] {$A$};
\draw (ph) -- +(-90:1.2cm) node[pos=1,below] {$A'$};
\node[small_morphism] (al1) at (-0.6,0.5) {$\albar$};
\draw[midarrow_rev] (ph) -- (al1)
  node[pos=0.4,above] {\tiny $B_2$};
\draw[midarrow_rev=0.7] (al1) -- (-0.6,1.2);
\draw[midarrow=0.9] (al1) -- (-0.6,-1.2);
\node at (-0.8,1) {\tiny $j$};
\node at (-0.8,-0.9) {\tiny $j'$};
\node[small_morphism] (be1) at (-1.2,-0.5) {\small $\bar{\beta}$};
\draw[midarrow_rev={0.3}, overline] (ph) -- (be1)
  node[pos=0.2,below] {\tiny $B_1$};
\draw[midarrow_rev=0.9] (be1) -- (-1.2,1.2);
\draw[midarrow={0.8}] (be1) -- (-1.2,-1.2);
\node at (-1.4,1) {\tiny $i$};
\node at (-1.4,-1) {\tiny $i'$};
\node[small_morphism] (be2) at (1.2,0.5) {\small $\bar{\beta}$};
\draw (ph) -- (be2);
\draw[midarrow=0.8] (be2) -- (1.2,1.2);
\draw[midarrow_rev=0.9] (be2) -- (1.2,-1.2);
\node at (1.4,1) {\tiny $i$};
\node at (1.4,-1) {\tiny $i'$};
\node[small_morphism] (al2) at (0.6,-0.5) {$\albar$};
\draw (ph) -- (al2);
\draw[midarrow=0.9, overline] (al2) -- (0.6,1.2);
\draw[midarrow_rev={0.7}] (al2) -- (0.6,-1.2);
\node at (0.8,1) {\tiny $j$};
\node at (0.8,-0.9) {\tiny $j'$};
\end{tikzpicture}
\end{equation}

It is clear that this assignment respects composition of morphisms.
The following sequence of isomorphisms shows that
this functor is fully faithful:

\begin{align*}
  \Hom_{\htr(\htr(\cA))}(A,A')
  &\cong \bigoplus_{i_1,i_2 \in \Irr(\cA)}
    \Hom_\cA(X_{i_1} \tnsr X_{i_2} \tnsr A, A' \tnsr X_{i_2} \tnsr X_{i_1})
    \; (\text{by \lemref{l:isom1}})\\
  &\cong \Hom_\cA(X_{i_1} \tnsr X_{i_2} \tnsr A \tnsr X_{i_2}^* \tnsr X_{i_1}^*, A') \\
  &\cong \Hom_{\ZZA}(\Iel(A),\Iel(A'))
    \; (\text{by \prpref{prp:elliptic_adjoint}})
\end{align*}
Since $\ZZA$ is abelian,
we have that the extension to the Karoubi envelope is an equivalence:
\[
  \ZCY(\punctorus) \simeq \Kar(\htr_\cA(\htr_\cA(\cA))) \simeq \ZZA
\]
and we are done. But before we end the proof,
we provide an explicit inverse functor that will be useful later:
on objects,
\begin{equation}
\label{e:P_elliptic}
(A,\lmb^1,\lmb^2) \mapsto
\im(P_{(A,\lmb^1,\lmb^2)})
\text{, where }
P_{(A,\lmb^1,\lmb^2)} :=
\frac{1}{\cD^2}
\begin{tikzpicture}
\begin{scope}[shift={(0,-1)}]
\node[small_morphism] (lmb1) at (0,1.3) {\tiny $\lmb^1$};
\node[small_morphism] (lmb2) at (0,0.7) {\tiny $\lmb^2$};
\draw (lmb1) -- (0,2) node[pos=1,above] {$A$};
\draw (lmb1) -- (lmb2);
\draw (lmb2) -- (0,0) node[pos=1,below] {$A$};
\draw[regular] (lmb2) -- +(150:1.5cm) node[pos=1,above left] {2};
\draw[regular] (lmb2) -- +(-30:0.6cm) node[pos=1,below right] {4};
\draw[thick_overline,regular] (lmb1) -- +(-150:1.5cm) node[pos=1,below left] {1};
\draw[regular] (lmb1) -- +(30:0.6cm) node[pos=1,above right] {3};
\end{scope}
\end{tikzpicture}
=
\frac{1}{\cD^2}
\begin{tikzpicture}
\begin{scope}[shift={(0,-1)}]
\node[small_morphism] (lmb1) at (0,0.7) {\tiny $\lmb^1$};
\node[small_morphism] (lmb2) at (0,1.3) {\tiny $\lmb^2$};
\draw (lmb2) -- (0,2) node[pos=1,above] {$A$};
\draw (lmb2) -- (lmb1);
\draw (lmb1) -- (0,0) node[pos=1,below] {$A$};
\draw[regular] (lmb1) -- +(-150:0.6cm) node[pos=1,below left] {1};
\draw[regular] (lmb1) -- +(30:1.5cm) node[pos=1,above right] {3};
\draw[regular] (lmb2) -- +(150:0.6cm) node[pos=1,above left] {2};
\draw[thick_overline,regular] (lmb2) -- +(-30:1.5cm) node[pos=1,below right] {4};
\end{scope}
\end{tikzpicture}
\end{equation}
where the equality of diagrams follows from the COMM requirement
\eqnref{e:COMM},
and the dashed line represents a weighted sum over simples
(see Appendix).
On morphisms,
\[
\Hom_\ZZA((A,\lmb^1,\lmb^2),(A',\mu^1,\mu^2)) \ni f
  \mapsto P_{(A',\mu^1,\mu^2)} \circ f \circ P_{(A,\lmb^1,\lmb^2)}
\]

Thus we have a 2-commutative diagram
\begin{equation}
\label{e:2comm}
\begin{tikzcd}
\cA \ar[r,"\Iel"] \ar[d,"\simeq"]
& \ZZA \ar[d,"\simeq"]
\\
\ZCY(\disk) \ar[r,"\text{incl.}_*"]
& \ZCY(\punctorus)
\end{tikzcd}
\end{equation}
\end{proof}

In particular, when $\cA$ is modular,
we have an equivalence 
$\text{incl.}_* : \ZCY(\disk) \xrightarrow{\simeq} \ZCY(\punctorus)$.
Our next task is to upgrade this equivalence to an equivalence
of left $\ZCY(\Ann)$-modules.
In the rightmost figure in \eqnref{e:ell_obj}, the gray area
is a collar neighborhood of the puncture of $\punctorus$.
By \prpref{p:collared_module},
there is a (left) $\hatZCY(\Ann)$-module structure
on $\hatZCY(\punctorus)$: on objects,

\begin{equation}
\label{e:ell_obj}
\input{fig_ell_obj.tikz}
\end{equation}
while on morphisms, the module structure,
employing the equivalences of \eqnref{eqn:morphism_punctorus}
and \eqnref{e:hA_Ann}, is given as follows:

\begin{equation}
\input{fig_ell_morph.tikz}
\end{equation}
(The $D$-labeled strand originally goes around the annulus
in $\Ann \times [0,1]$;
after inserting into $\punctorus \times [0,1]$,
it wraps around like the gray area in \eqnref{e:ell_obj}).
This extends to a left $\ZCY(\Ann)$-module structure
on $\ZCY(\punctorus)$.

Similarly, there is a left $\hatZCY(\Ann)$-module structure
on $\hatZCY(\disk)$ (which extends to $\ZCY$):
\begin{equation}
\begin{tikzpicture}
\node at (0,3) {$\Hom_{\hatZCY(\Ann)}(C,C')$};
\node[small_morphism] (psi) at (0,1) {\small $\psi$};
\draw (psi) -- +(90:1cm) node[pos=1,above] {$C$};
\draw (psi) -- +(-90:1cm) node[pos=1,below] {$C'$};
\draw[midarrow_rev={0.7}] (psi) -- +(150:0.5cm) node[pos=1,above left] {$D$};
\draw[midarrow={0.7}] (psi) -- +(-30:0.5cm) node[pos=1,below right] {$D$};
\end{tikzpicture}
\begin{tikzpicture}
\node at (0,3) {$\tnsr$};
\node at (0,1) {$\tnsr$};
\end{tikzpicture}
\begin{tikzpicture}
\node at (0,3) {$\Hom_{\hatZCY(\disk)}(A,A')$};
\node[small_morphism] (ph) at (0,1) {\small $\vphi$};
\draw (ph) -- +(90:1cm) node[pos=1,above] {$A$};
\draw (ph) -- +(-90:1cm) node[pos=1,below] {$A'$};
\end{tikzpicture}
\begin{tikzpicture}
\node at (0,3) {$\to$};
\node at (0,1) {$\to$};
\end{tikzpicture}
\begin{tikzpicture}
\node at (0,3) {$\Hom_{\hatZCY(\disk)}(A,A')$};
\node[small_morphism] (ph) at (0,1) {\tiny $\vphi$};
\draw (ph) -- +(-90:1cm) node[pos=1,below] {$A'$};

\node[small_morphism] (psi) at (-0.3,0.7) {\tiny $\psi$};
\draw (psi) to[out=150,in=150] (0.3,1.3);
\draw[overline,midarrow_rev] (0.3,1.3) to[out=-30,in=-30] (psi);
\draw[overline] (psi) -- (-0.3,2) node[above] {$C$};
\draw (psi) -- (-0.3,0) node[below] {$C'$};
\node at (0.6,1) {\small $D$};

\draw[overline] (ph) -- +(90:1cm) node[pos=1,above] {$A$};
\end{tikzpicture}
\end{equation}

In light of \eqnref{e:2comm},
the following theorem is an upgrade
of \thmref{thm:modular}:

\begin{theorem}
\label{thm:modular_module}
Let $\cA$ be modular. There is an equivalence of
left $\ZCY(\Ann)$-modules
\[
  \ZCY(\disk) \simeq \ZCY(\punctorus)
\]
\end{theorem}

\begin{proof}
Under the equivalence $\ZZA \simeq \ZCY(\punctorus)$,
it is easy to see that the equivalence of \thmref{thm:modular}
can be rewritten as
\begin{align*}
\ZCY(\disk) \simeq \cA &\simeq \ZCY(\punctorus) \\
A &\mapsto
\im \Bigg(\sum_{i,j} \frac{\sqrt{d_i}\sqrt{d_j}}{\cD^2}
\begin{tikzpicture}
\draw (0.3,1) -- (0.3,-1) 
  node[pos=0,above] {\small \text{ }$i^*$}
  node[pos=1,below] {\small \text{ }$j^*$};
\draw[overline,regular] (-1,0.5) -- (0.86,-0.7)
  node[pos=0,above left] {$2$}
  node[pos=1,below right] {$4$};
\draw[overline] (0,1) -- (0,-1)
  node[pos=0,above] {\small $A$}
  node[pos=1,below] {\small $A$};
\draw[overline] (-0.3,1) -- (-0.3,-1)
  node[pos=0,above] {\small $i$}
  node[pos=1,below] {\small $j$};
\node[small_morphism] (al) at (-0.3,-0.15) {\tiny $\al$};
\draw[regular] (al) -- (-1,-0.5) node[pos=1,below left] {1};
\node[small_morphism] (al2) at (0.3,0.15) {\tiny $\al$};
\draw[regular] (al2) -- (1,0.5) node[pos=1,above right] {3};
\end{tikzpicture}
\Bigg)
\cong
\im \Bigg(
\frac{1}{\cD}
\begin{tikzpicture}
\draw[regular] (-0.86,0.5) -- (0.86,-0.5)
  node[pos=0,above left] {2}
  node[pos=1,below right] {4};
\draw[thick_overline] (0,1) -- (0,-1)
  node[pos=0,above] {$A$}
  node[pos=1,below] {$A$};
\end{tikzpicture}
\Bigg)
\end{align*}
where the isomorphism is essentially given by
$\hat{P}'_A$ and $\check{P}'_A$
from \lemref{l:IM_htr_proj} (with $\cM = \cC$),
which is easily seen to be natural in $A$.

Then we see that
\begin{align*}
\begin{tikzpicture}
\node[small_morphism] (psi) at (0,0) {\small $\psi$};
\draw (psi) -- +(90:1cm) node[pos=1,above] {$C$};
\draw (psi) -- +(-90:1cm) node[pos=1,below] {$C'$};
\draw (psi) -- +(150:0.5cm) node[pos=1,above left] {$D$};
\draw (psi) -- +(-30:0.5cm) node[pos=1,below right] {$D$};
\end{tikzpicture}
\boxtimes
\begin{tikzpicture}
\draw[regular] (-0.86,0.5) -- (0.86,-0.5)
  node[pos=0,above left] {2}
  node[pos=1,below right] {4};
\node[small_morphism] (ph) at (0,0.5) {\small $\vphi$};
\draw (ph) -- (0,1) node[pos=1,above] {$A$};
\draw[overline] (ph) -- (0,-1) node[pos=1,below] {$A'$};
\end{tikzpicture}
\mapsto
\begin{tikzpicture}
\draw (-0.93, -0.4) to[out=30,in=-30] (-0.93,0.4);
\draw (-0.8,0.6) to[out=-30,in=-150] (0.8,0.6);
\draw (0.93, 0.4) to[out=-150,in=150] (0.93,-0.4);
\draw[regular] (-0.86,0.5) -- (0.86,-0.5)
  node[pos=0,above left] {2}
  node[pos=1,below right] {4};
\node[small_morphism] (ph) at (0,0.7) {\tiny $\vphi$};
\draw (ph) -- (0,1) node[pos=1,above] {$A$};
\draw[overline] (ph) -- (0,-1) node[pos=1,below] {$A'$};
\node[small_morphism] (psi) at (-0.3,-0.45) {\tiny $\psi$};
\draw[overline] (psi) -- (-0.3,1) node[above] {$C$};
\draw (psi) -- (-0.3,-1) node[below] {$C'$};
\draw (psi) to[out=-170,in=30] (-0.8,-0.6);
\draw[overline] (psi) to[out=10,in=150] (0.8,-0.6);
\node at (-1,-0.7) {1};
\node at (1,0.7) {3};
\end{tikzpicture}
=
\begin{tikzpicture}
\draw (-0.93, -0.4) to[out=30,in=-30] (-0.93,0.4);
\draw (0.8,0.6) to[out=-150,in=150] (0.3,0.2);
\draw (0.3,0.2) to[out=-30,in=-150] (0.93,0.4);
\draw[regular] (-0.86,0.5) -- (0.86,-0.5)
  node[pos=0,above left] {2}
  node[pos=1,below right] {4};
\node[small_morphism] (ph) at (0,0.6) {\tiny $\vphi$};
\draw (ph) -- (0,1) node[pos=1,above] {$A$};
\draw[overline] (ph) -- (0,-1) node[pos=1,below] {$A'$};
\node[small_morphism] (psi) at (-0.3,-0.45) {\tiny $\psi$};
\draw[overline] (psi) -- (-0.3,1) node[above] {$C$};
\draw (psi) -- (-0.3,-1) node[below] {$C'$};
\draw (psi) to[out=-170,in=30] (-0.8,-0.6);
\draw[overline] (psi) to[out=10,in=150] (0.8,-0.6);
\node at (-1,-0.7) {1};
\node at (1,0.7) {3};
\end{tikzpicture}
=
\begin{tikzpicture}
\draw[regular] (-0.86,0.5) to[out=-30,in=120] (0.86,-0.3);
\node at (-1,0.6) {\small 2};
\node at (1,-0.4) {\small 4};
\node[small_morphism] (ph) at (0,-0.3) {\tiny $\vphi$};
\draw (ph) -- (0,-1) node[pos=1,below] {$A'$};
\node[small_morphism] (psi) at (-0.3,-0.6) {\tiny $\psi$};
\draw (psi) to[out=150,in=150] (0.3,0);
\draw[overline] (0.3,0) to[out=-30,in=-30] (psi);
\draw[overline] (psi) -- (-0.3,1) node[above] {$C$};
\draw (psi) -- (-0.3,-1) node[below] {$C'$};
\draw[overline] (ph) -- (0,1) node[pos=1,above] {$A$};
\end{tikzpicture}
\end{align*}
where we use the sliding lemma (\lemref{lem:sliding})
for both equalities, and isotopies to move the strands around.
The final diagram is what one obtains
if we apply $\psi$ to $\vphi \in \Hom_\cA(A,A')$ first
and then send it to $\ZCY(\punctorus)$.
Hence, the equivalence does respect the module structure
and we are done.
\end{proof}

Finally, we state the main result of this section:

\begin{theorem}\label{thm:cy_modular}
Let $\cA$ be modular.
Let $\Sigma$ be a connected compact oriented surface with
$b$ boundary components and genus $g$,
and let $S_{0,b} = S^2 \backslash (\DD^2)^{\sqcup b}$ be a genus 0 surface
with $b$ boundary components.
Then
\[
  \ZCY(\Sigma) \cong \ZCY(S_{0,b})
\]
In particular,
$\ZCY(\text{closed surface}) \cong \ZCY(S^2) \cong \Vect$
and $\ZCY(\text{once-punctured surface}) \cong \ZCY(\DD^2) \cong \cA$.
\end{theorem}

\begin{proof}
Suppose $g>0$, so that we can present $\Sigma$
as a connect sum $\Sigma' \# \torus$,
where $\Sigma'$ is a connected compact oriented surface with
$b$ boundary components and genus $g-1$.
We think of the connect sum as
$\Sigma = \Sigma_0' \cup_{\Ann} (\punctorus)$,
where $\Sigma_0' = \Sigma' \backslash \{pt\}$
is a punctured surface.
Then by \thmref{t:excision-skein} and \thmref{thm:modular_module},
$\ZCY(\Sigma) \simeq \ZCY(\Sigma_0') \boxtimes_{\ZCY(\Ann)} \ZCY(\punctorus)
\simeq \ZCY(\Sigma_0') \boxtimes_{\ZCY(\Ann)} \ZCY(\DD^2)
\simeq \ZCY(\Sigma_0' \cup_{\Ann} \DD^2) = \ZCY(\Sigma')$.
Thus, by induction on the genus, we have
$\ZCY(\Sigma) \simeq \ZCY(S_{0,b})$.

The final statements follow from the $b=0,1$ cases
and \corref{cor:muger}.
\end{proof}

\begin{remark}
Here is an alternative proof to \thmref{thm:cy_modular} and
\thmref{thm:modular_module},
pointed out by Jin-Cheng Guu, which avoids passage to
the elliptic Drinfeld center and may be of independent interest.
Recall the well-known equivalence
\begin{align} 
\cA \boxtimes \cA &\simeq \ZA \\
A \boxtimes B &\mapsto (A \tnsr B, c^\inv \tnsr c)
\end{align}
when $\cA$ is modular \ocite{muger}.
This can be interpreted as an equivalence
\[
\ZCY(\DD^2 \sqcup \DD^2) \simeq \ZCY(\Ann)
\]
which, using the equivalences established in \xmpref{xmp:annulus},
is given by
\[
A \boxtimes B \mapsto \im (
\begin{tikzpicture}
\draw (0.2, 0.6) -- (0.2,-0.6);
\draw[overline, regular] (-0.6, 0.3) -- (0.6, -0.3);
\draw[overline] (-0.2, 0.6) -- (-0.2,-0.6);
\node at (-0.2, -0.8) {\small $A$};
\node at (0.2, -0.8) {\small $B$};
\end{tikzpicture}
)
\]
By similar reasoning as in the proof of \thmref{thm:modular_module},
this can be shown to be a $\ZCY(\Ann)$-bimodule equivalence.
Thus, performing a surgery (replacing an annulus with two disks or vice versa)
does not affect $\ZCY$ of a surface.
In particular, this yields
\[
\ZCY(\Sigma) \simeq \cA^{\boxtimes b}
\]
where $b$ is the number of boundary components in $\Sigma$.
\end{remark}

%% file: fig_ell_obj.tikz
\begin{tikzpicture}
\draw (0,0) circle (1cm);
\draw (0,0) circle (0.7cm);
\node[dotnode] at (0.6,0.6) {};
\node at (0.9,0.8) {$C$};
\end{tikzpicture}
\boxtimes
\begin{tikzpicture}
\begin{scope}[shift={(0,-1)}]
\node[dotnode] at (1,1) {};
\node at (1.2,1) {$A$};
\draw[regular] (1.5,2) -- (1.5,2.5);
\draw (1.5,2) to[out=-90,in=180] (2,1.5);
\draw[regular] (2,1.5) -- (2.5,1.5);
\draw[regular] (2,0.5) -- (2.5,0.5);
\draw (2,0.5) to[out=180,in=90] (1.5,0);
\draw[regular] (1.5,0) -- (1.5,-0.5);
\draw[regular] (0.5,0) -- (0.5,-0.5);
\draw (0.5,0) to[out=90,in=0] (0,0.5);
\draw[regular] (0,0.5) -- (-0.5,0.5);
\draw[regular] (0,1.5) -- (-0.5,1.5);
\draw (0,1.5) to[out=0,in=-90] (0.5,2);
\draw[regular] (0.5,2) -- (0.5,2.5);
\end{scope}
\end{tikzpicture}
\mapsto
\begin{tikzpicture}
\begin{scope}[shift={(0,-1)}]
\node[dotnode] at (1,1) {};
\node at (1.2,1) {$A$};
\draw[really_thick] (1.35,1.99) -- (1.35,2.5);
\draw[regular] (1.5,2) -- (1.5,2.5);
\draw[really_thick] (1.35,2) to[out=-90,in=180] (2,1.35);
\draw (1.5,2) to[out=-90,in=180] (2,1.5);
\draw[really_thick] (1.99,1.35) -- (2.5,1.35);
\draw[regular] (2,1.5) -- (2.5,1.5);
\draw[really_thick] (1.99,0.65) -- (2.5,0.65);
\draw[regular] (2,0.5) -- (2.5,0.5);
\draw[really_thick] (2,0.65) to[out=180,in=90] (1.35,0);
\draw (2,0.5) to[out=180,in=90] (1.5,0);
\draw[really_thick] (1.35,0.01) -- (1.35,-0.5);
\draw[regular] (1.5,0) -- (1.5,-0.5);
\draw[really_thick] (0.65,0.01) -- (0.65,-0.5);
\draw[regular] (0.5,0) -- (0.5,-0.5);
\draw[really_thick] (0.65,0) to[out=90,in=0] (0,0.65);
\draw (0.5,0) to[out=90,in=0] (0,0.5);
\draw[really_thick] (0.01,0.65) -- (-0.5,0.65);
\draw[regular] (0,0.5) -- (-0.5,0.5);
\draw[really_thick] (0.01,1.35) -- (-0.5,1.35);
\draw[regular] (0,1.5) -- (-0.5,1.5);
\draw[really_thick] (0,1.35) to[out=0,in=-90] (0.65,2);
\draw (0,1.5) to[out=0,in=-90] (0.5,2);
\draw[really_thick] (0.65,1.99) -- (0.65,2.5);
\draw[regular] (0.5,2) -- (0.5,2.5);
\node[dotnode] at (0.45,0.45) {};
\node at (0.65,0.45) {$C$};
\end{scope}
\end{tikzpicture}

%% file: fig_ell_morph.tikz
\begin{tikzpicture}
\begin{scope}[shift={(0,-1.5)}]
\node at (0,3) {$\Hom_{\hatZCY(\Ann)}(C,C')$};
\node[small_morphism] (psi) at (0,1) {\small $\psi$};
\draw (psi) -- +(90:1cm) node[pos=1,above] {$C$};
\draw (psi) -- +(-90:1cm) node[pos=1,below] {$C'$};
\draw[midarrow_rev={0.7}] (psi) -- +(150:0.5cm) node[pos=1,above left] {$D$};
\draw[midarrow={0.7}] (psi) -- +(-30:0.5cm) node[pos=1,below right] {$D$};
\end{scope}
\end{tikzpicture}
\begin{tikzpicture}
\begin{scope}[shift={(0,-1.5)}]
\node at (0,3) {$\tnsr$};
\node at (0,1) {$\tnsr$};
\end{scope}
\end{tikzpicture}
\begin{tikzpicture}
\begin{scope}[shift={(0,-1.5)}]
\node at (0,3) {$\Hom_{\hatZCY(\punctorus)}(A,A')$};
\node[small_morphism] (ph) at (0,1) {\small $\vphi$};
\draw (ph) -- +(90:1cm) node[pos=1,above] {$A$};
\draw (ph) -- +(-90:1cm) node[pos=1,below] {$A'$};
\draw[midarrow_rev] (ph) -- +(-150:1cm) node[pos=1,below left] {1}
                          node[pos=0.5,below] {\small $B_1$};
\draw[midarrow_rev] (ph) -- +(150:1cm) node[pos=1,above left] {2}
                         node[pos=0.5,above] {\small $B_2$};
\draw[midarrow]  (ph) -- +(30:1cm) node[pos=1,above right] {3}
                        node[pos=0.5,above] {\small $B_1$};
\draw[midarrow] (ph) -- +(-30:1cm) node[pos=1,below right] {4}
                         node[pos=0.5,below] {\small $B_2$};
\end{scope}
\end{tikzpicture}
\begin{tikzpicture}
\begin{scope}[shift={(0,-1.5)}]
\node at (0,3) {$\to$};
\node at (0,1) {$\to$};
\end{scope}
\end{tikzpicture}
\begin{tikzpicture}
\begin{scope}[shift={(0,-1.5)}]
\node at (0,3) {$\Hom_{\hatZCY(\punctorus)}(A,A')$};
\draw[midarrow_rev] (-0.93, 0.6) to[out=30,in=-30] (-0.93,1.4); 
\draw[midarrow_rev={0.7}] (-0.8,1.6) to[out=-30,in=-150] (0.8,1.6);
\draw[midarrow_rev] (0.93, 1.4) to[out=-150,in=150] (0.93,0.6);
\node[small_morphism] (ph) at (0,1) {\tiny $\vphi$};
\draw[overline] (ph) -- +(90:1cm) node[pos=1,above] {$A$};
\draw (ph) -- +(-90:1cm) node[pos=1,below] {$A'$};
\draw (ph) -- +(-150:1cm) node[pos=1,below left] {1};
\draw (ph) -- +(150:1cm) node[pos=1,above left] {2};
\draw (ph) -- +(30:1cm) node[pos=1,above right] {3};
\draw (ph) -- +(-30:1cm) node[pos=1,below right] {4};
\node[small_morphism] (psi) at (-0.3,0.55) {\tiny $\psi$};
\draw[overline] (psi) -- (-0.3,2) node[above] {$C$};
\draw (psi) -- (-0.3,0) node[below] {$C'$};
\draw (psi) to[out=-170,in=30] (-0.8,0.4);
\draw[midarrow,overline] (psi) to[out=10,in=150] (0.8,0.4);
\node at (0.4,0.4) {\tiny $D$};
\end{scope}
\end{tikzpicture}

%% file: appendix.tex
\section{Appendix: Pivotal Multifusion Categories Conventions}\label{s:appendix}

This appendix is dedicated to notation and basic results about pivotal multifusion categories.
It is adapted from \ocite{kirillov-stringnet},
modified to accommodate for the non-spherical non-fusion case.
We also point the reader to \ocite{EGNO}*{Chapter 4}
and \ocite{ENO2005} for further reference.

Let $\cC$ be a $\kk$-linear pivotal multifusion category,
where $\kk$ is an algebraically closed field of characteristic 0.
In all our formulas and computations, we will be suppressing the
associativity and unit morphisms;
we also suppress the pivotal morphism
$V \simeq V^{**}$ when there is little cause for confusion.

We denote by $\Irr(\cC)$ the set of isomorphism classes of simple objects in $\cC$,
and denote by $\Irr_0(\cC) \subseteq \Irr(\cC)$
the subset of simple objects appearing in the direct sum decomposition of the
unit object $\one$;
it is known that $\one$ decomposes into a direct sum of distinct simples,
so $\End(\one) \cong \bigoplus_{l\in \Irr_0(\cC)} \End(\one_l)$.
We fix a representative $X_i$ for each isomorphism class $i\in \Irr(\cC)$;
abusing language, we will frequently use the same letter $i$ for 
denoting both a simple object and its isomorphism class.
Rigidity gives us an involution $-^*$ on $\Irr(\cC)$;
it is known that $l^* = l$ for $l\in \Irr_0(\cC)$.
For $l\in \Irr_0(\cC)$, we may use the notation
$\one_l := X_l$ to emphasize that it is part of the unit.

For $k,l\in \Irr_0(\cC)$,
let $\cC_{kl} := \one_k \tnsr \cC \tnsr \one_l$,
so that $\cC = \bigoplus_{k,l\in \Irr_0(\cC)} \cC_{kl}$.
Any simple $X_i$ is contained in exactly one of these $\cC_{kl}$'s,
or in other words, there are unique $k_i,l_i \in \Irr_0(\cC)$
such that $\one_{k_i} \tnsr X_i \tnsr \one_{l_i} \neq 0$.
Since $\cC_{kl}^* = \cC_{lk}$,
we have that $k_{i^*} = l_i$.

When $\cC$ is spherical fusion, the categorical dimension
is a scalar, defined as a trace, but here the non-simplicity of $\one$
and non-sphericality complicates things.
To avoid confusion, denote by $\delta: V \to V^{**}$
the pivotal morphism.
The \emph{left dimension} of an object $V \in \Obj \cC$
is the morphism
\[
  \ldim{V} := (\one
                \xrightarrow{\coev} V \tnsr V^*
                \xrightarrow{\delta \tnsr \id} V^{**} \tnsr V^*
                \xrightarrow{\ev} \one)
                \in \End(\one)
\]
Similarly, the \emph{right dimension} of $V$ is the morphism
\[
  \rdim{V} := (\one 
                \xrightarrow{\coev} \rdual{V} \tnsr V
                \xrightarrow{\id \tnsr \delta^\inv} \rdual{V}\tnsr \prescript{**}{}V
                \xrightarrow{\ev} \one)
                \in \End(\one)
\]
Note that these are \emph{vectors} and not scalars,
since $\one$ may not be simple.
It is easy to see that $\rdim{V} = \ldim{V^*} = \ldim{\rdual{V}}$.
When $\cC$ is spherical, we will drop the superscripts.

When $V = X_i$ is simple, we can interpret its left and right dimensions
as scalars as follows. We have $X_i \in \cC_{k_i l_i}$,
so $\Hom(\one, X_i \tnsr X_i^*) = \Hom(\one, \one_{k_i} \tnsr X_i \tnsr X_i^*)
\simeq \Hom(\one_{k_i}, X_i \tnsr X_i^*)$,
and likewise
$\Hom(X_i \tnsr X_i^*,\one) \simeq \Hom(X_i \tnsr X_i^*,\one_{k_i})$,
so $d_{X_i}^L$ factors through $\one_{k_i}$,
and hence we may interpret $d_{X_i}^L$ as an element of
$\End(\one_{k_i}) \cong \kk$.
Similarly, $d_{X_i}^R$ may be interpreted as an element of
$\End(\one_{l_i}) \cong \kk$.
We denote these scalar dimensions by $d_i^L, d_i^R$,
and fix square roots such that $\sqrt{d_i^L} = \sqrt{d_{i^*}^R}$.
The dimensions of simple objects are nonzero.

The \emph{dimension} of $\cC_{kl}$ is the sum
\begin{equation}\label{e:dimC}
\cD := \sum_{i\in \Irr(\cC_{kl})} d_i^R d_i^L
\end{equation}
By \ocite{ENO2005}*{Proposition 2.17},
this is the same for all pairs $k,l \in \Irr_0(\cC)$,
and by \ocite{ENO2005}*{Theorem 2.3},
they are nonzero.

We define functors $\cC^{\boxtimes n}\to \Vect$ by
\begin{align}
\label{e:vev}
\eval{V_1,\dots,V_n} &=
  \Hom_\cC(\one, V_1\otimes\dots\otimes V_n) \\
\label{e:vev_multi}
\eval{V_1,\dots,V_n}_l &=
  \Hom_\cC(\one_l, V_1\otimes\dots\otimes V_n)
  \text{ for }l\in \Irr_0(\cC) \\
  &\simeq \eval{\one_l, V_1,\ldots,V_n} \nonumber
\end{align}
for any collection $V_1,\dots, V_n$ of objects of $\cC$.
Clearly $\eval{V_1,\ldots,V_n} = \bigoplus_l \eval{V_1,\ldots,V_n}_l$.

Note that the pivotal structure gives functorial isomorphisms
\begin{equation}\label{e:cyclic}
z\colon\<V_1,\dots,V_n\>\simeq \<V_n, V_1,\dots,V_{n-1}\>
\end{equation}
such that $z^n=\id$ (see \ocite{BK}*{Section 5.3}); thus, up to a canonical
isomorphism, the space $\<V_1,\dots,V_n\>$ only depends on the cyclic order
of $V_1,\dots, V_n$.
In general, $z$ does not preserve the direct sum decomposition
of $\eval{V_1,\ldots,V_n}$ above.
For example, for a simple $X_i \in \cC_{k_i l_i}$,
we have $z : \eval{X_i, X_i^*}_{k_i} \simeq \eval{X_i^*, X_i}_{l_i}$.

We will commonly use graphic presentation of morphisms in a category, 
representing a morphism 
$W_1\otimes \dots \otimes  W_m\to V_1\otimes\dots\otimes V_n$ by a  
diagram with $m$ strands at the top, labeled by $W_1, \dots, W_m$ and $n$ strands at the bottom, labeled 
$V_1,\dots, V_n$ (Note: this differs from the convention in many other papers!). 
We will allow diagrams with  with oriented strands, 
using the convention that a strand labeled by $V$ is the same as the strands labeled by $V^*$ with opposite orientation
(suppressing isomorphisms $V\simeq V^{**}$).

We will show a morphism
$\ph\in \eval{V_1,\dots, V_n}$ by a  
round circle labeled by $\ph$
with outgoing edges labeled $V_1, \dots, V_n$ in
counter-clockwise order,
as shown in \firef{f:coloring}.
By \eqnref{e:cyclic} and the fact that $z^n=\id$,
this is unambiguous.
We will show a morphism
$\ph\in \eval{V_1,\ldots,V_n}_l$ by
a semicircle labeled by $\ph$ and $l$ as shown in
\firef{f:coloring};
in contrast with a circular node,
a semicircle imposes a strict ordering on the outgoing legs,
not just a cyclic ordering.
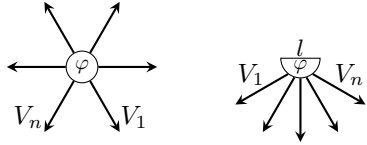
\begin{figure}[ht]
\begin{tikzpicture}
\node[morphism] (ph) at (0,0) {$\ph$};
\draw[->] (ph)-- +(240:1cm) node[pos=0.7, left] {$V_n$} ;
\draw[->] (ph)-- +(180:1cm);
\draw[->] (ph)-- +(120:1cm);
\draw[->] (ph)-- +(60:1cm);
\draw[->] (ph)-- +(0:1cm);
\draw[->] (ph)-- +(-60:1cm) node[pos=0.7, right] {$V_1$};
\end{tikzpicture}
\hspace{20pt}
\begin{tikzpicture}
\node[semi_morphism={180,\small $l$}] (ph) at (0,0) {$\ph$};
\draw[->] (ph)-- +(210:1cm) node[pos=0.7, above] {$V_1$} ;
\draw[->] (ph)-- +(240:1cm);
\draw[->] (ph)-- +(270:1cm);
\draw[->] (ph)-- +(300:1cm);
\draw[->] (ph)-- +(-30:1cm) node[pos=0.7, above] {$V_n$};
\end{tikzpicture}
\caption{Labeling of colored graphs}\label{f:coloring}
\end{figure}

We have a natural composition map 
\begin{equation}\label{e:composition}
\begin{aligned}
 \<V_1,\dots,V_n, X\>\otimes\<X^*, W_1,\dots,
W_m\>&\to\<V_1,\dots,V_n, W_1,\dots, W_m\>\\
\ph\otimes\psi\mapsto \ph\cc{X}\psi= \ev_{X^*}\circ (\ph\otimes\psi)
\end{aligned}
\end{equation}
where $\ev_{X^*}\colon X\otimes  X^*\to \one$ is the evaluation morphism
(the pivotal structure is suppressed).

Repeated applications of the composition map above gives us
a non-degenerate pairing
\begin{equation}\label{e:pairing}
\eval{V_1,\dots,V_n}\otimes \eval{V_n^*,\dots,V_1^*}\to \End(\one)
\end{equation}
More precisely, when restricted to the subspaces,
\begin{equation}\label{e:pairing_sub}
\eval{V_1,\dots,V_n}_k \otimes \eval{V_n^*,\dots,V_1^*}_l \to \End(\one)
\end{equation}
the pairing is 0 if $k\neq l$, and is non-degenerate if $k=l$.
The pairing is illustrated below for
$\ph_1 \in \eval{V_1,\dots,V_n}_k,
\ph_2 \in \eval{V_n^*,\dots,V_1^*}_l$:
\[
(\ph_1, \ph_2) = 
\begin{tikzpicture}
\begin{scope}[shift={(0,0.3)}]
\node[semi_morphism={180,\small $k$}] (ph1) at (0,0) {$\ph_1$};
\node[semi_morphism={180,\small $l$}] (ph2) at (1,0) {$\ph_2$};
\draw (ph1) to[out=-40,in=-140] (ph2);
\draw (ph1) ..controls (0,-0.5) and (1,-0.5) .. (ph2);
\draw (ph1) ..controls (-0.5,-0.8) and (1.5,-0.8) .. (ph2);
\draw (ph1) ..controls (-1.5, -1) and (2.5, -1) .. (ph2);
\end{scope}
\end{tikzpicture}
\overset{\text{if }\cC\text{ not spherical}}
  {\not\equiv}
\begin{tikzpicture}
\begin{scope}[shift={(0,0.3)}]
\node[semi_morphism={180,\small $k$}] (ph1) at (0,0) {$\ph_1$};
\node[semi_morphism={180,\small $l$}] (ph2) at (1,0) {$\ph_2$};
\draw (ph1) to[out=-40,in=-140] (ph2);
\draw (ph1) ..controls (0,-0.5) and (1,-0.5) .. (ph2);
\draw (ph1) ..controls (-0.5,-0.8) and (1.5,-0.8) .. (ph2);
\draw (ph1) ..controls (-1, 0) and (-0.3, 0.5) ..
(0.5, 0.5) .. controls (1.3, 0.5) and (2, 0) .. (ph2);
\end{scope}
\end{tikzpicture}
= (z^\inv \cdot \ph_1, z \cdot \ph_2) 
\]
Thus, we have functorial isomorphisms
\begin{equation}\label{e:dual}
\<V_1,\dots,V_n\>^*\simeq \<V_n^*,\dots,V_1^*\>
\end{equation}

When $\cC$ is spherical,
this pairing is compatible with the cyclic permutations \eqref{e:cyclic},
in the sense that
$(\ph_1,\ph_2) = (z \cdot \ph_1, z^\inv \cdot \ph_2)$.
Compatibility fails when $\cC$ is not spherical;
for example, it is easy to see that for
$\ph_1 = \ph_2 = \coev_{X_i} \in \eval{X_i,X_i^*}$,
one has $(\ph_1,\ph_2) = d_i^L$,
while for
$z \cdot \ph_1 = z^\inv \cdot \ph_2 = \coev_{X_i^*} \in \eval{X_i^*,X_i}$,
one has instead $(z \cdot \ph_1, z^\inv \ph_2) = d_i^R$.

\begin{lemma}
\label{l:pairing_property}
For $\ph \in \eval{V_1,\dots,V_n}_l,
\ph' \in \eval{V_n^*,\ldots,V_1^*}_l,
\psi \in \eval{W_n^*,\dots,W_1^*}_l$,
and
$f \in \Hom(V_1 \tnsr \cdots \tnsr V_n, W_1 \tnsr \cdots \tnsr W_n)$,
we have
\begin{align}
  (\ph,\ph') &= (\ph',\ph) \\
  (f \circ \ph_1, \ph_2) &= (\ph_1, f^* \circ \ph_2)
\end{align}
\end{lemma}

\begin{proof}
  Straightforward from definitions.
\end{proof}

We will make two additional conventions related to the graphic presentation of morphisms.

\begin{notation}\label{n:dashed}
A dashed line in the picture stands for the sum of all colorings of an edge by 
simple objects $i$, each taken with coefficient $d_i^R$:
     \begin{equation} \label{e:regular_color}
          \begin{tikzpicture}
        \draw[->, regular] (0, 0.5)--(0, -0.5);
       \end{tikzpicture}
      \;
      =\sum_{i\in \Irr(\cC)} d_i^R \quad 
      \begin{tikzpicture}
        \draw[->] (0, 0.5)--(0, -0.5) node[pos=0.8, right] {$i$};
       \end{tikzpicture}
     \end{equation}
\end{notation}
When $\cC$ is spherical, the orientation of such a dashed line is irrelevant.

\begin{notation}\label{n:summation}
Let $\cC$ be spherical.
If a figure contains a pair of circles, one with outgoing edges labeled $V_1,\dots, V_n$ and
the other with edges labeled $V_n^*,\dots, V_1^*$ , and the vertices  are
labeled by the same letter $\al$  (or $\be$, or \dots)
it will stand for summation over the dual bases:
\begin{equation}\label{e:summation_convention}
\begin{tikzpicture}
\node[morphism] (ph) at (0,0) {$\al$};
\draw[->] (ph)-- +(240:1cm) node[pos=0.7, left] {$V_n$};
\draw[->] (ph)-- +(180:1cm);
\draw[->] (ph)-- +(120:1cm);
\draw[->] (ph)-- +(60:1cm);
\draw[->] (ph)-- +(0:1cm);
\draw[->] (ph)-- +(-60:1cm) node[pos=0.7, right] {$V_1$};
\node[morphism] (ph') at (3,0) {$\al$};
\draw[->] (ph')-- +(240:1cm) node[pos=0.7, left]  {$V^{*}_1$};
\draw[->] (ph')-- +(180:1cm);
\draw[->] (ph')-- +(120:1cm);
\draw[->] (ph')-- +(60:1cm);
\draw[->] (ph')-- +(0:1cm);
\draw[->] (ph')-- +(-60:1cm) node[pos=0.7, right] {$V^{*}_n$};
\end{tikzpicture}
\quad := \sum_\al\quad
\begin{tikzpicture}
\node[morphism] (ph) at (0,0) {$\ph_\al$};
\draw[->] (ph)-- +(240:1cm) node[pos=0.7, left] {$V_n$};
\draw[->] (ph)-- +(180:1cm);
\draw[->] (ph)-- +(120:1cm);
\draw[->] (ph)-- +(60:1cm);
\draw[->] (ph)-- +(0:1cm);
\draw[->] (ph)-- +(-60:1cm) node[pos=0.7, right] {$V_1$};
\node[morphism] (ph') at (3,0) {$\ph {}^\al$};
\draw[->] (ph')-- +(240:1cm) node[pos=0.7, left] {$V^{*}_1$};
\draw[->] (ph')-- +(180:1cm);
\draw[->] (ph')-- +(120:1cm);
\draw[->] (ph')-- +(60:1cm);
\draw[->] (ph')-- +(0:1cm);
\draw[->] (ph')-- +(-60:1cm) node[pos=0.7, right] {$V^{*}_n$};
\end{tikzpicture}
\end{equation}
where $\ph_\al\in \<V_1,\dots, V_n\>$, $\ph^\al\in \<V_n^*,\dots, V_1^*\>$
are dual bases with respect to pairing \eqref{e:pairing}.

When $\cC$ is not spherical, the pairing is no longer
compatible with $z$ from \eqnref{e:cyclic},
so such notation can only make sense with semicircles:
\begin{equation}
\begin{tikzpicture}
  \node[semi_morphism={180,}] (al1) at (0,0.3) {$\al$};
\draw[->] (al1) -- +(-150:1cm) node[pos=0.8, left] {$V_1$};
\draw[->] (al1) -- +(-110:1cm);
\draw[->] (al1) -- +(-70:1cm);
\draw[->] (al1) -- +(-30:1cm) node[pos=0.8, right] {$V_n$};
\node[semi_morphism={180,}] (al2) at (3,0.3) {$\al$};
\draw[->] (al2) -- +(-150:1cm) node[pos=0.8, left] {$V_n^*$};
\draw[->] (al2) -- +(-110:1cm);
\draw[->] (al2) -- +(-70:1cm);
\draw[->] (al2) -- +(-30:1cm) node[pos=0.8, right] {$V_1^*$};
\end{tikzpicture}
\quad := \sum_{\al, l} \quad
\begin{tikzpicture}
  \node[semi_morphism={180,\small $l$}] (al1) at (0,0.3) {$\ph_\al$};
\draw[->] (al1) -- +(-150:1cm) node[pos=0.8, left] {$V_1$};
\draw[->] (al1) -- +(-110:1cm);
\draw[->] (al1) -- +(-70:1cm);
\draw[->] (al1) -- +(-30:1cm) node[pos=0.8, right] {$V_n$};
\node[semi_morphism={180,\small $l$}] (al2) at (3,0.3) {\small $\ph^\al$};
\draw[->] (al2) -- +(-150:1cm) node[pos=0.8, left] {$V_n^*$};
\draw[->] (al2) -- +(-110:1cm);
\draw[->] (al2) -- +(-70:1cm);
\draw[->] (al2) -- +(-30:1cm) node[pos=0.8, right] {$V_1^*$};
\end{tikzpicture}
\end{equation}
where $\ph_\al\in \eval{V_1,\dots, V_n}_l$, $\ph^\al\in \eval{V_n^*,\dots, V_1^*}_l$
are dual bases with respect to the pairing \eqref{e:pairing}.

\end{notation}

The following lemma illustrates the use of the notation above.

\begin{lemma}\label{l:summation}
For any $V_1,\ldots,V_n \in \cC$, we have 
$$
\begin{tikzpicture}
\draw[->] (-0.4,1.2) node[above] {$V_1$} -- (-0.4,-1.2) node[below] {$V_1$};
\draw[->] (0,1.2) node[above] {$\cdots$} -- (0,-1.2) node[below] {$\cdots$};
\draw[->] (0.4,1.2) node[above] {$V_n$} -- (0.4,-1.2) node[below] {$V_n$};
\end{tikzpicture}
\quad = \quad 
\sum_{i\in \Irr(\cC)} d_i^R \quad 
\begin{tikzpicture}
\node[semi_morphism={90,}] (T) at (0, 0.5) {$\al$};
\node[semi_morphism={90,}] (B) at (0, -0.5) {$\al$};
\draw (T) -- +(-0.3, 0.7) node[above left] {$V_1$};
\draw (T) -- +(0, 0.7) node[above] {$\cdots$};
\draw (T) -- +(0.3, 0.7) node[above right] {$V_n$};
\draw[->] (B) -- +(-0.3, -0.7) node[below left] {$V_1$};
\draw[->] (B) -- +(0, -0.7) node[below] {$\cdots$};
\draw[->] (B) -- +(0.3, -0.7) node[below right] {$V_n$};
\draw[midarrow={0.6}] (T)--(B) node[pos=0.5, right] {$i$};
\end{tikzpicture}
\quad = \quad 
\begin{tikzpicture}
\node[semi_morphism={90,}] (T) at (0, 0.5) {$\al$};
\node[semi_morphism={90,}] (B) at (0, -0.5) {$\al$};
\draw (T) -- +(-0.3, 0.7) node[above left] {$V_1$};
\draw (T) -- +(0, 0.7) node[above] {$\cdots$};
\draw (T) -- +(0.3, 0.7) node[above right] {$V_n$};
\draw[->] (B) -- +(-0.3, -0.7) node[below left] {$V_1$};
\draw[->] (B) -- +(0, -0.7) node[below] {$\cdots$};
\draw[->] (B) -- +(0.3, -0.7) node[below right] {$V_n$};
\draw[->, regular] (T)--(B);
\end{tikzpicture}
\quad = \quad 
\sum_{i\in \Irr(\cC)} d_i^L \quad 
\begin{tikzpicture}
\node[semi_morphism={270,}] (T) at (0, 0.5) {$\al$};
\node[semi_morphism={270,}] (B) at (0, -0.5) {$\al$};
\draw (T) -- +(-0.3, 0.7) node[above left] {$V_1$};
\draw (T) -- +(0, 0.7) node[above] {$\cdots$};
\draw (T) -- +(0.3, 0.7) node[above right] {$V_n$};
\draw[->] (B) -- +(-0.3, -0.7) node[below left] {$V_1$};
\draw[->] (B) -- +(0, -0.7) node[below] {$\cdots$};
\draw[->] (B) -- +(0.3, -0.7) node[below right] {$V_n$};
\draw[midarrow={0.6}] (T)--(B) node[pos=0.5, right] {$i$};;
\end{tikzpicture}
\quad = \quad 
\begin{tikzpicture}
\node[semi_morphism={270,}] (T) at (0, 0.5) {$\al$};
\node[semi_morphism={270,}] (B) at (0, -0.5) {$\al$};
\draw (T) -- +(-0.3, 0.7) node[above left] {$V_1$};
\draw (T) -- +(0, 0.7) node[above] {$\cdots$};
\draw (T) -- +(0.3, 0.7) node[above right] {$V_n$};
\draw[->] (B) -- +(-0.3, -0.7) node[below left] {$V_1$};
\draw[->] (B) -- +(0, -0.7) node[below] {$\cdots$};
\draw[->] (B) -- +(0.3, -0.7) node[below right] {$V_n$};
\draw[<-, regular] (T)--(B);
\end{tikzpicture}
$$
\end{lemma}
Proof of this lemma is straightforward:
first show it for simple $X$, then for direct sums;
interested reader can find a proof
for spherical $\cC$ in \ocite{kirillov-stringnet}.

\begin{lemma}
  \label{lem:sliding}
The following is a generalization of the ``sliding lemma'':
\[
\begin{tikzpicture}
  \draw[regular, midarrow={0.5}] (0,0) circle(0.4cm);
 \path[subgraph] (0,0) circle(0.3cm); 
 \draw (0,1.2)..controls +(-90:0.8cm) and +(90:0.4cm) ..
                 (-0.6, 0) ..controls +(-90:0.4cm) and +(90:0.8cm) ..
                 (0,-1.2);
\end{tikzpicture}
\quad=\quad
\begin{tikzpicture}
  \draw[regular, midarrow={0.5}] (0,0) circle(0.4cm);
  \path[subgraph] (0,0) circle(0.3cm); 
  \draw (0,1.2)..controls +(-90:0.8cm) and +(90:0.4cm) ..
                 (0.6, 0) ..controls +(-90:0.4cm) and +(90:0.8cm) ..
                 (0,-1.2);
\end{tikzpicture}
\]
These relations hold regardless of the contents of the shaded region.
\end{lemma}
\begin{proof}
\[
 \begin{tikzpicture}
   \draw[regular, midarrow={0.5}] (0,0) circle(0.4cm);
   \path[subgraph] (0,0) circle(0.3cm); 
   \draw (0,1.2)..controls +(-90:0.8cm) and +(90:0.4cm) ..
                   (-0.6, 0) ..controls +(-90:0.4cm) and +(90:0.8cm) ..
                   (0,-1.2);
 \end{tikzpicture}
\quad=\quad
\begin{tikzpicture}
  \draw[regular, midarrow={0.53}] (0,0) circle(0.45cm);
  \draw[regular, midarrow={0.03}] (0,0) circle(0.45cm);
  \node[semi_morphism={0,}] (top) at (0,0.4) {$\al$};
  \node[semi_morphism={180,}] (bot) at (0,-0.4) {$\al$};
  \path[subgraph] circle(0.3cm); 
  \draw (0,1.2)--(top) (bot)--(0,-1.2);    
\end{tikzpicture}
\quad=\quad 
  \begin{tikzpicture}
    \draw[regular, midarrow={0.55}] (0,0) circle(0.4cm);
    \path[subgraph] (0,0) circle(0.3cm); 
    \draw (0,1.2)..controls +(-90:0.8cm) and +(90:0.4cm) ..
                   (0.6, 0) ..controls +(-90:0.4cm) and +(90:0.8cm) ..
                   (0,-1.2);
  \end{tikzpicture}
\]
where we use \lemref{l:summation} in the equalities.
See also \ocite{kirillov-stringnet}*{Corollary 3.5}.
Note this trick doesn't work when the circle is oriented the other way
(unless of course if $\cC$ is spherical).
\end{proof}

\begin{lemma} \label{l:dashed_circle}
\[
\frac{1}{|\Irr_0(\cC)|\cD}
\begin{tikzpicture}
\draw[regular, midarrow] (0,0) circle(0.4cm);
\end{tikzpicture}
=
\id_\one
=
\frac{1}{|\Irr_0(\cC)|\cD}
\sum_{i\in \Irr(\cC)} d_i^L
\begin{tikzpicture}
\draw[midarrow_rev={0.1}] (0,0) circle(0.4cm);
\node at (0.5,0.3) {\small $i$};
\end{tikzpicture}
\]
\end{lemma}
\begin{proof}
Let $\Irr^{kl} = \Irr(\cC_{kl})$,
and let $\Irr^{k*} := \bigcup_l \Irr(\cC_{kl})$,
i.e. the set of simples $X_i$ such that $\one_k \tnsr X_i = X_i$.
Then
\begin{gather*}
\begin{tikzpicture}
\draw[regular, midarrow] (0,0) circle(0.4cm);
\end{tikzpicture}
=
\sum_{k\in \Irr_0(\cC)} \sum_{i \in \Irr^{k*}(\cC)} d_i^R
\hspace{3pt}
\begin{tikzpicture}
\draw[midarrow={0.5}] (0,0) circle(0.4cm);
\node at (0.5,-0.2) {$i$};
\end{tikzpicture}
=
\sum_{k\in \Irr_0(\cC)} \sum_{i \in \Irr^{k*}} d_i^R d_i^L \id_{\one_k} \\
=
\sum_{k\in \Irr_0(\cC)} \sum_{l \in \Irr_0(\cC)}
\sum_{i \in \Irr^{kl}} d_i^R d_i^L \id_{\one_k}
=
\sum_{k\in \Irr_0(\cC)} \sum_{l \in \Irr_0(\cC)} \cD \id_{\one_k}
=
|\Irr_0(\cC)|\cD \id_\one
\end{gather*}
The second equality is proved in a similar manner.
\end{proof}

The following lemma is used to prove that \firef{f:I(M)} is a half-braiding
and the functor $G$ in the proof of \thmref{t:htr} respects composition:

\begin{lemma}
\label{l:halfbrd}
For $X \in \Obj \cC$, define the following element $\Gamma_X$ of
$\Hom(X \tnsr X_i, X_j) \tnsr \Hom(X_i^*, X_j^* \tnsr X)$:
\begin{equation}
\Gamma_X :=
\begin{tikzpicture}
\node[semi_morphism={90,}] (lf) at (0,0) {$\ov{\al}$};
\draw[midarrow_rev] (lf) -- +(90:1cm) node[pos=0.5, right] {$i$};
\draw[midarrow] (lf) -- +(270:1cm) node[pos=0.5, right] {$j$};
\draw[midarrow_rev] (lf) -- +(150:1cm) node[above left] {$X$};
\node[semi_morphism={270,}] (rt) at (2,0) {$\ov{\al}$};
\draw[midarrow] (rt) -- +(90:1cm) node[pos=0.5, left] {$i$};
\draw[midarrow_rev] (rt) -- +(270:1cm) node[pos=0.5, left] {$j$};
\draw[midarrow] (rt) -- +(-30:1cm) node[below right] {$X$};
\end{tikzpicture}
\quad := \quad
\sum_{i,j \in \Irr(\cC)} \sqrt{d_i^R}\sqrt{d_j^R}
\begin{tikzpicture}
\node[semi_morphism={90,}] (lf) at (0,0) {$\al$};
\draw[midarrow_rev] (lf) -- +(90:1cm) node[pos=0.5, right] {$i$};
\draw[midarrow] (lf) -- +(270:1cm) node[pos=0.5, right] {$j$};
\draw[midarrow_rev] (lf) -- +(150:1cm) node[above left] {$X$};
\node[semi_morphism={270,}] (rt) at (2,0) {$\al$};
\draw[midarrow] (rt) -- +(90:1cm) node[pos=0.5, left] {$i$};
\draw[midarrow_rev] (rt) -- +(270:1cm) node[pos=0.5, left] {$j$};
\draw[midarrow] (rt) -- +(-30:1cm) node[below right] {$X$};
\end{tikzpicture}
\end{equation}

$\Gamma_X$ satisfies the following properties:

\begin{enumerate}
\item $\Gamma_-$ respects tensor products:
\begin{equation}
\begin{tikzpicture}
\node[semi_morphism={90,}] (lf) at (0,0) {$\ov{\al}$};
\draw[midarrow_rev] (lf) -- +(90:1cm) node[pos=0.5, right] {$i$};
\draw[midarrow] (lf) -- +(270:1cm) node[pos=0.5, right] {$j$};
\draw[midarrow_rev] (lf) -- +(150:1cm) node[above left] {$X$};
\draw[midarrow_rev] (lf) -- +(120:1cm) node[above] {$Y$};
\node[semi_morphism={270,}] (rt) at (2,0) {$\ov{\al}$};
\draw[midarrow] (rt) -- +(90:1cm) node[pos=0.5, left] {$i$};
\draw[midarrow_rev] (rt) -- +(270:1cm) node[pos=0.5, left] {$j$};
\draw[midarrow] (rt) -- +(-60:1cm) node[below] {$X$};
\draw[midarrow] (rt) -- +(-30:1cm) node[below right] {$Y$};
\end{tikzpicture}
=
\begin{tikzpicture}
\node[semi_morphism={90,}] (lf) at (0,0.4) {\tiny $\ov{\al}$};
\node[semi_morphism={90,}] (lf2) at (0,-0.4) {\tiny $\ov{\beta}$};
\draw[midarrow_rev] (lf) -- +(90:0.7cm) node[pos=0.5, right] {$i$};
\draw[midarrow={0.7}] (lf) -- (lf2) node[pos=0.5, right] {$j$};
\draw[midarrow] (lf2) -- +(270:0.7cm) node[pos=0.5, right] {$k$};
\draw[midarrow_rev] (lf) -- +(150:1cm) node[above left] {$Y$};
\draw[midarrow_rev] (lf2) -- +(150:1cm) node[above left] {$X$};
\node[semi_morphism={270,}] (rt) at (2,0.4) {\tiny $\ov{\al}$};
\node[semi_morphism={270,}] (rt2) at (2,-0.4) {\tiny $\ov{\beta}$};
\draw[midarrow] (rt) -- +(90:0.7cm) node[pos=0.5, left] {$i$};
\draw[midarrow_rev={0.7}] (rt) -- (rt2) node[pos=0.5, left] {$j$};
\draw[midarrow_rev] (rt2) -- +(270:0.7cm) node[pos=0.5, left] {$k$};
\draw[midarrow_rev] (rt) -- +(-30:1cm) node[below right] {$Y$};
\draw[midarrow_rev] (rt2) -- +(-30:1cm) node[below right] {$X$};
\end{tikzpicture}
\end{equation}

\item $\Gamma_X$ is natural in $X$: for $f: X \to Y$,
\begin{equation}
\begin{tikzpicture}
\node[semi_morphism={90,}] (lf) at (0,0) {$\ov{\al}$};
\draw[midarrow_rev] (lf) -- +(90:1cm) node[pos=0.5, right] {$i$};
\draw[midarrow] (lf) -- +(270:1cm) node[pos=0.5, right] {$j$};
\node[small_morphism] (f) at (-0.5,0.3) {\tiny $f$};
\draw[midarrow_rev={0.7}] (lf) -- (f);
\draw[midarrow_rev] (f) -- (-1,0.6) node[left] {$X$};
\node[semi_morphism={270,}] (rt) at (2,0) {$\ov{\al}$};
\draw[midarrow] (rt) -- +(90:1cm) node[pos=0.5, left] {$i$};
\draw[midarrow_rev] (rt) -- +(270:1cm) node[pos=0.5, left] {$j$};
\draw[midarrow] (rt) -- +(-30:1cm) node[below right] {$Y$};
\end{tikzpicture}
\quad = \quad
\begin{tikzpicture}
\node[semi_morphism={90,}] (lf) at (0,0) {$\ov{\beta}$};
\draw[midarrow_rev] (lf) -- +(90:1cm) node[pos=0.5, right] {$i$};
\draw[midarrow] (lf) -- +(270:1cm) node[pos=0.5, right] {$j$};
\draw[midarrow_rev] (lf) -- +(150:1cm) node[above left] {$X$};
\node[semi_morphism={270,}] (rt) at (2,0) {$\ov{\beta}$};
\draw[midarrow] (rt) -- +(90:1cm) node[pos=0.5, left] {$i$};
\draw[midarrow_rev] (rt) -- +(270:1cm) node[pos=0.5, left] {$j$};
\node[small_morphism] (f) at (2.5,-0.3) {\tiny $f$};
\draw[midarrow={0.7}] (rt) -- (f);
\draw[midarrow] (f) -- (3,-0.6) node[right] {$Y$};
\end{tikzpicture}
\end{equation}

\end{enumerate}
\end{lemma}

\begin{proof}
The second property follows from \lemref{l:pairing_property}.
The first property follows from using \lemref{l:pairing_property}
to ``pull'' $\ov{\al}$ through $\ov{\beta}$,
then use \lemref{l:summation}
to contract the $k$ strand.
\end{proof}

Finally, we give a proof of \thmref{t:center}:

\begin{proof}[Proof of \thmref{t:center}]
\label{pf:t:center}
This is essentially the same as when $\cC$ is spherical,
but we provide it to assuage any doubts that the non-sphericality,
manifested in requiring semicircular morphisms $\al$,
does not lead to problems.

Let us check that the morphism on the left side of \eqnref{e:adj_isom_2}
intertwines half-braidings:
\begin{equation*}
\sum_{j\in \Irr(\cC)} \sqrt{d_j^R}
\begin{tikzpicture}
\node[small_morphism] (ga2) at (0,0.9) {$\ga$};
\node[small_morphism] (ga) at (0,0.3) {$\ga$};
\node[small_morphism] (f) at (0,-0.3) {\small $f$};
\draw (ga) -- (ga2);
\draw (ga2) -- (0,1.5) node[above] {$M_1$};
\draw (ga) -- (f);
\draw (f) -- (0,-1.5) node[below] {$M_2$};
\draw (ga2) to[out=0,in=-90] (1.5,1.5) node[above] {$X$};
\draw (ga2) to[out=180,in=90] (-1.5,-1.5) node[below] {$X$};
\draw[midarrow] (ga) to[out=180,in=90] (-1,-1.5);
\draw[midarrow_rev] (ga) to[out=0,in=90] (1,-1.5);
\node at (0.9,0.1) {$j$};
\end{tikzpicture}
=
\sum_{i,j\in \Irr(\cC)} d_i^R \sqrt{d_j^R}
\begin{tikzpicture}
\node[small_morphism] (ga2) at (0,0.9) {$\ga$};
\node[small_morphism] (ga) at (0,0.3) {$\ga$};
\node[small_morphism] (f) at (0,-0.3) {\small $f$};
\draw (ga) -- (ga2);
\draw (ga2) -- (0,1.5) node[above] {$M_1$};
\draw (ga) -- (f);
\draw (f) -- (0,-1.5) node[below] {$M_2$};
\draw (ga2) to[out=0,in=-90] (1.5,1.5) node[above] {$X$};
\draw[midarrow_rev] (ga) to[out=0,in=90] (1,-1.5);
\node at (0.9,0.1) {$j$};
\node[semi_morphism={90,}] (al1) at (-1,-0.3) {\tiny $\al$};
\node[semi_morphism={90,}] (al2) at (-1,-0.9) {\tiny $\al$};
\draw (ga) to[out=180,in=80] (al1);
\draw (ga2) to[out=180,in=110] (al1);
\draw[midarrow={0.7}] (al1) to[out=-110,in=110] (al2);
\node at (-1.3,-0.6) {$i$};
\draw (al2) -- (-1,-1.5);
\draw (al2) to[out=-150,in=90] (-1.5,-1.5) node[below] {$X$};
\end{tikzpicture}
=
\sum_{i\in \Irr(\cC)} \sqrt{d_i^R}
\begin{tikzpicture}
\node[small_morphism] (ga) at (0,0.5) {$\ga$};
\node[small_morphism] (f) at (0,-0.3) {\small $f$};
\draw (ga) -- (0,1.5) node[above] {$M_1$};
\draw (ga) -- (f);
\draw (f) -- (0,-1.5) node[below] {$M_2$};
\node[semi_morphism={90,}] (al1) at (-1,-0.3) {\tiny $\albar$};
\node[semi_morphism={270,}] (al2) at (1,-0.3) {\tiny $\albar$};
\draw (al2) to[out=0,in=-90] (1.5,1.5) node[above] {$X$};
\draw (al1) to[out=180,in=90] (-1.5,-1.5) node[below] {$X$};
\draw[midarrow] (ga) to[out=180,in=90] (al1);
\node at (-0.8, 0.6) {$i$};
\draw (ga) to[out=0,in=90] (al2);
\draw (al1) -- (-1,-1.5);
\draw (al2) -- (1,-1.5);
\end{tikzpicture}
\end{equation*}

In the first equality, we use \lemref{l:summation};
in the second equality, we use the naturality of $\ga$
to pull the top $\al$ to the right,
and absorb the $\sqrt{d_i^R}\sqrt{d_j^R}$ factor
into $\al$ to get $\albar$.

Next we check that applying \eqnref{e:adj_isom} then \eqnref{e:adj_isom_2}
is the identity map.
Let $m_i = 1$ if $i\in \Irr_0(\cC)$, 0 otherwise.
In the following diagrams,
we implicitly sum lowercase latin alphabets
over $\Irr(\cC)$.
Then the composition is the map

\begin{gather*}
\begin{tikzpicture}
\node[morphism] (ph) at (0,0) {\small $\ph_i$};
\draw (ph) -- (0,1.5) node[above] {$M_1$};
\draw (ph) -- (0,-1.5) node[below] {$M_2$};
\draw[midarrow] (ph) to[out=180,in=90] (-1,-1.5);
\node at (-0.9,-0.1) {$i$};
\draw[midarrow_rev] (ph) to[out=0,in=90] (1,-1.5);
\node at (0.9,-0.1) {$i$};
\end{tikzpicture}
\mapsto
m_i\sqrt{d_j^R}
\begin{tikzpicture}
  \draw[midarrow_rev={0.6}] (0,0) circle(0.4cm);
\node at (-0.4,-0.4) {$i$};
\node[small_morphism] (ph) at (0,-0.4) {\small $\ph_i$};
\node[small_morphism] (ga1) at (0,0.9) {$\ga$};
\node[small_morphism] (ga2) at (0,0.4) {$\ga$};
\draw (ga1) -- (0,1.5) node[above] {$M_1$};
\draw (ga1) -- (ga2);
\draw (ga2) -- (ph);
\draw (ph) -- (0,-1.5) node[below] {$M_2$};
\draw[midarrow] (ga1) to[out=180,in=90] (-1,-1.5);
\node at (1.15,0) {$j$};
\draw[midarrow_rev] (ga1) to[out=0,in=90] (1,-1.5);
\end{tikzpicture}
=
m_i \sqrt{d_j^R} d_k^R
\begin{tikzpicture}
\node[semi_morphism={90,}] (al1) at (-1,0) {\tiny $\al$};
\node[semi_morphism={270,}] (al2) at (1,0) {\tiny $\al$};
\node[small_morphism] (ga) at (0,0.9) {$\ga$};
\node[small_morphism] (ph) at (0,0) {\small $\ph_i$};
\draw[midarrow] (ph) to[out=180,in=-80] (al1);
\node at (-0.5,-0.3) {$i$};
\draw (ph) to[out=0,in=-100] (al2);
\draw[midarrow] (ga) to[out=180,in=90] (al1);
\node at (-0.7,1) {$k$};
\draw (ga) to[out=0,in=90] (al2);
\draw (ga) -- (0,1.5) node[above] {$M_1$}; 
\draw (ga) -- (ph);
\draw (ph) -- (0,-1.5) node[below] {$M_2$};
\draw[midarrow] (al1) to[out=-100,in=90] (-1,-1.5);
\node at (-0.9,-0.8) {$j$};
\draw (al2) to[out=-80,in=90] (1,-1.5);
\end{tikzpicture}
=
m_i \sqrt{d_j^R} d_k^R
\sqrt{d_i^R} \sqrt{d_l^R}
\begin{tikzpicture}
\node[semi_morphism={90,}] (al1) at (-1,-0.5) {\tiny $\al$};
\node[semi_morphism={270,}] (al2) at (1,-0.5) {\tiny $\al$};
\node[small_morphism] (ph) at (0,0.6) {\small $\ph_l$};
\node[semi_morphism={90,}] (bt1) at (-0.5,0) {\tiny $\beta$};
\node[semi_morphism={270,}] (bt2) at (0.5,0) {\tiny $\beta$};
\draw[midarrow] (ph) to[out=180,in=90] (bt1);
\node at (-0.5,0.6) {$l$};
\draw (ph) to[out=0,in=90] (bt2);
\draw[midarrow] (bt1) to[out=180,in=90] (al1);
\node at (-0.8,0.2) {$k$};
\draw (bt2) to[out=0,in=90] (al2);
\draw[midarrow] (bt1) to[out=-90,in=-80] (al1);
\node at (-0.5,-0.6) {$i$};
\draw (bt2) to[out=-90,in=-100] (al2);
\draw (ph) -- (0,1.5) node[above] {$M_1$}; 
\draw (ph) -- (0,-1.5) node[below] {$M_2$};
\draw[midarrow] (al1) to[out=-100,in=90] (-1,-1.5);
\node at (-0.9,-1.1) {$j$};
\draw (al2) to[out=-80,in=90] (1,-1.5);
\end{tikzpicture}
\\
=
m_i \sqrt{d_j^R} \sqrt{d_i^R} \sqrt{d_l^R}
\begin{tikzpicture}
\node[small_morphism] (ph) at (0,0.6) {\small $\ph_l$};
\node[semi_morphism={90,}] (bt1) at (-0.5,0) {\tiny $\beta$};
\node[semi_morphism={270,}] (bt2) at (0.5,0) {\tiny $\beta$};
\draw[midarrow] (ph) to[out=180,in=90] (bt1);
\node at (-0.5,0.6) {$l$};
\draw (ph) to[out=0,in=90] (bt2);
\draw[midarrow] (bt1)
  ..controls (-0.4,-0.5) and (-0.6,-0.8) .. (-0.8,-0.6)
  ..controls (-1,-0.4) and (-0.7,-0.2) .. (bt1);
\draw[midarrow_rev] (bt2)
  ..controls (0.4,-0.5) and (0.6,-0.8) .. (0.8,-0.6)
  ..controls (1,-0.4) and (0.7,-0.2) .. (bt2);
\node at (-0.5,-0.7) {$i$};
\draw (ph) -- (0,1.5) node[above] {$M_1$}; 
\draw (ph) -- (0,-1.5) node[below] {$M_2$};
\draw[midarrow] (bt1) to[out=180,in=90] (-1,-1.5);
\node at (-0.9,-1.1) {$j$};
\draw (bt2) to[out=0,in=90] (1,-1.5);
\end{tikzpicture}
=
m_i \sqrt{d_j^R} \sqrt{d_i^R} \sqrt{d_l^R}
\begin{tikzpicture}
\node[small_morphism] (ph) at (0,0.6) {\small $\ph_l$};
\node[semi_morphism={90,}] (bt1) at (-1,0.1) {\tiny $\beta$};
\node[semi_morphism={270,}] (bt2) at (1,0.1) {\tiny $\beta$};
\draw[midarrow] (ph) to[out=180,in=90] (bt1);
\node at (-0.5,0.8) {$l$};
\draw (ph) to[out=0,in=90] (bt2);
\draw[midarrow_rev={0.75}] (-0.5,-0.4) circle(0.3cm);
\node at (-0.5,-0.9) {$i$};
\draw (ph) -- (0,1.5) node[above] {$M_1$}; 
\draw (ph) -- (0,-1.5) node[below] {$M_2$};
\draw[midarrow] (bt1) -- (-1,-1.5);
\node at (-1.1,-1.1) {$j$};
\draw (bt2) -- (1,-1.5);
\end{tikzpicture}
=
\begin{tikzpicture}
\node[morphism] (ph) at (0,0) {\small $\ph_l$};
\draw (ph) -- (0,1.5) node[above] {$M_1$};
\draw (ph) -- (0,-1.5) node[below] {$M_2$};
\draw[midarrow] (ph) to[out=180,in=90] (-1,-1.5);
\node at (-0.9,-0.1) {$l$};
\draw[midarrow_rev] (ph) to[out=0,in=90] (1,-1.5);
\end{tikzpicture}
\end{gather*}

The first equality is the same as the previous computation.
The second equality uses the fact that $\sum \ph_i$
intertwines half-braidings,
so that we ``pull'' the $k$ strand through $\ph_i$
The third equality comes from ``pulling'' $\al$
through $\beta$.
The fourth equality comes from ``pulling''
the $i$ loop through $\beta$.
Finally, for the last equality,
we observe that (1) only $j=k$ terms in the sum contribute,
and so we have a $d_j^R$ coefficient,
and we may apply \lemref{l:summation};
(2) since
$d_i^R = 1$ for $i\in \Irr_0(\cC)$,
\[
\sum_i m_i \sqrt{d_i^R}
\begin{tikzpicture}
\draw[midarrow_rev={0}] (0,0) circle(0.3cm);
\node at (0.4,-0.2) {$i$};
\end{tikzpicture}
= \sum_{l\in \Irr_0(\cC)} \id_{\one_l}
= \id_\one
\]
\end{proof}

The following is a lemma used in \secref{s:module}:
\begin{lemma} \label{l:Mga_proj}
Let $(M,\ga) \in \cZ_\cC(\cM)$. The morphism
\begin{equation}
\label{e:P_Mga}
P_{(M,\ga)} :=
\frac{1}{|\Irr_0(\cC)| \cD} G(\sum d_i^R \ga_{X_i})
=  
\sum_{i,j,k \in \Irr(\cC)} \frac{\sqrt{d_i^R}\sqrt{d_j^R}d_k^R}{|\Irr_0(\cC)| \cD}
\begin{tikzpicture}
\node[small_morphism] (ga) at (0,0) {$\ga$}; 
\node[semi_morphism={90,}] (L) at (-1,0) {$\al$};
\node[semi_morphism={270,}] (R) at (1,0) {$\al$};
\draw (ga)-- +(0,1.5) node[above] {$M$}; \draw (ga)-- +(0,-1.5) node[below] {$M$}; 
\draw[midarrow_rev] (L)-- +(0,1.5) node[pos=0.5,left] {$i$};
\draw[midarrow] (L) -- +(0,-1.5) node[pos=0.5,left] {$j$}; 
\draw[midarrow] (R) -- +(0,1.5) node[pos=0.5,right] {$i$};
\draw[midarrow_rev] (R)-- +(0,-1.5) node[pos=0.5,right] {$j$};
\draw[midarrow={0.6}] (L) to[out=-80,in=180] (ga);
\node at (-0.5, -0.5) {$k$};
\draw[midarrow_rev={0.6}] (R) to[out=-100,in=0] (ga);
\node at (0.5, -0.5) {$k$};
\end{tikzpicture}
=  
\sum_{i,j \in \Irr(\cC)} \frac{\sqrt{d_i^R}\sqrt{d_j^R}}{|\Irr_0(\cC)| \cD}
\begin{tikzpicture}
\node[small_morphism] (ga) at (0,0.4) {$\ga$}; 
\node[small_morphism] (ga2) at (0,-0.4) {$\ga$};
\draw (ga) -- (0,1.5) node[above] {$M$};
\draw (ga) -- (ga2);
\draw (ga2) -- (0,-1.5) node[below] {$M$}; 
\draw[midarrow_rev] (ga) to[out=180,in=-90] (-1,1.5);
\node at (-0.7,1) {$i$};
\draw[midarrow] (ga) to[out=0,in=-90] (1,1.5);
\draw[midarrow] (ga2) to[out=180,in=90] (-1,-1.5);
\node at (-0.7,-1) {$j$};
\draw[midarrow_rev] (ga2) to[out=0,in=90] (1,-1.5);
\end{tikzpicture}
\end{equation}
is a projection in $\End_{\cZ_\cC(\cM)}(I(M))$.
Furthermore, it can be written as a composition
$P_M = \hat{P}_M \circ \check{P}_M$,
where
\begin{equation}
\check{P}_M :=
\sum_{i \in \Irr(\cC)} \frac{\sqrt{d_i^R}}{\sqrt{|\Irr_0(\cC)| \cD}}
\;
\begin{tikzpicture}
\node[small_morphism] (ga) at (0,0) {$\ga$}; 
\draw (ga) -- (0,1) node[above] {$M$};
\draw (ga) -- (0,-1) node[below] {$M$}; 
\draw[midarrow_rev] (ga) to[out=180,in=-90] (-1,1);
\node at (-0.7,0.6) {$i$};
\draw[midarrow] (ga) to[out=0,in=-90] (1,1);
\end{tikzpicture}
\;\;\;
,
\;\;\;
\hat{P}_M :=
\sum_{j \in \Irr(\cC)} \frac{\sqrt{d_j^R}}{\sqrt{|\Irr_0(\cC)| \cD}}
\;
\begin{tikzpicture}
\node[small_morphism] (ga) at (0,0) {$\ga$}; 
\draw (ga) -- (0,1) node[above] {$M$};
\draw (ga) -- (0,-1) node[below] {$M$}; 
\draw[midarrow] (ga) to[out=180,in=90] (-1,-1);
\node at (-0.7,-0.6) {$j$};
\draw[midarrow_rev] (ga) to[out=0,in=90] (1,-1);
\end{tikzpicture}
\end{equation}
such that $\check{P}_M \circ \hat{P}_M = \id_{(M,\ga)}$,
thus exhibiting $(M,\ga)$ as a direct summand of $I(M)$.
\end{lemma}

\begin{proof}
The second equality in \eqnref{e:P_Mga}
follows from pulling $\al$ through $\ga$
and using \lemref{l:summation}.
$\hat{P}_M$ was shown to be a morphism
in $\Hom_{\cZ_\cC(\cM)}((M,\ga), I(M))$
in the proof of \thmref{t:center},
and one shows $\check{P}_M \in \Hom_{\cZ_\cC(\cM)}(I(M), (M,\ga))$
in a similar fashion.
The following computation shows that
$\check{P}_M \circ \hat{P}_M = \id_{(M,\ga)}$:
\begin{align*}
\check{P}_M \circ \hat{P}_M =
\sum_{i\in \Irr(\cC)} \frac{d_i^R}{|\Irr_0(\cC)| \cD}
\begin{tikzpicture}
\draw (0,-0.8) -- (0,0.8) node[above] {$M$} node[pos=0,below] {$M$};
\draw[midarrow] (0,0) circle(0.4cm);
\node at (-0.6,0) {\small $i$};
\node[small_morphism] (ga2) at (0,0.4) {\tiny $\ga$};
\node[small_morphism] (ga3) at (0,-0.4) {\tiny $\ga$};
\end{tikzpicture}
=
\frac{1}{|\Irr_0(\cC)| \cD}
\begin{tikzpicture}
\draw[regular, midarrow] (0,0) circle(0.4cm);
\draw (0.6,-0.8) -- (0.6,0.8) node[above] {$M$} node[pos=0,below] {$M$};
\end{tikzpicture}
=
\id_{(M,\ga)}
\end{align*}
The second equality comes from ``pulling'' the $j$ loop
out to the left,
and the last equality follows from \lemref{l:dashed_circle}.
\end{proof}

The following is a similar result, used in the proof of
\thmref{thm:modular_module}:

\begin{lemma}
\label{l:IM_htr_proj}
Let $M \in \cM$. The morphism
\begin{equation}
P'_M :=
\sum_{i,j,k \in \Irr(\cC)} \frac{\sqrt{d_i^L}\sqrt{d_j^L}d_k^R}{|\Irr_0(\cC)| \cD}
\;
\begin{tikzpicture}
\node[semi_morphism={90,}] (L) at (0.5,0) {$\al$};
\node[semi_morphism={270,}] (R) at (-0.5,0) {$\al$};
\draw (0,1) -- (0,-1) node[below] {$M$} node[above,pos=0] {$M$};
\draw[midarrow_rev] (L)-- +(0,1) node[pos=0.5,right] {$i$};
\draw[midarrow] (L) -- +(0,-1) node[pos=0.5,right] {$j$}; 
\draw[midarrow] (R) -- +(0,1) node[pos=0.5,left] {$i$};
\draw[midarrow_rev] (R)-- +(0,-1) node[pos=0.5,left] {$j$};
\draw[midarrow={0.6}] (L) to[out=-80,in=180] (1.2,0);
\node at (1, -0.5) {$k$};
\draw[midarrow_rev={0.6}] (R) to[out=-100,in=0] (-1.2,0);
\node at (-1, -0.5) {$k$};
\end{tikzpicture}
=  
\sum_{i,j \in \Irr(\cC)} \frac{\sqrt{d_i^L}\sqrt{d_j^L}}{|\Irr_0(\cC)| \cD}
\;
\begin{tikzpicture}
\draw (0,1) -- (0,-1) node[below] {$M$} node[above,pos=0] {$M$};
\draw[midarrow_rev] (0.4,1) to[out=-90,in=180] (1,0.1);
\draw[midarrow] (-0.4,1) to[out=-90,in=0] (-1,0.1);
\node at (-0.7,0.7) {$i$};
\draw[midarrow] (0.4,-1) to[out=90,in=180] (1,-0.1);
\draw[midarrow_rev] (-0.4,-1) to[out=90,in=0] (-1,-0.1);
\node at (-0.7,-0.7) {$j$};
\end{tikzpicture}
\end{equation}
is a projection in $\End_{\htr(\cM)}(\bigoplus X_i \lact M \ract X_i^*)$.
Furthermore, it can be written as a composition
$P'_M = \hat{P}'_M \circ \check{P}'_M$,
where
\begin{equation}
\check{P}'_M :=
\sum_{i \in \Irr(\cC)} \frac{\sqrt{d_i^L}}{\sqrt{|\Irr_0(\cC)| \cD}}
\;
\begin{tikzpicture}
\draw (0,0.6) -- (0,-0.6) node[below] {$M$} node[above,pos=0] {$M$};
\draw[midarrow_rev] (0.3,0.6) to[out=-90,in=180] (0.8,0);
\draw[midarrow] (-0.3,0.6) to[out=-90,in=0] (-0.8,0);
\node at (-0.6,0.4) {$i$};
\end{tikzpicture}
\;\;\;
,
\;\;\;
\hat{P}'_M :=
\sum_{j \in \Irr(\cC)} \frac{\sqrt{d_j^L}}{\sqrt{|\Irr_0(\cC)| \cD}}
\;
\begin{tikzpicture}
\draw (0,0.6) -- (0,-0.6) node[below] {$M$} node[above,pos=0] {$M$};
\draw[midarrow] (0.3,-0.6) to[out=90,in=180] (0.8,0);
\draw[midarrow_rev] (-0.3,-0.6) to[out=90,in=0] (-0.8,0);
\node at (-0.6,-0.4) {$j$};
\end{tikzpicture}
\end{equation}
such that $\check{P}'_M \circ \hat{P}'_M = \id_M$,
thus as objects
in $\Kar(\htr(\cM))$, we have
$M \simeq (\bigoplus X_i \lact M \ract X_i^*, P_M')$.
\end{lemma}

\begin{proof}
Essentially the same as \lemref{l:Mga_proj}.
(Note the use of left dimensions $d_i^L$ instead of right dimensions
$d_i^R$.)
\end{proof}